\documentclass[a4paper,reqno,11pt]{amsart}

\usepackage[all]{xy}
\usepackage[english]{babel}
\usepackage{amssymb}
\usepackage{graphicx}
\usepackage{hhline}
\usepackage[dvipsnames]{xcolor}
\usepackage{shuffle}
\usepackage{wasysym}
\usepackage{amsmath,amsthm,amssymb,verbatim}

\usepackage[colorlinks]{hyperref}
\hypersetup{
    colorlinks,
    linkcolor={ForestGreen},
    citecolor={MidnightBlue},
    urlcolor={Magenta}
}

\usepackage{geometry}
\newcommand{\calL}{\mathcal{L}}

\newcommand{\bbC}{\mathbb{C}}

\newcommand{\bbN}{\mathbb{N}}

\newcommand{\bbP}{\mathbb{P}}

\newcommand{\bbZ}{\mathbb{Z}}

\newcommand{\ttc}{\mathtt{c}}

\newcommand{\ttB}{\mathtt{B}}
\newcommand{\ttS}{\mathtt{S}}
\newcommand{\ttP}{\mathtt{P}}
\newcommand{\ttM}{\mathtt{M}}
\newcommand{\ttV}{\mathtt{V}}

\newcommand{\p}{\mathbb{P}}
\newcommand{\A}{\mathbb{A}}

\newcommand{\N}{\mathbb{N}}

\newcommand{\C}{\mathbb{C}}

\newcommand{\AAA}{\mathbb{A}}

\newcommand{\Bs}{\mathop{\rm Bs}\nolimits}
\renewcommand{\vert}{\mathop{\rm vert}\nolimits}
\newcommand{\SV}{\mathop{\rm SV}\nolimits}
\newcommand{\V}{\mathop{\rm V}\nolimits}
\newcommand{\Bl}{\mathop{\rm Bl}\nolimits}

\newcommand{\codim}{\mathop{\rm codim}\nolimits}
\newcommand{\Sym}{\mathop{{\rm Sym}}\nolimits}
\newcommand{\Span}[1]{\langle#1\rangle}
\newcommand{\edim}{\mathop{\rm edim}\nolimits}
\newcommand{\vdim}{\mathop{\rm vdim}\nolimits}

\newcommand{\Tr}{\mathop{\rm Tr}\nolimits}

\newcommand{\Res}{\mathop{\rm Res}\nolimits}

\theoremstyle{plain}
\newtheorem{theorem}{Theorem}[section]
\newtheorem{proposition}[theorem]{Proposition}
\newtheorem{lemma}[theorem]{Lemma}   
\newtheorem{corollary}[theorem]{Corollary}
\newtheorem{claim}{Claim}

\theoremstyle{definition}

\newtheorem{notation}[theorem]{Notation}

\newtheorem{example}[theorem]{Example}

\newtheorem{definition}[theorem]{Definition}

\newtheorem{remark}[theorem]{Remark}

\numberwithin{equation}{section}

\newcommand{\im}[0]{\operatorname{im}}

\newcommand{\LL}{\mathcal{L}}

\newcommand{\rd}[1]{\left\lfloor #1\right\rfloor}
\newcommand{\ru}[1]{\left\lceil #1\right\rceil}

\newcommand{\Ale}[1]{{\color{black}#1}}
\newcommand{\Fra}[1]{{\color{black} #1}}

\title{Secant non-defectivity via collisions of fat points}

\author{Francesco Galuppi and Alessandro Oneto}
\address[Francesco Galuppi]{Institute of Mathematics of the Polish Academy of Sciences, ulica Śniadeckich 8,
00-656 Warszawa, Poland.}
\email{fgaluppi@impan.pl}

\address[Alessandro Oneto]{Department of Mathematics, 
Universit\`a di Trento, Via Sommarive, 14, 38123 Povo (Trento), Italy}
\email{alessandro.oneto@unitn.it}


\begin{document}
	\maketitle

\begin{abstract}
The computation of dimensions of secant varieties of projective varieties is classically approached via dimensions of linear systems with multiple base points in general position. Non-defectivity can be proved via degenerations. In this paper, we use a technique that allows some of the base points to collapse together in order to deduce a new general criterion for non-defectivity. We apply this criterion to prove a conjecture by Abo and Brambilla: for $c\ge 3$ and $d\ge 3$, the Segre-Veronese embedding of $\p^m\times\p^n$ in bidegree $(c,d)$ is non-defective.
\end{abstract}

\section{Introduction}

A classical problem in algebraic geometry that goes back to late XIX century concerns the classification of \textit{defective varieties}, i.e., algebraic varieties whose \textit{secant varieties} have dimension strictly smaller than the one expected by a direct parameter count. In the last decades, this problem gained a lot of attention due to its relation with \textit{additive decompositions} of \textit{tensors} which are used in many areas of applied mathematics and engineering. Indeed, \textit{Segre varieties} parametrize \textit{decomposable tensors}; similarly, \textit{Veronese varieties} and \textit{Segre-Veronese varieties} are the symmetric and partially-symmetric analogous. We refer to \cite{carlinifourlectures} and \cite{guida} for an overview on the geometric problem and to \cite{Lan} for relations between secant varieties and questions on tensors.

The most celebrated result in this area of research is the celebrated \textit{Alexander-Hirschowitz Theorem}, proven in \cite{AH}, which classifies defective Veronese varieties by completing the work started more than 100 years earlier by the classical school of algebraic geometry. Denote by $\V_n^d$ the Veronese variety given by the embedding of $\bbP^n$ via the linear system of degree $d$ divisors. Several examples of defective Veronese varieties were known already at the time of Clebsch, Palatini and Terracini, but we had to wait until a series of enlightening papers by Alexander and Hirschowitz which culminated in 1995 with the proof that those are the only exceptional cases among Veronese varieties. We refer to \cite[Section~7]{BraOtt08} for an historical overview on this theorem. 
\begin{theorem}[Alexander-Hirschowitz]\label{thm:AH}\label{thm: AH}
	Let $n$, $d$ and $r$ be positive integers. The Veronese variety $\V^d_n$ is $r$-defective if and only if either
	\begin{enumerate}
		\item $d = 2$ and $2 \leq r \leq n$, or
		\item $(n,d,r) \in \{(2,4,5), (3,4,9), (4,3,7), (4,4,14)\}$.
	\end{enumerate}
\end{theorem}
After this result, the community tried to extend the classification of defective varieties to Segre and Segre-Veronese varieties by applying and refining the powerful methods introduced by Alexander and Hirschowitz. In the present paper, we focus on Segre-Veronese varieties with two factors, i.e., the image $\SV_{m\times n}^{c,d}$ of the embedding of $\bbP^m\times\bbP^n$ via the linear system of divisors of bidegree $(c,d)$. 
Among them, a long list of defective cases have been found (see \cite[Table 1]{AboBra13}) by Catalisano, Geramita and Gimigliano \cite{CatGerGim05,CatGerGim08}, 
Abrescia \cite{Abrescia}, Bocci \cite{Bocci}, Dionisi and Fontanari \cite{DioFon01}, Abo and Brambilla \cite{AboBra09}, Carlini and Chipalkatti \cite{CarChi03} and Ottaviani \cite{Ott07}. In \cite[Conjecture~5.5]{AboBra13}, Abo and Brambilla conjectured that these are the only defective cases. The fact that in all defective examples either $c$ or $d$ is strictly smaller than three suggested a weaker conjecture, stated in \cite[Conjecture~5.6]{AboBra13}. In this paper we prove this conjecture by proving the following result which is a step towards a complete classification of \Fra{defective} Segre-Veronese varieties with two factors.
\begin{theorem}\label{thm: main}
Let $m$ and $n$ be positive integers. If $c\ge 3$ and $d\ge 3$, then $\SV_{m\times n}^{c,d}$ is not defective.
\end{theorem}

 Abo and Brambilla themselves managed to greatly reduce the problem. Thanks to \cite[Theorem~1.3]{AboBra13}, in order to prove Theorem \ref{thm: main} it is enough to prove the base cases $(c,d) \in \{(3,3),(3,4),(4,4)\}$. This reminds what happened with Alexander-Hirschowitz Theorem, where the last to be overcome was the case of cubics which constituted the base case of the inductive proof. Low degrees are difficult to handle because they are rich of defective cases. 

Dimensions of secant varieties of polarized varieties are classically computed by translating the problem to the computation of dimensions of certain linear systems of divisors with multiple base points in general position; see Section \ref{sec: linear systems}. In this context, defectivity means that the linear system has not the dimension expected by a direct parameter count. Degeneration techniques are very powerful tools to tackle these problems. Indeed, the expected dimension is always a lower bound for the actual one, while a specialization of the base points can only increase it. Hence, the whole approach boils down to finding a good specialization for which we can prove that the dimension is equal to the expected one. The seminal idea of Castelnuovo was to specialize some of the supports on a hypersurface. Then, via the so-called \textit{Castelnuovo exact sequence}, one can proceed with a double induction on dimensions and degrees; see Section \ref{sec:inductive}. However, this approach has arithmetic constrains: when the virtual dimension, i.e., the parameter count, is close to zero then there might be not enough freedom to find a good specialization to make the double induction work. At the same time, the cases with small virtual dimension are particular compelling because most defective cases appear among them.
In the 1980s, Alexander and Hirschowitz improved drastically this method by introducing a new degeneration technique, called \textit{differential Horace method}, which allowed them to complete the classification of defective Veronese varieties overcoming the arithmetic issues. The drawback of the differential Horace method is that, when the virtual dimension is close to zero, it might lead to consider linear systems whose base locus has a complicated non-reduced structure. In order to overcome the latter complication, in the literature, there are several examples of results on linear systems having virtual dimension sufficiently distant from zero; see e.g. \cite[Theorem 2.3]{BCC} or \cite[Theorem~1.3]{Abo10}.

In this paper, we employ a different degeneration technique, introduced by Evain in \cite{EvainIdea}: the base points are not only degenerated to a special position, but also allowed to collide together; see Section~\ref{sec:collapsing_points} for details. The degenerated linear system has 
a $0$-dimensional base point with a very special, yet understood, non-reduced structure which is contained in a $4$-fat point and contains a $3$-fat point. Apparently a disadvantage, this new scheme turned out to be very useful to prove the following criterion for non-defectivity. Given a linear system $\calL$ on a variety, we denote by $\calL(m,2^r)$ the linear subsystem of $\calL$ obtained by imposing a base point of multiplicity $m$ and $r$ general base double points.
\begin{theorem}\label{thm: quello che usiamo}
	Let $(V,\calL)$ be a polarized smooth irreducible projective variety of dimension $n$. Suppose that $\calL$ embeds $V$ as a proper closed subvariety of $\bbP\calL^\vee$ and $W$ is the image of such embedding. Assume that
	\begin{enumerate}
		\item $\LL(3,2^{r-n-1})$ is regular, i.e., $\dim\LL(3,2^{r-n-1}) = \dim\calL - \binom{n+2}{2} - (r-n-1)(n+1)$,
        \item $\LL(4,2^{r-n-1})$ is zero,
        \item $\dim\LL(3)-\dim\LL(4)\ge\binom{n+1}{2}$ and
		\item $\dim\LL\ge (n+1)^2$
	\end{enumerate}
	for $r\in\left\{\rd{\frac{\dim\LL}{n+1}},\ru{\frac{\dim\LL}{n+1}}\right\}$. Then $W$ is not defective.
\end{theorem}
The key of success of the latter criterion should be made clear, in light of the historical remarks above, when we consider one of the problematic cases where $\calL(2^r)$ has virtual dimension close to zero. By letting collide $n+1$ of the general $2$-fat base points, we get a base locus with $r-(n+1)$ $2$-fat points and one component which is between a $3$-fat and a $4$-fat point. Then, regularity follows from proving the conditions (1) and (2), where the latter component of the base locus is replaced once with the $3$-fat point and once with the $4$-fat point. Even if in the latter two cases we no longer have only double points, the virtual dimension becomes sufficiently distant from zero to let us find further classic degenerations to complete the proof.

Our proof of Theorem \ref{thm: main} is indeed an application of Theorem \ref{thm: quello che usiamo}. Our result adds up to previous successful applications of this degeneration technique: in \cite{EvainSHGH} for linear systems of plane curves, in \cite{Galuppi} for linear systems in $\bbP^3$ and in \cite{GM} in the context of Waring decompositions of polynomials. For this reason, we believe that our general result can be used to attack also other questions on non-defectivity of projective varieties for which previous methods presented technical obstacles.

\Fra{Our Theorems \ref{thm: main} and \ref{thm: quello che usiamo} have consequences also on \textit{identifiability}. Given a non-degenerate projective variety $W\subset\p^N$ and a point $p\in\p^N$ in the ambient space, the \textit{rank} of $p$ with respect to $W$ is \Ale{the smallest cardinality of a set of distinct points on $W$ whose linear span contains $p$. The $k$-th secant variety of $W$ is the Zariski-closure of the set of points of rank at most $k$}. 
The point $p$ is \textit{$k$-identifiable} with respect to $W$ \Ale{if such spanning set of minimal cardinality is unique}. The variety $W$ is called \textit{$k$-identifiable} if the general rank-$k$ point of the ambient space is $k$-identifiable. In \cite{casarotti2019non}, the authors relate the study of defectivity to the study of identifiability. Combining their main result with our Theorems \ref{thm: main} and \ref{thm: quello che usiamo}, we obtain the following  identifiability statement for all ranks smaller than the general rank.

\begin{corollary}
\label{cor: identifiability in general} Under the same hypothesis as in Theorem \ref{thm: quello che usiamo}, the variety $W$ is $(k-1)$-identifiable for every $2\dim(W)<k\le\rd{\frac{\dim\LL}{\dim(W)+1}}$.
\end{corollary}

\Ale{In the particular case of Segre-Veronese varieties of two factors, we conclude the following.}

\begin{corollary}\label{cor: identifiability for SV}
Let $m$, $n$, $c$ and $d$ be positive integers such that $c\ge 3$ and $d\ge 3$. If \begin{equation}\label{eq:upper_bound_ident}
k\le\left\lfloor \frac{1}{m+n+1}\binom{m+c}{c}\binom{n+d}{d}\right\rfloor,
\end{equation}
then $\SV_{m\times n}^{c,d}$ is $(k-1)$-identifiable. \end{corollary}

\Ale{As mentioned, when the algebraic variety $W$ with respect to which we consider the notion of rank is a variety like Veronese, Segre, or Segre-Veronese varieties, then the geometric notion of rank corresponds to the one of symmetric rank, tensor rank and partially symmetric rank of tensors. If we consider partially symmetric tensors $T \in {\rm Sym}^c\bbC^{m+1}\otimes{\rm Sym}^d\bbC^{n+1}$, our results are rephrased as follows.
\begin{itemize}
    \item (Theorem \ref{thm: main}) If $T$ is general, then it has rank $k = \left\lceil {m+c \choose m}{n+d \choose n}/(m+n+1)\right\rceil$, i.e., there exists $\{v_1,\ldots,v_k\} \subset \bbC^{m+1}$ and $\{w_1,\ldots,w_k\}\subset\bbC^{n+1}$ such that
    \begin{equation}\label{eq:tensorDec}
    T = \sum_{i=1}^k v_i^{\otimes c}\otimes w_i^{\otimes d}.
    \end{equation}
    \item (Corollary \ref{cor: identifiability for SV}) If $T$ is general of rank $k < \left\lfloor {m+c \choose m}{n+d \choose n}/(m+n+1)\right\rfloor$ then the additive decomposition \eqref{eq:tensorDec} is uniquely determined up to permutations of the summands and where the $v_i$'s (respectively, $w_i$'s) are determined up to $c$-th (respectively $d$-th) roots of unity.
\end{itemize}
}
}

\subsection*{Structure of the paper}
In Section \ref{sec:basics} we recall the basic definitions for secant varieties and linear systems with multiple base points. We also illustrate the tools we use in our computation, such as Castelnuovo exact sequence and collisions of fat points. In Section \ref{sec: parte teorica interessante} we prove Theorem \ref{thm: quello che usiamo} via collisions of fat points. The rest of the paper is an application of Theorem \ref{thm: quello che usiamo} to Segre-Veronese varieties with two factors $\SV_{m\times n}^{c,d}$ in order to prove Theorem \ref{thm: main}. The three key cases $(c,d)\in\{(3,3),(3,4),(4,4)\}$ are solved in Sections \ref{sec: 33}, \ref{sec: bidegree (3,4)} and \ref{sec: bidegree (4,4)} respectively. This completes the proof of Theorem \ref{thm: main}. In Appendix \ref{appendix: M2} we describe the software computations we performed to check the initial cases of our inductive proofs. In Appendix \ref{appendix: contacci} we collect the long and tedious arithmetic computations needed in our proofs.

\subsection*{Acknowledgements.} 
The project started while FG was a postdoc at MPI MiS Leipzig (Germany) and AO was a postdoc at OVGU Magdeburg (Germany). AO thanks MPI MiS Leipzig for providing a perfect environment  during several visits at the beginning of the project. The authors also thank Maria Virginia Catalisano, Alessandro Gimigliano and Massimiliano Mella for useful discussions. \Fra{The authors also wish to thank the anonymous referee for pointing out a mistake in the first draft of the article.} FG is supported by the National Science Center, Poland, project "Complex contact manifolds and geometry of secants", 2017/26/E/ST1/00231, and acknowledges partial support from the fund FRA 2018 of University of Trieste - project DMG
. AO acknowledges partial financial support from A. v. Humboldt Foundation/Stiftung through a fellowship for postdoctoral researchers. \Ale{AO is a member of  GNSAGA of INdAM (Italy)}. 

\section{Basics and background}\label{sec:basics}
We recall useful definitions and constructions in the context of secant varieties and linear systems of divisors with multiple base points. We will work over the field of complex numbers $\C$. 
\subsection{Segre-Veronese varieties and their secants.}
\begin{definition}
Fixed $m,n,c,d\in\N$, the \textbf{Segre-Veronese variety} $\SV_{m\times n}^{c,d}$ is the image of the embedding of $\bbP^m\times\bbP^n$ via the linear system of divisors of bidegree $(c,d)$.
\end{definition}
The Segre-Veronese variety has a precise interpretation in terms of partially symmetric tensors. Let $\Sym^d_{n+1}$ be the vector space of degree $d$ homogeneous polynomials in $n+1$ variables with complex coefficients. The variety $\SV_{m\times n}^{c,d}$ is parametrized by partially symmetric tensors in $\Sym^c_{m+1}\otimes \Sym^d_{n+1}$ which are \textit{decomposable}, or of \textit{rank $1$}, i.e., 
\[
	\SV_{m\times n}^{c,d} = \{f\otimes g ~:~ f \in \Sym^c_{m+1}, g \in \Sym^d_{n+1}\} \subset \bbP \left(\Sym^c_{m+1}\otimes \Sym^d_{n+1}\right).
\]
\begin{definition}\label{def: edim per secanti}
	Let $V \subset \bbP^N$ 
be a projective variety. The \textbf{$r$-th secant variety} of $V$ is the Zariski-closure of the union of all linear spaces spanned by $r$ points on $V$, i.e.,
	\[
		\sigma_r(V) := \overline{\bigcup_{p_1,\ldots,p_r \in V} \langle p_1,\ldots,p_r \rangle} \subset \bbP^N.
	\]
In the case of Segre-Veronese varieties, the $r$-th secant variety consists of the Zariski-closure of the set of \textit{rank-$r$} partially symmetric tensors, i.e.,
\[
\sigma_r(\SV_{m\times n}^{c,d}) = \overline{\left\{ \sum_{i=1}^r 
	f_i \otimes g_i ~:~ f_i \in \Sym^c_{m+1}, g_i \in \Sym^d_{n+1}\right\}} \subset \bbP\left(\Sym^c_{m+1}\otimes \Sym^d_{n+1}\right).
\]
\end{definition}
Our goal is to compute the dimension of $\sigma_r(V)$. By a count of parameters, its \textbf{expected dimension} is
\begin{align*}
\edim\sigma_r\left(V\right) & = \min\left\{N, r(\dim V + 1) - 1 \right\}.
\end{align*}
The actual dimension is always smaller than or equal to the expected one. If it is strictly smaller, then we say that $V$ is \textbf{$r$-defective}. A variety is \textbf{defective} if it is $r$-defective for some $r$.
\subsection{Linear systems with non-reduced base locus}\label{sec: linear systems}
In this section we recall how to translate the problem to a question about dimensions of linear systems with multiple base points. For this purpose, we fix some notation we will use through the paper. Let $V$ be a smooth projective variety.

\begin{definition}
   Let $p$ be a point on $V$ defined by an ideal $I_p$ in the coordinate ring of $V$. If $a\in\N$, then the \textbf{$a$-fat point} supported at $p$ is the $0$-dimensional scheme defined by $I_p^a$. If $p_1,\dots,p_r\in V$, then the \textbf{scheme of fat points of type $(a_1,\ldots,a_r)$}, denoted by $a_1p_1+\ldots+a_rp_r$, is the union of fat points defined by the ideal  $I_{p_1}^{a_1}\cap\ldots\cap I_{p_r}^{a_r}$. We call it \textbf{general} when $p_1,\dots,p_r$ are general points of $V$. 
\end{definition}

Since the arguments used in this paper do not require cohomology of sheaves, with a slight abuse of terminology we always regard linear systems simply as $\bbC$-vector spaces.

\begin{definition}\label{def: edim}
Let $p_1,\dots,p_r\in V$ and let $a_1,\dots,a_r\in\N$. Let $I_X$ be the ideal in the coordinate ring of $V$ defining the scheme of fat points $X=a_1p_1+\ldots+a_rp_r$. If $\LL$ is a complete linear system on $V$, then we denote by $\calL(X)$ the vector space $I_X \cap \LL$ of divisors in $\LL$ passing through each $p_i$ with mulitplicity \Fra{at least $a_i$}. The \textbf{virtual dimension} of $\calL(X)$ is given by the count of parameters
\[
\vdim \calL(X) := \dim\LL - \sum_{i=1}^r\binom{a_i+\dim V-1}{\dim V}.\]
If $V=\p^n$, then we write $\LL_n^d$ for the complete linear system of degree $d$ hypersurfaces in $\p^n$. In this case $\LL_n^d(X)$ is identified with the vector space of degree $d$ homogeneous polynomials vanishing with multiplicity at least $a_i$ at $p_i$. 
Its virtual dimension is
\[
\vdim \calL_n^d(X) = \binom{n+d}{n} - \sum_{i=1}^r\binom{a_i+n-1}{n}.\]
In a similar way, if we consider  $V=\bbP^m\times\bbP^n$ then we denote by $\calL_{m\times n}^{c,d}(X)$ the vector space of bihomogeneous polynomials of bidegree $(c,d)$ vanishing with multiplicity at least $a_i$ at $p_i$. Its virtual dimension is
\[
\vdim \calL_{m\times n}^{c,d}(X) = \binom{m+c}{m}\binom{n+d}{n} - \sum_{i=1}^r\binom{a_i+m+n-1}{m+n}.\]
In all cases, the \textbf{expected dimension} is the maximum between $0$ and the virtual dimension. Hence the actual dimension is larger than or equal to the expected one. If the inequality is strict, then we say that the linear system is {\bf special}. If the virtual dimension is non-negative and the linear system is not special, then we say that it is {\bf regular}. 
\end{definition}

It is important to recall what happens to the dimension of linear systems under specialization of the base locus. Consider a scheme of fat points $X = a_1p_1 + \ldots + a_rp_r\subset V$. By semicontinuity, there exists a Zariski-open subset of $V^{\times r}$ where the dimension of $\calL(X)$ is constant and takes the minimal value among all possible choices of $(p_1,\ldots,p_r) \in V^{\times r}$. We denote by $\calL(a_1,\ldots,a_r)$ the linear system associated to a general choice of the support. In particular, 
\[
\dim \calL(X) \geq \dim\calL(a_1,\ldots,a_r).
\]
In case of repetitions in the vector $(a_1,\ldots,a_r)$, we use the notation $a^s$ for the $s$-tuple $(a,\ldots,a)$. 

If $(V,\calL)$ is a polarized smooth irreducible projective variety such that $\calL$ embeds $V$ as a closed subvariety $W \subset \bbP\calL^\vee$, then the dimension of secant varieties of $W$ is related to the dimension of certain linear systems on $V$. Indeed, 
\[
    \codim\sigma_r(W)=\dim\calL(2^r).
\]
This is a consequence of Terracini's Lemma. See \cite[Corollary 1]{guida} for a \Ale{recent} reference. In particular,  $\SV_{m\times n}^{c,d}$ is $r$-defective if and only if $\calL_{m\times n}^{c,d}(2^r)$ is special. Note that in order to prove that  $W\subset\p\LL^\vee$ is not $r$-defective for every $r\in\N$, it is enough to check few values of $r$ thanks to the following straightforward observation.
\begin{remark}
    \label{rmk: bastano r alto e r basso}
Let $\LL$ be a linear system on the variety $V$ and let $X^\prime\subset X$ be two schemes of fat points on $V$. Then:
\begin{itemize}
\item if $\calL(X)$ is regular, then $\calL(X')$ is regular.
	\item if $\calL(X^\prime)=0$, then $\calL(X)=0$.
\end{itemize}
Hence, in order to prove that $\calL(2^r)$ is non-special for every $r$, it is enough to consider 
\[
	r_*:=\max\left\{s\in\N ~:~ \vdim \calL(2^s)\ge 0\right\}\mbox{ and }
r^*:=\min\left\{s\in\N ~:~ \vdim \calL(2^s)< 0\right\}
\]
and prove that $\calL(2^{r_*})$ is regular and $\calL(2^{r^*})$ is zero.
\end{remark}

\subsection{Inductive methods}\label{sec:inductive}
Our proof of Theorem \ref{thm: main} relies on a classical inductive approach. Let $V$ be a smooth projective variety and let $H$ be a subvariety of $V$. Consider a linear system $\calL$ on $V$ and let $\calL_{H}$ be the linear system on $H$ given by
\[
	\calL_{H} := \{D \cap H ~:~ D \in \calL\}.
\]
Let $X \subset V$ be a scheme of fat points. Then there is an exact sequence of vector spaces
\begin{equation}\label{exact_sequence}
	0 \rightarrow \calL(X + H) \rightarrow \calL(X) \rightarrow \calL_H(X \cap H),
\end{equation}
sometimes called \textit{Castelnuovo exact sequence}, where $\calL(X + H)$ denotes the subsystem of $\calL$ of divisors containing $X \cup H$.
\begin{definition}
	\label{def: trace and residue}
Let $H$ be a divisor of $V$. Let $I_X$ be the ideal defining a scheme of fat points $X\subset V$.
	\begin{itemize}
\item The \textbf{residue} of $X$ with respect to $H$ is the subscheme $\Res_H(X)\subset V$ defined by the saturation $(I_X : I_H)^{\rm sat}$.
\item The \textbf{trace} of $X$ on $H$ is the scheme-theoretic intersection $\Tr_H(X) := H \cap X$, defined by $I_X+I_H$.
	\end{itemize}
\end{definition}
In this paper we are interested in the case $V=\p^m\times\p^n$ and $\LL=\LL_{m\times n}^{c,d}$. Let $H\cong\p^{m-1}\times\p^n$ be a divisor of bidegree $(1,0)$. Then \eqref{exact_sequence} becomes
\begin{equation}\label{exact_sequence_bihom}
	0 \rightarrow \calL_{m\times n}^{c-1,d}(\Res_H(X)) \rightarrow  \calL_{m\times n}^{c,d}(X) \rightarrow \calL_{(m-1)\times n}^{c,d}(\Tr_H(X)),
\end{equation}
	where the left-most arrow corresponds to the multiplication by $H$. Hence
	\begin{equation}\label{eq:chain_dim}
	    \vdim\calL_{m\times n}^{c,d}(X) \leq \dim\calL_{m\times n}^{c,d}(X) \leq \dim\calL_{m\times n}^{c-1,d}(\Res_H(X)) + \dim\calL_{(m-1)\times n}^{c,d}(\Tr_H(X)).
	\end{equation}
	It is possible to consider analogous constructions also for divisors of bidegree $(0,1)$. If $X=a_1p_1+\ldots+a_rp_r+b_1q_1+\ldots+b_sq_s$, where $p_1,\dots,p_r$ are general points in $\bbP^m \times \bbP^n$ and $q_1,\dots,q_s$ are general on $H$, then
\begin{align*}
&\Res_HX=a_1p_1+\ldots+a_rp_r+(b_1-1)q_1+\ldots+(b_s-1)q_s\subset\p^m\times\p^n\mbox{ and}\\
&\Tr_HX=b_1q_1+\ldots+b_sq_s\subset H \cong \bbP^{m-1}\times\bbP^n.
\end{align*}
A straightforward computation gives
\begin{equation}\label{eq:straightforward_vdim}
    \vdim\calL_{m\times n}^{c-1,d}(\Res_H(X))+\vdim\calL_{(m-1)\times n}^{c,d}(\Tr_H(X))=\vdim\calL_{m\times n}^{c,d}(X).
\end{equation}
By employing \eqref{eq:chain_dim}, we use the Castelnuovo exact sequence in two ways.
\begin{enumerate}
\item If we want to prove that $\calL_{m\times n}^{c,d}(X)=0$, then it is enough to prove that \[
    \calL_{m\times n}^{c-1,d}(\Res_H(X))=\calL_{(m-1)\times n}^{c,d}(\Tr_H(X))=0.\]
\item If we want to prove that $\calL_{m\times n}^{c,d}(X)$ is regular, by \eqref{eq:straightforward_vdim} it is enough to prove that both $\calL_{m\times n}^{c-1,d}(\Res_H(X))$ and $\calL_{(m-1)\times n}^{c,d}(\Tr_H(X))$ are regular. 
\end{enumerate}
Other classical tools are degeneration arguments. If $\tilde{X}$ is a specialization of the scheme $X$, then $\dim\calL_{m\times n}^{c,d}(\tilde{X})\ge\dim\calL_{m\times n}^{c,d}(X)$. Sometimes it is convenient to use exact sequence \eqref{exact_sequence} not only when $H$ is an hyperplane, but also with $H$ of higher codimension. In \cite[Section 5]{BraOtt08}, Brambilla and Ottaviani used this approach to obtain a different proof of Alexander-Hirschowitz Theorem. Here is how we will use these observations: if $H \cong \bbP^{m'}\times\bbP^{n'}$ is a subvariety of $\bbP^m\times\bbP^n$
, then
\begin{align}
	\vdim\calL_{m\times n}^{c,d}(\tilde{X}) = \vdim\calL_{m\times n}^{c,d}(X) & \leq \dim \calL_{m\times n}^{c,d}(X) \leq \dim \calL_{m\times n}^{c,d}(\tilde{X}) \nonumber \\ & \leq \dim \calL_{m\times n}^{c,d}(\tilde{X}+H) + \dim\calL_{m'\times n'}^{c,d}(\tilde{X}\cap H) \label{chain_inequalities}
\end{align}
where the latter inequality follows from \eqref{exact_sequence}. Therefore, in order to prove that $\calL_{m\times n}^{c,d}(X)$ is non-special, the task is to find a suitable specialization $\tilde{X}$ for which we are able to compute the two summands on the right-hand-side of \eqref{chain_inequalities} and such that upper and lower bound coincide. 

Note that \eqref{eq:chain_dim} allows a double induction: in one of the summands we have a lower degree while in the other we have a lower dimension. The classical method of specializing the support of $X$ does not always work due to arithmetic constrains that do not allow to match the upper and the lower bound in \eqref{chain_inequalities}. This was the case for systems of cubics in projective space with general $2$-fat base points. After a series of papers, Alexander and Hirschowitz refined the classical method and managed to complete the classification of special linear systems $\calL_n^d(2^r)$ and, as byproduct, to completely classify defective Veronese varieties. This method is called \textit{differential Horace method} and, in the last decades, it has been used to prove the non-speciality of several linear systems in projective and multiprojective space. Despite its success, this method requires a deep understanding of the geometry of the problem and a choice of specialization which sometimes is too difficult to handle. For this reason, in this paper we consider a different specialization in which the components of the base locus are allowed to collide together. We explain it in the next section. 
\subsection{Collapsing points}\label{sec:collapsing_points}
In this section we recall the specialization method we use to prove Theorem~\ref{thm: quello che usiamo}. We refer to \cite[Remark 20 and Proposition 21]{GM} or \cite[Construction 10]{Galuppi} for details.
 \begin{remark}\label{rmk: collisione di punti doppi}
Let $V$ be a smooth variety of dimension $n$. We consider a general scheme of fat points of type $(2^{n+1})$ on $V$ and we let it \textit{collapse} to one component, i.e., we let all the points of its support approach the same point $q\in V$ from general directions. The result of such a limit is a scheme supported at $q$, containing the 3-fat point $3q$ with the following property: its restriction to a general line $L$ containing $q$ is a 3-fat point of $L$, but there are $\binom{n+1}{2}$ lines through $q$ such that the restriction is a $4$-fat point on each of these special lines. We call it a \textit{triple point with $\binom{n+1}{2}$ tangent directions}. If we call $E$ the exceptional divisor of $\Bl_qV$, then these tangent directions correspond to simple points $\{t_{ij}\mid 0\le i<j\le n\}$ of $E$.
\end{remark}

The following instructive example is discussed also in \cite[Example 22]{GM} and \cite[Example~15]{Galuppi}.
\begin{example}
\label{example: spiegazione della figura}Since the limit is a local construction, we work out an example on $\A^2$. Let $\Delta$ be a complex disk around the origin. Let $Y=\A^2\times\Delta$ and $Y_t=\A^2\times\{t\}$ for $t \in\Delta$. Fix a point $q\in Y_0$ and three general maps $\sigma_1,\sigma_2,\sigma_3:\Delta\to Y$ such that $\sigma_1(t),\sigma_2(t),\sigma_3(t)$ are general points of $Y_t$ for $t \neq 0$ and $\sigma_1(0)=\sigma_2(0)=\sigma_3(0)=q$. See Figure \ref{fig:deformation}. For every $t\neq 0$, let $$X_t=2\sigma_1(t)+2\sigma_2(t)+2\sigma_3(t)\subset Y_t$$ be a general scheme of fat points of type $(2^3)$. 
We are interested in the limit $X_0 := \lim_{t \rightarrow 0}X_t$. For every $t\neq 0$ the ideal $I_{X_t}$ contains a plane cubic $C_t$, consisting of the union of three lines. Hence the limit $C_0$ belongs to $I_{X_0}(3)$. Actually, one can show that $I_{X_0}(3)=\Span{C_0}$. By \cite[Proposition~13]{Galuppi}, $X_0$ strictly contains a 3-fat point but does not contain a 4-fat point. 
In order to completely understand 
its structure, we look at the blow-up $\mu:\tilde{Y}\to Y$ of $Y$ at the point $q$
with exceptional divisor $W$. Let $\tilde{\sigma}_i:\Delta\to\tilde{Y}$ be the map corresponding to $\sigma_i$. We want to stress that, since the sections $\sigma_i$ are general, $\tilde{\sigma}_{1}(0),\tilde{\sigma}_{2}(0),\tilde{\sigma}_{3}(0)$ are three general points of $W$. If we set $\tilde{Y}_t=\mu^{-1}(Y_t)$, then $\tilde{Y}_t\cong Y_t\cong\A^2$ for every $t\neq 0$, but the special fiber $\tilde{Y}_0$ has two irreducible components. We write $\tilde{Y}_0\cong W\cup \Bl_q Y_0$. If we call $E=W\cap \Bl_qY_0$, then $E\cong\p^1$ is the exceptional divisor of $\Bl_qY_0$. Let $\tilde{I}_{X_0}$ be the ideal consisting of all the strict transforms on $\tilde{Y}_0$ of elements of $I_{X_0}$. Since $X_0\supsetneq 3q$, then $\tilde{I}_{X_0}(3)_{|E}\subsetneq H^0\mathcal{O}_E(3)$. 
Indeed, $\tilde{I}_{X_0}(3)_{|W}$ is the system of cubics of $W$ containing the general scheme of fat points $2\tilde{\sigma}_{1}(0)+2\tilde{\sigma}_{2}(0)+2\tilde{\sigma}_{3}(0)$. There is exactly one such cubic
, consisting of the union of the three lines $\Span{\tilde{\sigma}_{i}(0),\tilde{\sigma}_{j}(0)}$, which cut three simple points $t_{ij}=E\cap\Span{\tilde{\sigma}_{i}(0),\tilde{\sigma}_{j}(0)}$ on $E$.  
We regard $X_0$ as the fat point $3q$ together with three infinitely near simple points corresponding to $t_{12},t_{13},t_{23}$.
\end{example}

\begin{figure}[ht]
    \centering
    \includegraphics[scale=0.45]{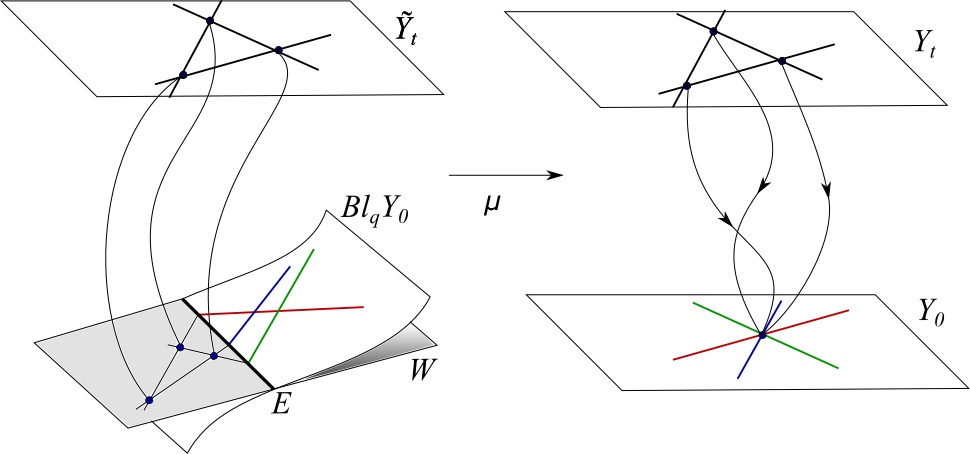}
    \caption{The collision of three $2$-fat points as described in Example \ref{example: spiegazione della figura}.}
    \label{fig:deformation}
\end{figure}

\begin{example}\label{example: punti infinitamente vicini in una star configuration}
Consider the collision of four 2-fat points in $\A^3$. We proceed in the same way as in Example \ref{example: spiegazione della figura} and we see that $\tilde{I}_{X_0}(3)_{|W}$ is the system of cubics of $W\cong\p^3$ containing a general scheme of fat points of type $(2^4)$ supported, say, at $p_0,p_1,p_2,p_3$. Its base locus consists of the $\binom{4}{2}=6$ lines joining each pair of points. These six lines cut six simple points $t_{ij}=E\cap\Span{p_i,p_j}$ on $E\cong\p^2$, but they are not in general position. For instance, all the three points $t_{12}$, $t_{13}$ and $t_{23}$ belong to the line $E\cap\Span{p_1,p_2,p_3}$. The six points are in the configuration described in Figure \ref{fig:starconfiguration}.

\begin{figure}[h!]
    \centering
    \includegraphics[scale=0.3]{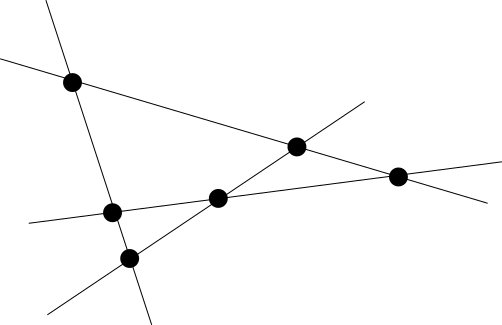}
    \caption{Six points in a \textit{star configurations}, i.e., as intersections of a four general lines.}
    \label{fig:starconfiguration}
\end{figure}
\end{example}

It is crucial to notice that the tangent directions described in Remark \ref{rmk: collisione di punti doppi} are not in general position: \Ale{as mentioned in \cite[Remark 26]{GM}}, for every choice of a set of indices $I\subset\{0,\dots, n\}$ of cardinality $s\ge 3$, the $\binom{s}{2}$ points $\{t_{ij}\mid i,j\in I\mbox{ and }i<j\}$ are contained in a linear space $\p^{s-2}\subset E$. For later purpose we need to check that, even if the infinitely near points are not in general position, they are not too special. More precisely, we show that they impose independent conditions on low degree divisors of the exceptional divisor $E$.

\begin{lemma}\label{lem: the tangent directions give independent conditions}\label{lemma:T_conditions}
Let $n\geq 2$. Let $p_0,\dots,p_n\in\p^n$ be general points and let $E$ be a hyperplane such that $\{p_0,\dots,p_n\}\cap E=\varnothing$. Define $t_{ij}:=\langle p_i,p_j\rangle\cap E$. Set
$$T:=\{t_{ij}\mid 0\le i<j\le n\} \quad \text{ and } \quad 2T = \sum_{0 \leq i < j \leq n}2t_{ij}.$$
Then the linear systems $\LL_{n-1}^2(T)$, $\LL_{n-1}^3(T)$ and $\LL_{n-1}^3(2T)$ on $E$ are non-special.
	\begin{proof}
The system $\LL_{n-1}^2(T)$ is regular by \cite[Lemma 25]{GM}. As a consequence, $\LL_{n-1}^3(T)$ is regular as well, so we focus on $\LL_{n-1}^3(2T)$. We argue by induction on $n$.
\begin{itemize}
    \item \textit{Case} $n=2$. It is enough to observe that every linear system on $\p^1$ is non-special.
\item \textit{Case} $n\ge 3$. For $i\in\{0,\dots,n\}$, let $H_i:=\langle p_j\mid j\neq i\rangle=\p^{n-1}$. We note that $H_i \cap E$ is a fixed component of $\LL_{n-1}^3(2 T)$ for every $i \in \{0,\ldots,n\}$. Indeed, \[\LL_{n-1}^3(2T)_{|H_i\cap E} \subset \LL_{n-2}^3(\Tr_{H_i \cap E}(2T))=0,\] where the latter equality holds by induction hypothesis. Since all the elements of $\LL_{n-1}^3(2 T)$ have degree $3<n+1$, we conclude that $\LL_{n-1}^3(2 T) = 0$.\qedhere
		\end{itemize}
	\end{proof}
\end{lemma}

\section{Non-defectivity via collisions of fat points}\label{sec: parte teorica interessante}
In this section we prove Theorem \ref{thm: main} and Theorem \ref{thm: quello che usiamo} by using the deformation methods described in the previous section. We compute the dimension of a linear system of divisors with $2$-fat base points on a smooth variety by specializing the points, allowing some of them to collapse. Theorem \ref{thm: generale, ambizioso, avveniristico} is the key of our proofs of Theorem \ref{thm: main} and Theorem \ref{thm: quello che usiamo}. We first recall some auxiliary results that will be useful in the proof.
\begin{lemma}[{\cite[Lemma 1.9]{catalisano2007segre}}]\label{lemma:extra_points}
	Let $V$ be a projective variety and let $H\subset V$ be a positive dimensional subvariety. Let $X \subset V$ be a scheme of fat points and let $ Y \subset H$ be a set of points. Let $\calL$ be a linear system on $V$. Assume that
	\begin{enumerate}
		\item \label{item: restr} $Y$ imposes independent conditions on $\calL(X)_{|H}$, and
		\item \label{item: cardin} $\dim \calL(X) - \dim (\calL(X) \cap I_H) \geq \# Y$.
	\end{enumerate}
	Then $Y$ imposes independent conditions on $\calL(X)$. In particular,
	if $\dim\LL(X) \le \#Y$ and $\calL(X)\cap I_H = 0$, then $\LL(X+ Y)=0$.
\end{lemma}

Theorem \ref{thm: quello che usiamo} will be a direct consequence of the following result.
\begin{theorem}\label{thm: generale, ambizioso, avveniristico} Let $(V,\calL)$ be a polarized smooth irreducible projective variety of dimension $n$. Assume that $\calL$ embeds $V$ as a proper closed subvariety of $\bbP\calL^\vee$ and $W$ is the image of such embedding.
	Let $r\ge n+1$ and assume that
\begin{enumerate}
\item\label{bullet: triplo e doppi impongono indipendenti su LL} $\dim\LL(3,2^{r-n-1})=\dim\LL-\binom{n+2}{2}-(n+1)(r-n-1)$,
\item\label{bullet: col quartuplo LL vuoto} $\dim\LL(4,2^{r-n-1})=0$, and
\item\label{bullet: quartuplo e triplo si comportano bene} $\dim\LL(3)-\dim\LL(4)\ge\binom{n+1}{2}$.
\end{enumerate}
Then 
$\dim\LL(2^{r})=\max\{0,\dim\LL-(n+1)r\}$.
\begin{proof}
Since $r\ge n+1$, we can degenerate $\Ale{\LL(2^r)}$ according to Remark \ref{rmk: collisione di punti doppi}. Let $X\subset V$ be a general scheme of fat points of type $(2^r)$ and let $X_0$ be a general specialization of $X$ when $n+1$ of the $r$ 2-fat points collide together to the point $q\in V$ and the remaining $r-n-1$ 2-fat points are left in general position. Let $E$ be the exceptional divisor of $\Bl_qV$. Let \Ale{$T$ be the set of $N:=\binom{n+1}{2}$} simple points in $E$ evincing the extra tangent directions of $X_0$ in $q$. If $S$ is a subset of $T$, we denote by $\LL(3[S])$ the linear subsystem of $\LL(3)$ of divisors whose strict transform on $\Bl_qV$ contains $S$. \Ale{By semicontinuity,} in order to prove that $\dim\LL(X)=\max\{0,\dim\LL-(n+1)r\}$ it is enough to prove that $\dim\LL(X_0)=\max\{0,\dim\LL-(n+1)r\}$. 

\Ale{Denote $D := \dim\LL(3,2^{r-n-1})=\dim\LL-\binom{n+2}{2}-(n+1)(r-n-1)$. If $D \in \{0,1\}$ there is nothing to prove. Indeed, if $D = 0$, then $\dim \calL(X_0) \leq \dim \calL(3,2^{r-n-1}) = 0$ while, for $D = 1$, $\dim \calL(X_0) = 0$ by Lemma \ref{lemma:extra_points}. From now on we assume that $D \geq 2$.}

By assumption \eqref{bullet: triplo e doppi impongono indipendenti su LL}, we only need to prove that $\Ale{T}$ impose the expected number of conditions on \Ale{the strict transform of} $\LL(3,2^{r-n-1})$ on $\Bl_qV$. Assume by contradiction that
\begin{equation}\label{eq: ipotesi dell-assurdo}
		\dim\LL(X_0)>\max\{0,\dim\LL-(n+1)r\}.
		\end{equation}
		\Ale{Then there exists a ordering of $T = \{t_1,\ldots,t_N\}$ and $1 \leq i \leq \min\left\{N,D\right\}-1$ such that 
		\[
		    \tilde{\calL}(3[t_1,\ldots,t_i],2^{r-n-1})\text{ is regular }
		    \]
		    and 
\begin{equation}\label{eq:base_condition_allT}
t_{j} \in \Bs(\tilde{\calL}(3[t_1,\ldots,t_i],2^{r-n-1}))\text{ for all }i+1 \leq j \leq N,
\end{equation} where the tilde means that we consider the strict transform with respect to the blow-up of $V$ at $q$ and $\Bs$ denotes the base locus of a linear system}. Recall that the definition of \Fra{$t_1,\dots,t_N$} depends only on the first $n+1$ colliding $2$-fat points of the original scheme $X$. Hence the special points \Fra{$t_{i+1},\dots,t_N$ do }
not depend on the support of the $r-n-1$ $2$-fat points, as long as they are in general position. Define
		\[
h := \min\left\{a \in \bbN \mid t_{i+1} \in \Bs(\calL(3[t_1,\ldots,t_i],2^{a})\right\}.
		\]
By construction $h\in\{0,\ldots,r-n-1\}$.
\begin{claim}\label{claim1}
			$h \geq 1$.
		\end{claim}
		\begin{proof}[Proof of Claim \ref{claim1}]
It is enough to show that $T$ imposes independent conditions on \Ale{$\tilde{\LL}(3)$}. Let $E \cong \bbP^{n-1}$ be the exceptional divisor of the blow-up $\Bl_qV$. We want to apply Lemma \ref{lemma:extra_points}, so we check that its two hypotheses hold. The first one 
	is satisfied because $T$ imposes independent conditions on $\tilde{\calL}(3)_{|E}$  by Lemma \ref{lem: the tangent directions give independent conditions}. For the second one, note  that the linear \Ale{subsystem} of $\tilde{\calL}(3)$ of divisors which contain the exceptional divisor $E$ is $\tilde{\calL}(4)$ on $\Bl_qV$. Hence the second hypothesis of Lemma~\ref{lemma:extra_points} is satisfied thanks to our assumption \eqref{bullet: quartuplo e triplo si comportano bene}.
		\end{proof}
		By definition of $h$,
		\begin{equation}\label{eq:base_condition}
		t_{i+1} \not\in \Bs(\tilde{\calL}(3[t_1,\ldots,t_i],2^{h-1})) \text{ but } t_{i+1}\in \Bs(\tilde{\calL}(3[t_1,\ldots,t_i],2^h)).
		\end{equation}
\Ale{Let $\varphi$ denote the rational function on $\Bl_q(V)$ defined by the linear system $\tilde{\calL}(3[t_1,\ldots,t_i],2^{h-1})$ and $\psi$ be the one defined by $\tilde{\calL}(3[T],2^{h-1})$. Note that the closure of the image of $\varphi$ is a proper subvariety because 
\[
    \dim \calL(3[t_1,\ldots,t_i],2^{h-1}) \geq \dim \calL(3,2^{r-n-2}) - i = D + (n+1) - i \geq n+2,
\]
since $i \leq D-1$. Hence it makes sense to consider its tangent space. Recall that the tangent space to $\overline{\im(\varphi)}$ at a general point $p \in \Bl_qV$ corresponds under $\varphi^*$ to the linear subsystem $\tilde{\calL}(3[t_1,\ldots,t_i],2^{h-1}) \cap I_p^2$.} As mentioned before, $t_1,\dots,t_N$ depend only on the collapsed $2$-fat points and not the remaining ones, as long as they are general. Therefore, by \eqref{eq:base_condition}, 
\[
t_{i+1} \in \Bs(\tilde{\calL}(3[t_1,\ldots,t_i],2^{h-1}) \cap I_p^2)\text{ for a general choice of }p.
\] 
Equivalently, $\overline{\im(\varphi)}$ is a cone and $\varphi(t_{i+1})$ is a point in the vertex of the Zariski closure of the image $\overline{\im(\varphi)}$ of $\varphi$. Indeed, recall that the vertex is given by the intersection of all tangent spaces to the variety (see \cite[Proposition 4.6.11]{FOV}). By \eqref{eq:base_condition_allT}, we deduce that $t_j \in \Bs(\tilde{\calL}(3[t_1,\ldots,t_i],2^h))$ for every $i+1 \leq j \leq N$; \Ale{in other words, \begin{equation}\label{eq:vertex_eq}
\tilde{\calL}(3[t_1,\ldots,t_i],2^h) = \tilde{\calL}(3[T],2^h).
\end{equation}
By \eqref{eq:vertex_eq} we have that, for a general point $p \in \Bl_qV$, 
\[
\tilde{\calL}(3[t_1,\ldots,t_i],2^{h-1}) \cap I_p^2 = \tilde{\calL}(3[T],2^{h-1}) \cap I_p^2,
\]
i.e., the vertices of two images coincide, namely $\vert(\overline{\im(\varphi)}) = \vert(\overline{\im(\psi)})$.} As illustrated in  Section~\ref{sec:collapsing_points}, for every $t\in T$ there are $n-1$ lines on the exceptional divisor $E \simeq \bbP^{n-1}$ each containing $t$ and two more points of $T$. Let $\ell_1,\ldots,\ell_{n-1}$ be the lines through $t_{i+1}$.  
 Since $\tilde{\calL}(3[T],2^h)_{|E}$ is a linear system of cubics on $E$ passing through $T$ and there is a unique cubic on a projective line with three base points, we have that the map $\psi$ contracts each one of the lines $\ell_1,\ldots,\ell_{n-1}$. From the previous part, we deduce that they are contracted to $\varphi(t_{i+1})$ which is a point of the vertex of $\overline{\im(\varphi)}$, that is $\psi(\ell_j) \in \vert(\overline{\im(\psi)})$.
 Therefore, we conclude that $\ell_j \in \Bs(\tilde{\calL}(3[T],2^h))$ for all $j \in \{1,\ldots,n-1\}$. The lines $\ell_1,\ldots,\ell_{n-1}$ span $E$, hence we conclude that all cubics in $\tilde{\calL}(3[T],2^h)_{|E}$ must be singular at $t_{i+1}$. By monodromy, the latter holds for all points in $T$. By Lemma \ref{lemma:T_conditions}, there are no cubics on $E$ which are singular at all the points of $T$, we conclude that $\calL(3[T],2^h) \subset \calL(4,2^h)$ and, in particular,
		\[
		\calL(3[T],2^{r-n-1}) \subset \calL(4,2^{r-n-1}).
		\]
		By assumption \eqref{bullet: col quartuplo LL vuoto}, we conclude that 	\begin{equation}\label{eq:main}
			\calL(3[T],2^{r-n-1})=0.
		\end{equation}
Here we get a contradiction. If $\vdim \calL(3[T],2^{r-n-1}) > 0$, then \eqref{eq:main} contradicts the fact that the actual dimension is at least the virtual one. If $\vdim\calL(3[T],2^{r-n-1}) \le 0$, then \eqref{eq:main} contradicts \eqref{eq: ipotesi dell-assurdo}.\end{proof}\end{theorem}

	\begin{proof}[Proof of Theorem \ref{thm: quello che usiamo}]
As we recalled in Section \ref{sec: linear systems}, in order to show that $V$ is not defective we have to prove that $\dim\LL(2^{r})=\max\{0,\dim\LL-(n+1)r\}$ for every $r\in\N$. By Remark \ref{rmk: bastano r alto e r basso}, it is enough to prove this for $r\in\left\{\rd{\frac{\dim\LL}{n+1}},\ru{\frac{\dim\LL}{n+1}}\right\}$. We note that
		\[r\ge\rd{\frac{\dim\LL}{n+1}}\ge\rd{\frac{(n+1)^2}{n+1}}=n+1.\]
Hence we have enough $2$-fat points to apply Theorem \ref{thm: generale, ambizioso, avveniristico}.
	\end{proof}
We use the latter results to prove the non-defectivity of Segre-Veronese varieties of $\bbP^m\times\bbP^n$ embedded in bidegree $(c,d)$ whenever $c,d \geq 3$, proving a conjecture by Abo and Brambilla. We believe that this underlines the strength of these type of deformations. We also recall that in \cite{GM} the first author and Mella used similar methods to completely solve the problem of identifiability of Waring decompositions for general polynomials in any degree and number of variables.


\begin{proof}[Proof of Theorem \ref{thm: main}]
Since a single $j$-fat point imposes the expected number of conditions on $\calL_{m\times n}^{c,d}$ whenever $c,d \geq j-1$, it is simple to check that Segre-Veronese varieties $\SV_{m \times n}^{c,d}$ satisfy the last two hypotheses of Theorem \ref{thm: quello che usiamo}. We devote Sections \ref{sec: 33}, \ref{sec: bidegree (3,4)} and \ref{sec: bidegree (4,4)} to prove the first two conditions for the cases $(c,d) \in \{(3,3),(3,4),(4,4)\}$. This completes the proof of Theorem \ref{thm: main}.
\end{proof}
\begin{remark}
If $c \leq j-2$ or $d \leq j-2$, then a $j$-fat point does not impose independent conditions on $\LL_{m\times n}^{c,d}$. Therefore our proof of Theorem \ref{thm: main} does not apply to lower bidegrees.
\end{remark}
\Ale{We can now prove also} our results on identifiability.

\Fra{
\begin{proof}[Proof of Corollary \ref{cor: identifiability in general}] 
Following the notation of \cite[Section 1]{casarotti2019non}, we let $W^{(k)}$ be the variety parametrizing unordered sets of $k$ points of $W$. Consider the \textit{abstract secant variety}
\[
\sec_k(W)=\overline{\{(w_1,\dots,w_k,p)\in W^{(k)}\times\p^N\mid p\in\langle w_1,\ldots,w_k\rangle\}}\subset W^{(k)}\times\p^N
\]
and call \textit{secant map} the second projection $\pi:\sec_k(W)\to\p^N$. The variety $W$ is not defective by Theorem \ref{thm: quello che usiamo}. Hence, since $k \leq \rd{\frac{\dim\LL}{\dim(W)+1}}$, the map $\pi$ is generically finite. This fact
allows us to conclude by applying \cite[Theorem, page 2]{casarotti2019non}.
\end{proof}

\begin{proof}[Proof of Corollary \ref{cor: identifiability for SV}]
\Ale{Since $k$-identifiability implies $(k-1)$-identifiability}, it is enough to prove our statement for $k=\left\lfloor \frac{1}{m+n+1}\binom{m+c}{c}\binom{n+d}{d}\right\rfloor$. Theorem \ref{thm: main}, together with the definition of $k$, tells us that the secant map is generically finite. In order to apply \cite[Theorem, page 2]{casarotti2019non} we just have to check that the assumption on the dimension of $\SV_{m\times n}^{c,d}$ is satisfied. Indeed,
\begin{align*}
k-2\dim(\SV_{m\times n}^{c,d})
&\ge \frac{1}{m+n+1}\binom{m+3}{3}\binom{n+3}{3}-1-2(m+n)\\
&=\frac{1}{m+n+1}\left[\binom{m+3}{3}\binom{n+3}{3}-(m+n+1)(1+2m+2n)\right]\\
&\ge \frac{1}{m+n+1}\left[(m+2)(m+1)(n+2)(n+1)-(m+n+1)(1+2m+2n)\right]\\
&=\frac{m^2n^2+3m^2n+3mn^2+5mn+3m+3n+3}{m+n+1}>0.\qedhere
\end{align*}
\end{proof}
}

\begin{notation}\label{notation}
	Throughout the next sections, we will use the following notation. For ${c,d,m,n\in\N}$, 
	\begin{align*}
&r^*(c,d;m,n):=\ru{\frac{\binom{m+c}{m}\binom{n+d}{n}}{m+n+1}},&& 	r_*(c,d;m,n):=\rd{\frac{\binom{m+c}{m}\binom{n+d}{n}}{m+n+1}},\\
		&k^*(c,d;m,n):=r^*(c,d;m,n)-m-n-1,&& k_*(c,d;m,n):=r_*(c,d;m,n)-m-n-1.
	\end{align*}
\end{notation}

\section{\texorpdfstring{$\SV_{m\times n}^{3,3}$}{SVmn33} is not defective.}\label{sec: 33}
In this section we show that Theorem \ref{thm: main} holds for $c=d=3$. The specializations need to be chosen carefully and satisfy several arithmetic properties: in order to make our proofs easier to read, we moved some elementary but tedious computations to Appendix \ref{appendix: contacci}. 

\begin{proposition}\label{pro: bigrado 33 come punti doppi} If $m$ and $n$ are positive integers, then $\LL_{m\times n}^{3,3}(2^r)$ is nonspecial for every $r\in\N$.
\end{proposition}
\begin{proof}
 We may assume $m \leq n$ by symmetry and $2 \leq m$ by \cite[Theorem 3.1]{BalBerCat12}. We apply Theorem~\ref{thm: quello che usiamo} by checking its hypothesis:
condition (1) in Proposition~\ref{pro: triplo e doppi}, condition (2) in Proposition~\ref{pro: quartuplo e doppi}, condition (3) can be checked easily since a single $j$-fat point always impose independent conditions on $\calL_{m\times n}^{c,d}$ for $c,d \geq j-1$ and condition (4) is trivial.
\end{proof}

\begin{proposition}\label{pro: triplo e doppi} 
If $m$ and $n$ are positive integers, then $\LL_{m\times n}^{3,3}(3,2^{k^*(3,3;m,n)})$ is regular.
\begin{proof} 
Without loss of generality, we assume that $m\le n$. We proceed by induction on $m$. The case $m=1$ is Lemma \ref{lem: caso iniziale triplo e doppi}. Assume that $m\ge 2$. Let $D\subset\p^m\times \p^n$ be a divisor of bidegree $(1,0)$ and consider a scheme of fat points $X$ of type $(3,2^{k^*(3,3;m,n)})$ such that $X \cap D$ is general of type $(3,2^{k^*(3,3;m-1,n)})$ on $D$. Note that
\begin{align}
\vdim & \LL_{m\times n}^{3,3}(X) \nonumber \\ & = \binom{m+3}{3}\binom{n+3}{3} - \binom{m+n+2}{2} - (m+n+1)\left(\left\lceil \frac{\binom{m+3}{3}\binom{n+3}{3}}{m+n+1} \right\rceil-(m+n+1)\right) \nonumber \\
& \geq (m+n+1)^2 - \binom{m+n+2}{2} - (m+n) = \binom{m+n+1}{2} - (m+n) \ge 0. \label{eq:vdim_withTriple}
\end{align}
As explained in Section \ref{sec:inductive}, it is enough to prove that  residue and trace of $\calL_{m\times n}^{3,3}(X)$ with respect to $D$ are regular.
		\begin{itemize}
\item \textit{Trace.}  The trace of $X$ on $D \cong \bbP^{m-1}\times\bbP^n$ is general of type $(3,2^{k^*(3,3;m-1,n)})$, hence the linear system $\calL_{(m-1)\times n}^{3,3}(\Tr_D(X))$ is regular by induction.
\item \textit{Residue.} The residue $\Res_D(X)$ is of type $(2^{1+k^*(3,3;m,n)-k^*(3,3;m-1,n)},1^{k^*(3,3;m-1,n)})$, where $\Res_D(X) \cap D$ is general of type $(2,1^{k^*(3,3;m-1,n)})$ on $D$. The system $\LL_{m\times n}^{2,3}(\Res_D(X))$ has non-negative virtual dimension by Lemma \ref{conto: vdim positiva, bigrado (2,3)}, and $\LL_{m\times n}^{2,3}(2^{k^*(3,3;m,n)-k^*(3,3;m-1,n)+1})$ is regular by Lemma \ref{lem: bigrado 2,3 con pochi punti doppi}. In order to prove that $\calL_{m\times n}^{2,3}(\Res_D(X))$ is regular, we need to show that the $k^*(3,3;m-1,n)$ simple points on $D$ impose independent conditions on the linear system $\LL_{m\times n}^{2,3}(2^{1+k^*(3,3;m,n)-k^*(3,3;m-1,n)})$. Thanks to Lemma \ref{lemma:extra_points}, we only have to prove that  $\LL_{m\times n}^{1,3}(1,2^{k^*(3,3;m,n)-k^*(3,3;m-1,n)})=0$. This is true by Lemma \ref{conto: kbasso, usare bcc} and \cite[Theorem~2.3]{BCC}. \qedhere
	\end{itemize}
	\end{proof}
\end{proposition}

\begin{lemma}
	\label{lem: caso iniziale triplo e doppi} If $n$ is a positive integer, then $\LL_{1\times n}^{3,3}(3,2^{k^*(3,3;1,n)})$ is regular.
	\begin{proof}
We proceed by induction on $n$. The case $n = 1$ is checked directly with the support of an algebra software; see Appendix \ref{appendix: M2}. Assume that $n \geq 2$. Let $D \subset \bbP^1\times\bbP^n$ be a divisor of bidegree $(0,1)$ and consider the scheme of fat points $X$ of type $(3,2^{k^*(3,3;1,n)})$ such that $X \cap D$ is general of type $(3,2^{k^*(3,3;1,n-1)})$ on $D$. Note that $\vdim \LL_{1\times n}^{3,3}(X) \geq 0$ by \eqref{eq:vdim_withTriple}. As explained in Section \ref{sec:inductive}, it is enough to prove that residue and trace of $\calL_{1\times n}^{3,3}(X)$ with resepct to $D$ are regular.
		\begin{itemize}
\item \textit{Trace.} The trace of $X$ on $D \cong \bbP^1\times\bbP^{n-1}$ is general of type $(3,2^{k^*(3,3;1,n-1)})$ and the linear system $\calL_{1\times (n-1)}^{3,3}(3,2^{k^*(3,3;1,n-1)})$ is regular by induction.
\item \textit{Residue.} $\Res_D(X)$ is a scheme of fat points of type $(2^{1+k^*(3,3;1,n)-k^*(3,3;1,n-1)},1^{k^*(3,3;1,n-1)})$, where $\Res_D(X) \cap D$ is general of type $(2,1^{3,3;k^*(1,n-1)})$ on $D$. By symmetry and by Lemma~\ref{conto: vdim positiva, bigrado (2,3)}, the virtual dimension of $\calL_{1\times n}^{3,2}(2^{1+k^*(3,3;1,n)-k^*(3,3;1,n-1)},1^{k^*(3,3;1,n-1)})$ is non-negative. The system $\calL_{1\times n}^{3,2}(2^{1+k^*(3,3;1,n)-k^*(3,3;1,n-1)})$ is regular by \cite[Theorem 3.1]{BalBerCat12}. Hence, in order to prove that $\calL_{1\times n}^{3,2}(\Res_D(X))$ is regular, we need to show that the additional $k^*(3,3;1,n-1)$ simple points lying on $D$ impose independent conditions on $\calL_{1 \times n}^{3,2}(2^{1+k^*(3,3;1,n)-k^*(3,3;1,n-1)})$. Thanks to Lemma~\ref{lemma:extra_points}, we just need to check that $\LL_{1\times n}^{3,1}(1,2^{k^*(3,3;1,n)-k^*(3,3;1,n-1)})=0$. This holds by Lemma \ref{conto: bigrado 3,1 atteso vuoto} and \cite[Theorem 3.1]{BalBerCat12}. 
\qedhere
	\end{itemize}
	\end{proof}
\end{lemma}

\begin{lemma}
	\label{lem: bigrado 2,3 con pochi punti doppi}
If $1 \leq m\le n$, then $\LL_{m\times n}^{2,3}(2^{1+k^*(3,3;m,n)-k^*(3,3;m-1,n)})$ is regular.
\begin{proof} 
In order to simplify the notation, we set $f(m,n):=1+k^*(3,3;m,n)-k^*(3,3;m-1,n)$. We argue by induction on $m$. The case $m = 1$ follows by \cite[Theorem 4.2]{Abrescia}. Assume that $m\ge 2$ and let $D\subset\p^m\times\p^n$ be a divisor of bidegree $(1,0)$. Let $X$ be a scheme of fat points of type $(2^{f(m,n)})$ such that $X \cap D$ is general of type $(2^{f(m-1,n)})$ on $D$. Note that we are allowed to do it because $f(m-1,n) \leq f(m,n)$ by Lemma \ref{conto:funzioniCrescenti}(4). By Lemma \ref{conto:f_vdim_positiva},  $\vdim\calL_{m\times n}^{2,3}(X) \geq 0$. Hence it is enough to prove that residue and trace of $\calL_{m\times n}^{2,3}(X)$ with respect to $D$ are regular.
		\begin{itemize}
\item {\it Trace}.  The trace is a general scheme of fat points on $D \cong \bbP^{m-1}\times\bbP^n$ of type $(2^{f(m-1,n)})$ and $\calL^{2,3}_{(m-1)\times n}(2^{f(m-1,n)})$ is regular by induction.
\item {\it Residue}. The residue $\Res_D(X)$ is a scheme of fat points of type $(2^{f(m,n)-f(m-1,n)},1^{f(m-1,n)})$, where $\Res_D(X) \cap D$ is a general scheme of type $(1^{f(m-1,n)})$ on $D$. By Lemma \ref{conto:vdim_lowerbound_13}, \[\vdim\LL_{m\times n}^{1,3}(\Res_D(X))\ge 0.\]
By Lemma \ref{conto:upperBound_l-l}(1),
	\[
f(m,n)-f(m-1,n)\le\rd{\frac{m+1}{m+n+1}\binom{n+3}{3}}-m,
	\]
hence $\LL_{m\times n}^{1,3}(2^{f(m,n)-f(m-1,n)})$ is regular by \cite[Corollary 2.2]{BCC}. In order to prove that $\calL_{m\times n}^{1,3}(\Res_D(X))$ is regular, we need to show that the $f(m-1,n)$ simple points on $D$ impose independent conditions on $\LL_{m\times n}^{1,3}(2^{f(m,n)-f(m-1,n)})$. By Lemma \ref{lemma:extra_points}, it is enough to check that $\dim\LL_{m\times n}^{0,3}(2^{f(m,n)-f(m-1,n)})=0$. Lemma \ref{conto:funzioniCrescenti}(3) and Theorem \ref{thm:AH} imply that \[\LL_{m\times n}^{0,3}(2^{f(m,n)-f(m-1,n)})\cong\LL_{n}^{3}(2^{f(m,n)-f(m-1,n)})=0.\qedhere\]\end{itemize}
	\end{proof}
\end{lemma}
Now the proof of Proposition \ref{pro: triplo e doppi} is complete, so $\SV_{m\times n}^{3,3}$ satisfies the first hypothesis of Theorem~\ref{thm: quello che usiamo}. Let us move to the second one.
\begin{proposition}
\label{pro: quartuplo e doppi}
If $n\ge m\ge 2$, then $\dim\LL_{m\times n}^{3,3}(4,2^{k_*(3,3;m,n)})=0$.
	\begin{proof}
	We argue by induction on $m$. By Lemma \ref{lem: quartuplo e doppi, 2 x n},  $\LL_{2\times n}^{3,3}(4,2^{k_*(3,3;2,n)})=0$. Assume that $m\ge 3$ and take a divisor $D\subset\p^m\times\p^n$ of bidegree $(1,0)$. Let $X$ be a scheme of fat points of type $(4,2^{k_*(3,3;m,n)})$ such that $X \cap D$ is general of type $(4,2^{k_*(3,3;m-1,n)})$ on $D$. As explained in Section \ref{sec:inductive}, it is enough to prove that residue and trace of $\calL_{m\times n}^{3,3}(X)$ are zero.
		\begin{itemize}
\item {\it Trace.} The trace of $X$ on $D$ is a general scheme of type $(4,2^{k_*(3,3;m-1,n)})$ and we know that $\LL_{(m-1)\times n}^{3,3}(4,2^{k_*(3,3;m-1,n)})=0$ by induction hypothesis.
\item {\it Residue.} $\Res_D(X)$ is a scheme of fat points  of type $(3,2^{k_*(3,3;m,n)-k_*(3,3;m-1,n)},1^{k_*(3,3;m-1,n)})$ where $X \cap D$ is general of type \Ale{$(3,1^{k_*(3,3;m-1,n)})$} on $D$. The residue linear system is expected to be zero by Lemma \ref{conto: bigrado 2,3, triplo e doppi, left hand side}. The linear system $\LL_{m\times n}^{2,3}(3,2^{k_*(3,3;m,n)-k_*(3,3;m-1,n)})$ is regular by Lemma \ref{lem: triplo e doppi, bigrado ()2,3)}. Now we need to prove that the extra $k_*(3,3;m-1,n)$ simple points on $D$ impose enough conditions to make $\LL_{m\times n}^{2,3}(\Res_D(X))$ \Ale{to be zero}. 
By Lemma \ref{lemma:extra_points}, it is enough to prove that $\LL_{m\times n}^{1,3}(2,2^{k_*(3,3;m,n)-k_*(3,3;m-1,n)})= 0$. Thanks to \cite[Corollary 2.2]{BCC}, we just need to show that
			\[
1+k_*(3,3;m,n)-k_*(3,3;m-1,n)\ge\ru{\frac{m+1}{m+n+1}\binom{n+3}{3}}+m,\]
and this is done in Lemma \ref{conto: bigrado 2,3, usare bcc}.\qedhere
		\end{itemize}
	\end{proof}
\end{proposition}

\begin{lemma}
	\label{lem: quartuplo e doppi, 2 x n} 
	If $n$ is a positive integer, then $\dim \LL_{2\times n}^{3,3}(4,2^{k_*(3,3;2,n)})=0$.
	\begin{proof}
We argue by induction on $n$. By a software computation, $\dim\LL_{2\times 1}^{3,3}(4,2^{k_*(3,3;2,1)})= 0$; see Appendix \ref{appendix: M2}. Assume that $n\ge 2$ and consider a divisor $D\subset\p^2\times\p^n$ of bidegree $(0,1)$. Let $X$ be a scheme of fat points of type $(4,2^{k_*(3,3;2,n)})$ such that $X \cap D$ is general of type $(4,2^{k_*(3,3;2,n-1)})$ on $D$.  As explained in Section \ref{sec:inductive}, it is enough to prove that the residue and trace of $\calL_{2\times n}^{3,3}(X)$ are zero.
		\begin{itemize}
\item {\it Trace.}  The trace of $X$ on $D$ is a general scheme of type $(4,2^{k_*(3,3;2,n-1)})$ and the linear system $\LL_{2\times (n-1)}^{3,3}(4,2^{k_*(3,3;2,n-1)})$ is zero by induction hypothesis.
			\item {\it Residue.}
$\Res_D(X)$ is a scheme of fat points of type $(3,2^{k_*(3,3;2,n)-k_*(3,3;2,n-1)},1^{k_*(3,3;2,n-1)})$, where $X \cap D$ is general of type $(3,1^{k_*(3,3;2,n-1)})$. The residue linear system is expected to be zero by Lemma \ref{conto: left hand side del punto quartuplo, atteso vuoto}. By Lemma \ref{lem: 2 x n, bigrado (3,2)}, $\LL_{2\times n}^{3,2}(3,2^{k_*(3,3;2,n)-k_*(3,3;2,n-1)})$ is regular. Now we prove that the $k_*(3,3;2,n-1)$ simple points on $D$ impose enough conditions to make $\calL_{2\times n}^{3,2}(\Res_D(X))$ zero. 
By Lemma \ref{lemma:extra_points}, it is enough to show that $\calL_{2\times n}^{3,1}(2^{1+k_*(3,3;2,n)-k_*(3,3;2,n-1)})=0$. By \cite[Corollary 2.2]{BCC}, the latter is guaranteed by
\[
1+k_*(3,3;2,n)-k_*(3,3;2,n-1) \geq \left\lceil \frac{10(n+1)}{n+3}\right\rceil + n
\]
which holds by Lemma \ref{conto: bigrado (3,1) è vuoto}. \qedhere
		\end{itemize}
	\end{proof}
\end{lemma}

\begin{lemma}
	\label{lem: triplo e doppi, bigrado ()2,3)}
Let $n\ge m\ge 2$. Then $\LL_{m\times n}^{2,3}(3,2^{k_*(3,3;m,n)-k_*(3,3;m-1,n)})$ is regular.
	\begin{proof}
In order to simplify the notation, set
\[
    \ell(m,n):=k_*(3,3;m,n)-k_*(3,3;m-1,n).\]
We proceed by induction on $m$. The case $m=2$ is solved by Lemma \ref{lem: 2 x n, bigrado (2,3)}. Assume $m\ge 3$ and consider a divisor $D\subset\p^m\times\p^n$ of bidegree $(1,0)$. Let $X$ be a scheme of fat points of type $(3,2^{\ell(m,n)})$ such that $X \cap D$ is general of type $(3,2^{\ell(m-1,n)})$ on $D$. We are allowed to do it because $\ell(m-1,n)\le \ell(m,n)$ by Lemma \ref{conto:funzioniCrescenti}(2). As explained in Section \ref{sec:inductive}, it is enough to prove that the residue and the trace of $\calL_{m\times n}^{2,3}(X)$ with respect to $D$ are regular. 
\begin{itemize}
\item {\it Trace.} The trace of $X$ on $D$ is a general scheme of fat points of type $(3,2^{\ell(m-1,n)})$ and $\LL_{(m-1)\times n}^{2,3}(3,2^{\ell(m-1,n)})$ is 
regular by inductive hypothesis.
\item {\it Residue.} The residue $\Res_D(X)$ is a scheme of fat points of type $(2^{1+\ell(m,n)-\ell(m-1,n)},1^{\ell(m-1,n)})$, where $\Res_D(X)\cap D$ is a general scheme of type $(2,1^{\ell(m-1,n)})$. The system  $\calL_{m\times n}^{1,3}(\Res_D(X))$ has non-negative virtual dimension by Lemma \ref{conto: bigrado 1,3, doppi e semplici}. Lemma \ref{conto:upperBound_l-l}(2) shows that
\[
1+\ell(m,n)-\ell(m-1,n)\le\rd{\frac{m+1}{m+n+1}\binom{n+3}{3}}-m,
\]
thus $\LL_{m\times n}^{1,3}(2^{1+\ell(m,n)-\ell(m-1,n)})$ is regular by \cite[Corollary 2.2]{BCC}. In order to prove that $\calL_{m\times n}^{1,3}(\Res_D(X))$ is regular, we need to prove that the $\ell(m-1,n)$ general simple points on $D$ impose independent conditions on $\LL_{m\times n}^{1,3}(2^{1+\ell(m,n)-\ell(m-1,n)})$. By Lemma \ref{lemma:extra_points} we only need to show that $\calL_{m\times n}^{0,3}(1,2^{\ell(m,n)-\ell(m-1,n)})=0$. We conclude by observing that $\calL_{m\times n}^{0,3}(1,2^{\ell(m,n)-\ell(m-1,n)})\cong\calL_n^3(1,2^{\ell(m,n)-\ell(m-1,n)})=0$ by Lemma \ref{conto:funzioniCrescenti}(1) and Theorem \ref{thm:AH}.\qedhere
\end{itemize}
\end{proof}
\end{lemma}

\begin{lemma}\label{lem: 2 x n, bigrado (2,3)}
Let $n \geq 2$. Then $\LL_{2\times n}^{2,3}(3,2^{\ell(2,n)})$ is regular.
\begin{proof} 
We check the case $n = 2$ by a software computation; see Appendix \ref{appendix: M2}. Let $n \geq 3$ and set
\[ s(n):=\frac{n(n+3)}{2}\in\N.
\]
Let $D\subset\p^2\times\p^n$ be a bidegree $(1,0)$ divisor and let $X$ be a scheme of fat points of type $(3,2^{\ell(2,n)})$ such that $X \cap D$ is of type $(3,2^{s(n)})$. Note that we are allowed to do it because $s(n) \leq \ell(2,n)$ by Lemma~\ref{conto: s minore di l}(1). As explained in Section \ref{sec:inductive}, it is enough to prove that the residue and trace of $\calL_{2 \times n}^{2,3}(X)$ with respect to $D$ are regular.
\begin{itemize}
\item {\it Trace.} The trace on $X$ is a general scheme of fat points of type $(3,2^{s(n)})$ on $D$ and the linear system $\LL_{1\times n}^{2,3}(3,2^{s(n)})$ is regular by Lemma \ref{lem:1xn_(3,2s)}.
\item {\it Residue.} The residue $\Res_D(X)$ is a scheme of fat points of type $(2^{1+\ell(2,n)-s(n)},1^{s(n)})$ where $X \cap D$ is general of type $(2,1^{s(n)})$ on $D$. The system $\LL^{1,3}_{m\times n}(\Res_D(X))$ has non-negative virtual dimension by Lemma \ref{conto:lowerBound_vdim_s}. By Lemma \ref{conto: s consente di usare bcc}, since $n \geq 3$ we have
\[1+\ell(2,n)-s(n)\le\rd{\frac{3}{n+3}\binom{n+3}{3}}-2,\]
hence $\LL_{2\times n}^{1,3}(2^{1+\ell(2,n)-s(n)})$ is regular by \cite[Theorem 2.3]{BCC}. Now we need to show that the $s(n)$ simple points on $D$ impose independent conditions on $\LL_{2\times n}^{1,3}(2^{1+\ell(2,n)-s(n)})$. By Lemma~\ref{lemma:extra_points}, it is enough to show that $\LL_{2\times n}^{0,3}(1,2^{\ell(2,n)-s(n)})=0$. We conclude by observing that $\LL_{2\times n}^{0,3}(1,2^{\ell(2,n)-s(n)})\cong\calL_n^3(1,2^{\ell(2,n)-s(n)})=0$ by Lemma \ref{conto: s minore di l}(2) and Theorem \ref{thm:AH}.\qedhere
\end{itemize}
\end{proof}
\end{lemma}

\begin{lemma}
\label{lem:1xn_(3,2s)}
If $n\ge 2$, then $\LL_{1\times n}^{2,3}(3,2^{s(n)})$ is regular. In particular, it is zero.
\begin{proof}
Note that
\begin{align*}
\vdim\LL_{1\times n}^{2,3}(3,2^{s(n)})=3\binom{n+3}{3}-\binom{n+3}{2}-\frac{n(n+2)(n+3)}{2}=
0.
\end{align*}
We have to prove that it is indeed zero. We proceed by a double-step induction on $n$. A software computation shows that $\dim \LL_{1\times 2}^{2,3}(3,2^{s(2)})=\dim \LL_{1\times 3}^{2,3}(3,2^{s(3)})=0$; see Appendix \ref{appendix: M2}. Assume that $n\ge 4$. Let $A\cong\p^1\times\p^{n-2}$ be a subvariety defined by two general forms of bidegree $(0,1)$. Let $X = X_A + X_\circ$ be a scheme of fat points of type $(3,2^{s(n)})$, where
\begin{center}
	\begin{tabular}{r l}\label{tab:X}
		$X_A$ & is a scheme of type $(3,2^{s(n-2)})$ with general support on $A \cong\bbP^1 \times\bbP^{n-2}$; \\
		$X_\circ$ & is a general scheme of type $(2^{2n+1})$ with support outside $A$.
	\end{tabular}
\end{center}
As explained in Section \ref{sec:inductive}, we consider the exact sequence
\[
0\to
\LL_{1\times n}^{2,3}(A+X)
\to
\LL_{1\times n}^{2,3}(X)
\to
\LL_{1\times (n-2)}^{2,3}(\Tr_AX).
\]
Then it is enough to prove that both the left-most and the right-most linear systems are zero. By induction hypothesis, $\LL_{1\times (n-2)}^{2,3}(\Tr_AX)=0$, while $\LL_{1\times n}^{2,3}(A+X)=0$ by Lemma \ref{lem:1xn_(3,2s)_2ndStep}
\end{proof}
\end{lemma}

\begin{lemma}
\label{lem:1xn_(3,2s)_2ndStep} Let $n\ge 4$. Let $A\cong\p^1\times\p^{n-2}$ be a subvariety of $\p^1\times\p^n$ defined by two general forms of bidegree $(0,1)$. Let $X = X_A + X_\circ$ be a scheme of fat points as in the proof of Lemma \ref{tab:X}. Then $\LL_{1\times n}^{2,3}(A+X)=0$.
\begin{proof}
We proceed by a double-step induction on $n$. A software computation shows that the statement holds for $n=4$ and $n=5$; see Appendix \ref{appendix: M2}. Assume that $n\ge 6$. Let $B\cong\p^1\times\p^{n-2}$ be another subvariety defined by two general forms of bidegree $(0,1)$. Consider a specialization $Y = Y_{A\cap B} + Y_A + Y_B + Y_\circ$ of $X$, where
\begin{center}
	\begin{tabular}{r l}\label{tab:Y}
$Y_{A\cap B}$ & is general of type $(3,2^{s(n-4)})$ on $A \cap B \cong \bbP^1 \times \bbP^{n-4}$; \\
		$Y_A$ & is general of type $(2^{2n-3})$ on $A$, outside $B$; \\
		$Y_B$ & is general of type $(2^{2n-3})$ on $B$, outside $A$; \\
		$Y_\circ$ & is general of type $(2^4)$ with support outside $A \cup B$.
	\end{tabular}
\end{center}
Now it is enough to prove that $\calL_{1\times n}^{2,3}(A + Y)=0$. Consider the exact sequence
\[
0\to
\LL_{1\times n}^{2,3}(A+B+Y)
\to
\LL_{1\times n}^{2,3}(A+Y)
\to
\LL_{1\times (n-2)}^{2,3}(A\cap B+\Tr_BY).
\]
\begin{itemize}
\item \textit{Trace}. By induction hypothesis, $\LL_{1\times (n-2)}^{2,3}(A\cap B+\Tr_BY)=0$.
\item \textit{Residue}. By Lemma \ref{lem:1xn_(3,2s)_3rdStep}, $\LL_{1\times n}^{2,3}(A+B+Y)=0$.\qedhere
\end{itemize}
\end{proof}
\end{lemma}

\begin{lemma}
\label{lem:1xn_(3,2s)_3rdStep}
Let $n\ge 6$. Let $A, B\cong\p^1\times\p^{n-2}$ be subvarieties of $\p^1\times\p^n$, each defined by two general forms of bidegree $(0,1)$.  Let $Y = Y_{A\cap B} + Y_A + Y_B + Y_\circ$ be a scheme of fat points as in the proof of Lemma \ref{tab:Y}. Then $\LL_{1\times n}^{2,3}(A+B+Y)=0$.
\begin{proof}
We proceed by induction on $n$. A software computation shows that our statement holds for $n=6$ and $n=7$; see Appendix \ref{appendix: M2}. Assume that $n\ge 8$. Let $C\cong\p^1\times\p^{n-2}$ be another subvariety defined by two general forms of bidegree $(0,1)$. We consider the specialization $Z = Z_{A\cap B \cap C} + Z_{A \cap B} + Z_{A \cap C} + Z_{B \cap C} + Z_A + Z_B + Z_C$ of $Y$, where
\begin{center}
	\begin{tabular}{r l}\label{tab:Z}
		$Z_{A \cap B \cap C}$ & is general of type $(3,2^{s(n-6)})$ on $A \cap B \cap C \cong \bbP^1 \times \bbP^{n-6}$; \\
		$Z_{A \cap B}$ & is general of type $(2^{2n-7})$ on $A \cap B \cong \bbP^1 \times \bbP^{n-4}$, outside $C$; \\
		$Z_{A \cap C}$ & is general of type $(2^{2n-7})$ on $A \cap C  \cong \bbP^1 \times \bbP^{n-4}$, outside $B$;\\
		$Z_{B \cap C}$ & is general of type $(2^{2n-7})$ on $B \cap C  \cong \bbP^1 \times \bbP^{n-4}$, outside $A$;\\
		$Z_A$ & is general of type $(2^4)$ on $A\cong\bbP^1 \times \bbP^{n-2}$, outside $B \cup C$; \\
		$Z_B$ & is general of type $(2^4)$ on $B\cong\bbP^1 \times \bbP^{n-2}$, outside $A \cup C$; \\
		$Z_C$ & is general of type $(2^4)$ on $C\cong\bbP^1 \times \bbP^{n-2}$, outside $A \cup B$.
	\end{tabular}
\end{center}	
Then it is enough to prove that $\calL_{1 \times n}^{2,3}(A+B+Z)=0$. Consider the exact sequence\[
0\to
\LL_{1\times n}^{2,3}(A+B+C+Z)
\to
\LL_{1\times n}^{2,3}(A+B+Z)
\to
\LL_{1\times (n-2)}^{2,3}(A\cap C+B\cap C+\Tr_CZ).
\]
\begin{itemize}
\item \textit{Trace}. By induction hypothesis, $\LL_{1\times (n-2)}^{2,3}(A\cap C+B\cap C+\Tr_CZ)=0$.
\item \textit{Residue}. The system $\LL_{1\times n}^{2,3}(A+B+C+Z)$ is zero by Lemma \ref{lem:1xn_(3,2s)_4thStep}.\qedhere
\end{itemize}
\end{proof}
\end{lemma}

\begin{lemma}
\label{lem:1xn_(3,2s)_4thStep}
Let $n\ge 8$. Let $A, B,C\cong\p^1\times\p^{n-2}$ be subvarieties of $\p^1\times\p^n$, each defined by two general forms of bidegree $(0,1)$. Let $Z = Z_{A\cap B \cap C} + Z_{A \cap B} + Z_{A \cap C} + Z_{B \cap C} + Z_A + Z_B + Z_C$ be a scheme of fat points as in the proof of Lemma \ref{tab:Z}. Then $\LL_{1\times n}^{2,3}(A+B+C+Z)=0$.
\begin{proof}
We proceed by induction on $n$. A software computation shows that the statement holds for $n=8$ and $n=9$; see Appendix \ref{appendix: M2}. Assume that $n\ge 10$. Let $E \cong\p^1\times\p^{n-2}$ be another subvariety defined by two general forms of bidegree $(0,1)$. Let $W = W_{A \cap B \cap C \cap E} + W_{A \cap B \cap E} + W_{A\cap C \cap E} + W_{B \cap C \cap E} + W_{A \cap E} + W_{B \cap E} + W_{C \cap E}$ be a specialization of $Z$ such that
\begin{center}
	\begin{tabular}{r l}
		$W_{A \cap B \cap C \cap E}$ & is general of type $(3,2^{s(n-8)})$ on $A \cap B \cap C \cap E \cong \bbP^1 \times \bbP^{n-8}$; \\
		$W_{A \cap B \cap E}$ & is general of type $(2^{2n-11})$ on $A \cap B \cap E \cong \bbP^1 \times \bbP^{n-6}$; \\
		$W_{A \cap C \cap E}$ & is general of type $(2^{2n-11})$ on $A \cap C \cap E\cong \bbP^1 \times \bbP^{n-6}$; \\
		$W_{B \cap C \cap E}$ & is general of type $(2^{2n-11})$ on $B \cap C \cap E \cong \bbP^1 \times \bbP^{n-6}$; \\
		$W_{A \cap E}$ & is general of type $(2^4)$ on $A\cong\bbP^1 \times \bbP^{n-4}$; \\
		$W_{B \cap E}$ & is general of type $(2^4)$ on $B\cong\bbP^1 \times \bbP^{n-4}$; \\
		$W_{C \cap E}$ & is general of type $(2^4)$ on $C\cong\bbP^1 \times \bbP^{n-4}$; \\
		$W_{A\cap B \cap C}$ & is general of type $(2^{2n-11})$ on $A \cap B \cap C\cong\bbP^1 \times \bbP^{n-6}$; \\
		$W_{A \cap B}$ & is general of type $(2^{4})$ on $A \cap B\cong\bbP^1 \times \bbP^{n-4}$; \\
		$W_{A \cap C}$ & is general of type $(2^{4})$ on $A \cap C \cong\bbP^1 \times \bbP^{n-4}$; \\
		$W_{B \cap C}$ & is general of type $(2^{4})$ on $B \cap C \cong\bbP^1 \times \bbP^{n-4}$.
	\end{tabular}
\end{center}
Then it is enough to prove that $\calL_{1 \times n}^{2,3}(A+B+C+W)=0$. Consider the exact sequence\[
0\to
\LL_{1\times n}^{2,3}(A+B+C+E+W)\to
\LL_{1\times n}^{2,3}(A+B+C+W)\to
\LL_{1\times (n-2)}^{2,3}((A+B+C)\cap E+\Tr_EW).
\]
\begin{itemize}
\item \textit{Trace}. By induction hypothesis, $\LL_{1\times (n-2)}^{2,3}((A+B+C)\cap E+\Tr_EW)=0$.
\item \textit{Residue}. The system $\LL_{1\times n}^{2,3}(A+B+C+E+W)$ is a subsystem of $\LL_{1\times n}^{2,3}(A+B+C+E)$ which is zero because the ideal of $A+B+C+E$ is generated in bidegree $(0,4)$. \qedhere
\end{itemize}
\end{proof}\end{lemma}

\begin{example}[Lemma \ref{lem:1xn_(3,2s)}, Lemma \ref{lem:1xn_(3,2s)_2ndStep}, Lemma \ref{lem:1xn_(3,2s)_3rdStep}, Lemma \ref{lem:1xn_(3,2s)_4thStep} for $n = 10$]\label{example6}
    We show that $\calL_{1\times 10}^{2,3}(3,2^{65}) = 0$. Let $A \cong \bbP^1 \times \bbP^8$ be a subvariety defined by two general forms of bidegree $(0,1)$. Let $X = X_A + X_\circ$ be a scheme of type $(3,2^{65})$ where
    \begin{center}
        \begin{tabular}{r l}
         $X_A$ & is of type $(3,2^{44})$ with general support on $A$, indeed $s(8) = 44$;\\
         $X_\circ$ & is general of type $(2^{21})$.
        \end{tabular}
    \end{center}
    Let $B, C$ be other two general subvarieties of $\bbP^1 \times \bbP^{10}$, each one defined by two general forms of bidegree $(0,1)$. We consider a series of specializations of the scheme $X$: we describe them as union of distinct components, each one with general support in the space indicated by the diagrams in Figure~\ref{fig:specialization}.
    
    \begin{figure}[ht]
        \centering
        \includegraphics[scale=0.5]{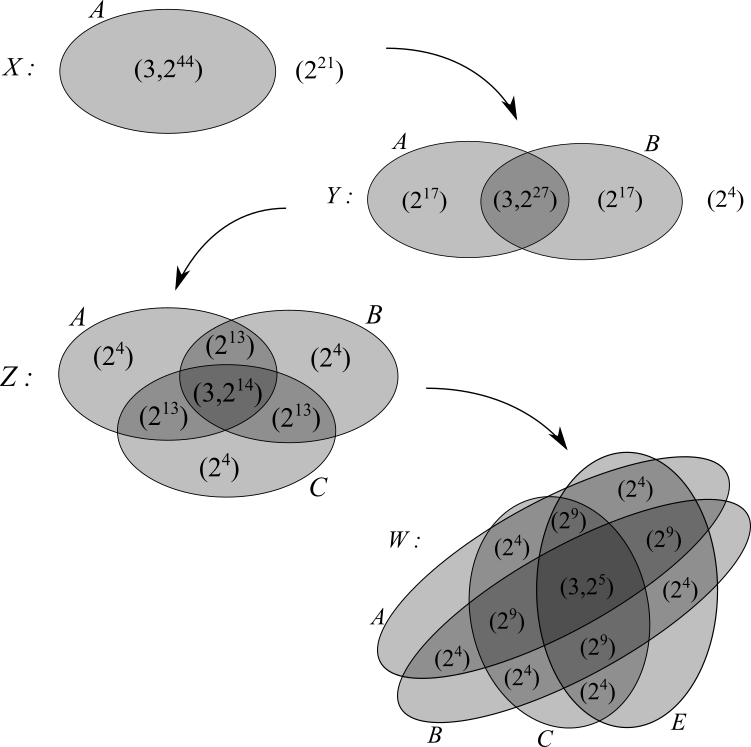}
        \caption{The specialization used in Example \ref{example6}}
        \label{fig:specialization}
    \end{figure}
    
    By using a series of Castelnuovo exact \Fra{sequences}, we obtain the following chain of inequalities
    \begin{align*}
        &\dim \calL_{1\times 10}^{2,3}(X) \leq \dim\calL_{1\times 10}^{2,3}(A + X) + \dim\calL_{1\times 8}^{2,3}(\Tr_A(X)) \\
        & \leq \left(\dim\calL_{1\times 10}^{2,3}(A + B + Y) + \dim\calL_{1\times 8}^{2,3}(A \cap B + \Tr_B(Y))\right) \\
        & \qquad  + \dim\calL_{1\times 8}^{2,3}(\Tr_A(X)) \\
        & \leq \left(\dim\calL_{1\times 10}^{2,3}(A + B + C + Z) + \dim\calL_{1\times 8}^{2,3}(A \cap C + B \cap C + \Tr_C(Z))\right) \\ &  \qquad   + \dim\calL_{1\times 8}^{2,3}(A \cap B + \Tr_B(Y))  + \dim\calL_{1\times 8}^{2,3}(\Tr_A(X)) \\
        & \leq \left(\dim\calL_{1\times 10}^{2,3}(A + B + C + E + W) + \dim\calL_{1\times 8}^{2,3}(A \cap E + B \cap E + C \cap E + \Tr_E(W))\right) \\ &  \qquad   + \dim\calL_{1\times 8}^{2,3}(A \cap C + B \cap C + \Tr_C(Z)) + \dim\calL_{1\times 8}^{2,3}(A \cap B + \Tr_B(Y))  + \dim\calL_{1\times 8}^{2,3}(\Tr_A(X)).
    \end{align*}
    In each step of the latter chain of inequalities, we may assume that the linear systems obtained from the traces on $\bbP^1 \times \bbP^8$ are known to be equal to zero by induction. Hence we are left with proving that $\calL_{1\times 10}^{2,3}(A + B + C + E + W) = 0$. This follows for the straightfoward observation that the ideal of $A \cup B \cup C \cup E$ is generated by forms in bidegree $(0,4)$, therefore $\calL_{1\times 10}^{2,3}(A + B + C + E) = 0.$
\end{example}

\begin{lemma}
	\label{lem: 2 x n, bigrado (3,2)} 
	If $n\ge 2$, then $\LL_{2\times n}^{3,2}(3,2^{k_*(3,3;2,n)-k_*(3,3;2,n-1)})$ is regular.
	\begin{proof}
In order to shorten the notation, we set
\[b(n):=k_*(3,3;2,n)-k_*(3,3;2,n-1).\]
We proceed by induction on $n$. A software computation shows that $\LL_{2\times 2}^{3,2}(3,2^{b(2)})$ is regular; see Appendix \ref{appendix: M2}. Assume that $n\ge 3$. Let $D\subset\p^2\times\p^n$ be a general divisor of bidegree $(0,1)$ and let $X$ be a scheme of fat points of type $(3,2^{b(n)})$ such that $\Tr_D(X)$ is general of type $(3,2^{b(n-1)})$ on $D \cong \bbP^2 \times\bbP^{n-1}$. As explained in Section \ref{sec:inductive}, it is enough to prove that residue and trace of $\calL_{2\times n}^{3,2}(X)$ with respect to $D$ are regular.
\begin{itemize}
	\item \textit{Trace.} The trace $\Tr_D(X)$ is a general scheme of type $(3,2^{b(n-1)})$ on $D \cong \bbP^2 \times \bbP^{n-1}$ and $\calL_{2 \times (n-1)}^{3,2}(3,2^{b(n-1)})$ is regular by induction.
	\item \textit{Residue.} The residue $\Res_D(X)$ is a scheme of fat points of type $(2^{1+b(n)-b(n-1)},1^{b(n-1)})$ such that $\Res_D(X) \cap D$ is general of type $(2,1^{b(n-1)})$ on $D$ and $\calL_{2\times n}^{3,1}(\Res_D(X))$ is regular by Lemma~\ref{lem: primo di due}. \qedhere
\end{itemize}
	\end{proof}
\end{lemma}

\begin{lemma}\label{lem: primo di due}
\Ale{Let $n\ge 3$}. Let $D \subset \bbP^2 \times \bbP^n$ be a general divisor of bidegree $(0,1)$ and let $Y$ be a scheme of fat points of type $(2^{1+b(n)-b(n-1)},1^{b(n-1)})$ such that $Y \cap D$ is a general scheme of type $(2,1^{b(n-1)})$ on $D$. Then $\LL_{2\times n}^{3,1}(Y)$ is regular.
\begin{proof}
Set $v(n):=\vdim\LL_{2\times n}^{3,1}(Y)$ and let $P \subset\p^2\times\p^n$ be a set of $v(n)$ general simple points. In order to prove the statement, it suffices to show that $\LL_{2\times n}^{3,1}(Y+P)=0$. We argue by a triple-step induction on $n$. The cases $n\in\{3,4,5\}$ are checked by an explicit software computation; see Appendix \ref{appendix: M2}. Now assume that $n\ge 6$. By  Lemma \ref{conto:b-b_modulo3},
\begin{align*}
b(n)-b(n-1)&=b(n-3)-b(n-4),
\end{align*}
so it makes sense to specialize some of the base points to a subvariety of codimension 3.
Let $A\subset\p^2\times\p^n$ be the subvariety defined by the vanishing of 3 general bidegree $(0,1)$ forms. We consider the specializations $Z = Z_{D \cap A} + Z_A + Z_D$ of $Y$ and and $Q = Q_A + Q_\circ$ of $P$, where
\begin{center}
	\begin{tabular}{r l}
		$Z_{A}$ & is a scheme of type $(2^{b(n-3)-b(n-4)})$ on $A \cong \bbP^2 \times \bbP^{n-3}$, outside $D$; \\
		$Z_{D \cap A}$ & is a scheme of type $(2,1^{b(n-4)})$ general points on $A \cap D$; \\
		$Z_D$ & is a set of $b(n-1)-b(n-4)$ general points on $D$, outside $A$; \\
		$Q_A$ & is a set of $v(n-3)$ general points on $A$, outside $D$; \\
		$Q_\circ$ & is a set of $v(n)-v(n-3)$ general points outside $A$.
	\end{tabular}
\end{center}
Such specialization is possible by Lemma \ref{conto:b-b_modulo3} and Lemma \ref{conto:v-v_modulo3}. Now we only need to prove that $\calL_{2\times n}^{3,1}(Z+Q)=0$. Consider the exact sequence
\[
0\to\LL_{2\times n}^{3,1}(A+Z + Q)
\to
\LL_{2\times n}^{3,1}(Z + Q)\to
\LL_{2\times (n-3)}^{3,1}(\Tr_A(Z+Q)).
\]
Then it is enough to prove the left-most and the right-most linear systems are zero.
\begin{itemize}
	\item The trace $\Tr_A(Z + Q) = Z_A + Z_{D \cap A} + Q_A$: the linear system $\LL_{2\times (n-3)}^{3,1}(Z_A + Z_{D\cap A})$ is regular by induction and its dimension is exactly the cardinality of $Q_A$; hence $\LL_{2\times (n-3)}^{3,1}(\Tr_A(Z+Q))=0$.
	\item By Lemma \ref{lem: secondo di due}, $\LL_{2\times n}^{3,1}(A+Z + Q)=0$.\qedhere
\end{itemize}
\end{proof}
\end{lemma}

\begin{lemma}
\label{lem: secondo di due} 
Let $n\ge 6$. Let $D\subset\p^2\times\p^n$ be a general divisor of bidegree $(0,1)$ and let $A\subset\p^2\times\p^n$ be the subvariety defined by $3$ general bidegree $(0,1)$ forms. Let $Z$ and $Q$ be schemes of fat points as in the proof of Lemma \ref{lem: primo di due}. Then  $\dim \LL_{2\times n}^{3,1}(A+ Z + Q)=0$.
\begin{proof}
We proceed by a triple-step induction on $n$. We check the cases $n\in\{6,7,8\}$ by an explicit software computation; see Appendix \ref{appendix: M2}. Assume that $n\ge 9$. Let $B\subset\p^2\times\p^n$ be a subvariety defined by 3 general bidegree $(0,1)$ forms. We consider a specialization $W$ of $Z$ and $R$ of $Q$ such that $W = W_{A\cap B} + W_{A \cap B \cap D} + W_{A \cap D} + W_D$ and $R = R_{A\cap B} + R_A + R_\circ$ where
\begin{center}
	\begin{tabular}{r l}
		$W_{A \cap B}$ & is a scheme of type $(2^{b(n-6)-b(n-7)})$ on $A \cap B \cong \bbP^2 \times \bbP^{n-6}$, outside $D$; \\
		$W_{A\cap B \cap D}$ & is a scheme of type $(2,1^{b(n-7)})$ on $A \cap B \cap D \cong \bbP^2 \times \bbP^{n-7}$; \\
		$W_{A \cap D}$ & is a set of $10$ points on $A \cap D$, outside $B$; \\
		$W_D$ & is a set of $10$ points on $D$, outside $A \cup B$; \\
		$R_{A \cap B}$ & is a set of $v(n-6)$ points on $A \cap B$, outside $D$; \\
		$R_A$ & is a set of $v(n-3) - v(n-6)$ points on $A$, outside $B \cup D$; \\
		$R_\circ$ & is a set of $v(n) - v(n-3)$ points outside $A \cup B \cup D$.
	\end{tabular}
\end{center}
Consider
\[
0\to
\LL_{2\times n}^{3,1}(A+ B+W+R)
\to
\LL_{2\times n}^{3,1}(A+W+R)
\to
\LL_{2\times (n-3)}^{3,1}((A\cap B)+\Tr_B(W+R)).
\]
It is enough to show that the left-most and right-most linear system are zero.
\begin{itemize}
	\item By Lemma \ref{conto:b-b_modulo3} and Lemma \ref{conto:v-v_modulo3}, the linear system $\LL_{2\times (n-3)}^{3,1}((A\cap B)+\Tr_B(W+R))$ is zero by induction.
	\item The linear system $\LL_{2\times n}^{3,1}(A+ B+W+R)$ is contained in the linear system $ \LL_{2\times n}^{3,1}(A+ B)$ which is zero because the ideal of $A \cup B$ is generated in bidegree $(0,2)$.\qedhere
\end{itemize}
\end{proof}
\end{lemma}

\Ale{
    \begin{example}[Lemma \ref{lem: 2 x n, bigrado (3,2)}, Lemma \ref{lem: primo di due}, Lemma \ref{lem: secondo di due} for $n = 8$]\label{example:2xn_(3,2)}
        We show that $\calL_{2\times 8}^{3,2}(3,2^{29})$ is regular, i.e., $\dim\calL_{2\times 8}^{3,2}(3,2^{29}) = 65$. Let $D$ be a general divisor of bidegree $(0,1)$ and let $X$ be a scheme of type $(3,2^{29})$ such that $X \cap D$ is general of type $(3,2^{26})$. By Castelnuovo exact sequence
        \begin{equation}\label{example:2xn_(3,2)_eq1}
            65 \leq \dim \calL_{2\times 8}^{3,2}(X) \leq \dim\calL_{2\times 8}^{3,1}(\Res_D(X)) + \calL_{2\times 7}^{3,2}(\Tr_D(X)).
        \end{equation}
We assume to know that $\calL_{2\times 7}^{3,2}(\Tr_D(X))$ is regular, i.e., $\dim\calL_{2\times 7}^{3,2}(\Tr_D(X)) = 45.$ The residue $Y = \Res_D(X)$ is a scheme of type $(2^4,1^{26})$ such that $Y \cap D$ is general of type $(2,1^{26})$. Since $\vdim\calL_{2\times 8}^{3,1}(Y) = 20$, we consider a set of twenty general points $P$. We want to prove that $\dim\calL_{2\times 8}^{3,1}(Y+P) = 0$. Let $A, B$ be two subvarieties of codimension $3$ defined by general forms of bidegree $(0,1)$. We consider a series of specializations of the scheme $Y+P$: we describe them as union of distinct components, each one with general support in the space indicated by the diagrams in Figure \ref{fig:specialization2}.
        
        \begin{figure}[ht]
            \centering
            \includegraphics[scale=0.5]{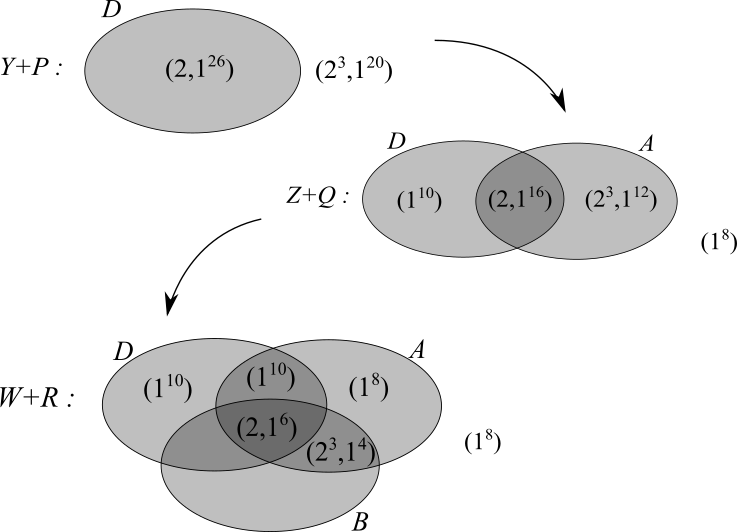}
            \caption{The specialization used in Example \ref{example:2xn_(3,2)}}
            \label{fig:specialization2}
        \end{figure}
        
        By using a series of Castelnuovo exact sequences, we obtain the following chain of inequalities
        \begin{align*}
  \dim& \calL_{2\times 8}^{3,1}(Y+P) \leq \dim\calL_{2\times 8}^{3,1}(Z+Q) \leq \dim\calL_{2\times 8}^{3,1}(A+Z+Q) + \dim\calL_{2\times 5}^{3,1}(\Tr_A(Z+Q)) \\
  & \leq \left(\dim\calL_{2\times 8}^{3,1}(A+B+W+R) + \dim\calL_{2\times 5}^{3,1}((A\cap B)+\Tr_B(W+R)\right)
  + \dim\calL_{2\times 5}^{3,1}(\Tr_A(Z+Q)).
        \end{align*}
        By inducion, we assume that the traces on $\bbP^2\times\bbP^5$ are zero at each step of the chain of inequalities. Hence we are left with checking that $\dim\calL_{2\times 8}^{3,1}(A+B+W+R) = 0$ which holds by the straightforward observation that the ideal of $A\cup B$ is generate in bidegree $(0,2)$ and therefore $\dim\calL_{2\times 8}^{3,1}(A+B) = 0$.
    \end{example}
}

\section{\texorpdfstring{$\SV_{m\times n}^{3,4}$}{SVmn34} is not defective}\label{sec: bidegree (3,4)}
Now that we have established Proposition \ref{pro: bigrado 33 come punti doppi}, it will be easier to show that $\SV_{m\times n}^{3,4}$ is also non-defective. As in Section \ref{sec: linear systems}, we can phrase this statement in terms of linear systems.
\begin{proposition}\label{pro: bigrado 34 come punti doppi} If $m$ and $n$ are positive integers, then $\LL_{m\times n}^{3,4}(2^r)$ is nonspecial for every $r\in\N$.
\end{proposition}
\begin{proof}
    For $n = 1$, the statement is known by \cite[Theorem 3.1]{BalBerCat12}. For $n \geq 2$ we apply Theorem~\ref{thm: quello che usiamo}: condition (1) is Proposition \ref{lemma:3+2kup_mxn}, condition (2) is Proposition \ref{pro: col quartuplo + vuoto}, condition (3) is checked directly because a single $j$-fat point always impose independent conditions on $\calL_{m\times n}^{c,d}$ for $c,d \geq j-1$ and condition (4) is trivial.
\end{proof}
Once more, we focus on proving that $\SV_{m\times n}^{3,4}$ satisfies the first two hypotheses of Theorem \ref{thm: quello che usiamo}.
\begin{proposition}\label{lemma:3+2kup_mxn}
	If $m$ and $n$ are positive integers, then $\calL_{m\times n}^{3,4}(3,2^{k^*(3,4;m,n)})$ is regular.
\end{proposition}
\begin{proof}
	We proceed by induction on $n$. The case $n = 1$ is proven in Lemma \ref{lemma:3+2kup_mx1}. Let $n \geq 2$. Let $D$ be a general divisor of bidegree $(0,1)$ in $\bbP^m \times \bbP^n$. Consider a scheme of fat points $X$ of type $(3,2^{k^*(3,4;m,n)})$ such that $X \cap D$ is general of type $(3,2^{k^*(3,4;m,n-1)})$ and the Castelnuovo exact sequence \eqref{exact_sequence_bihom}.
	\begin{itemize}
		\item \textit{Trace.} The trace $\Tr_D(X)$ is a general scheme of fat points of type $(3,2^{k^*(3,4;m,n-1)})$ on $\bbP^m \times \bbP^{n-1}$ and $\calL_{m\times (n-1)}^{3,4}(3,2^{k^*(3,4;m,n-1)})$ is regular by induction.
		\item \textit{Residue.} To ease the notation, we set
		\[
		u(m,n) := k^*(3,4;m,n)-k^*(3,4;m,n-1).
		\]
The residue $\Res_D(X)$ is a scheme of type $(2^{1+u(m,n)},1^{k^*(3,4;m,n-1)})$ such that $\Res_D(X) \cap D$ is general of type $(2,1^{k^*(3,4;m,n-1)})$ on $D$. By Lemma \ref{appendix:few2_3,3}, $\vdim\calL_{m\times n}^{3,3}(\Res_D(X))\ge 0$. By Proposition \ref{pro: bigrado 33 come punti doppi}, the linear system $\calL_{m\times n}^{3,3}(2^{1+u(m,n)})$ is regular. We need to show that the extra simple points on $D$ impose independent conditions on $\calL_{m\times n}^{3,3}(2^{1+u(m,n)})$. Thanks to Lemma~\ref{lemma:extra_points}, it suffices to show that $\calL_{m\times n}^{3,2}(1,2^{u(m,n)})=0$
. In Lemma \ref{lemma:enough2fat_3,2} we prove the stronger statement $\LL_{m\times n}^{3,2}(2^{u(m,n)})=0$.\qedhere	\end{itemize}
\end{proof}

\Fra{
\begin{lemma}\label{lemma:enough2fat_3,2}
If $m$ and $n$ are positive integers, then 
$\calL_{m\times n}^{3,2}(2^{u(m,n)}) = 0$.
\end{lemma}
\begin{proof}
The systems $\LL_{1\times n}^{3,2}(2^{h(1,n)})$ and $\LL_{m\times 1}^{3,2}(2^{h(m,1)})$ are expected
to be zero by Lemma~\ref{conto: 3, 2 con punti doppi atteso vuoto}, and they are indeed zero by \cite[Theorem 3.1]{BalBerCat12}, so we assume that $m\ge 2$ and $n\ge 2$ and we argue by induction on $n$. Let $D$ be a general divisor of bidegree $(0,1)$ in $\bbP^m\times \bbP^n$. Note that $u(m,n-1)\le u(m,n)$ by Lemma \ref{appendix:applyBCC_1+2u_mxn}\eqref{bullet: u crescente}, so we consider a scheme $X$ of fat points of type 
$(2^{u(m,n)})$ such that $X \cap D$ is general of type 
$(2^{u(m,n-1)})$ on $D \cong \bbP^m\times\bbP^{n-1}$. 
 We consider residue and trace with respect to $D$. 
	\begin{itemize}
		\item \textit{Trace.} The trace $\Tr_D(X)$ is a general scheme of type 
 $(2^{u(m,n-1)})$ on $D$ and the linear system 
 $\calL_{m\times (n-1)}^{3,2}(2^{u(m,n-1)})$ is zero by induction hypothesis.
\item \textit{Residue.} The residue $\Res_D(X)$ is a scheme of fat points of type $(2^{u(m,n)-u(m,n-1)},1^{u(m,n-1)})$ where $\Res_D(X) \cap D$ is a set of $u(m,n-1)$ simple points on $D$. By Lemma \ref{appendix:applyBCC_1+2u_mxn}(1),
\[	u(m,n) - u(m,n-1) \geq \left\lceil \frac{n+1}{m+n+1} \binom{m+3}{3}\right\rceil + n,
		\]
therefore $\calL_{m\times n}^{3,1}(2^{u(m,n)-u(m,n-1)})$ is zero by \cite[Corollary 2.2]{BCC}. A fortiori, $\calL_{m\times n}^{3,1}(\Res_D(X))$ is zero as well. 
\qedhere
	\end{itemize}
\end{proof}

\begin{lemma}\label{lemma:3+2kup_mx1}
	If $m$ is a positive integer, then $\calL_{m\times 1}^{3,4}(3,2^{k^*(3,4;m,1)})$ is regular.
\end{lemma}
\begin{proof}
	We proceed by induction on $m$. The case $m = 1$ is \Fra{checked directly. Let $m \geq 2$ be an integer}. Let $D$ be a general divisor of bidegree $(1,0)$ in $\bbP^m\times\bbP^1$ and let $X$ be a scheme of fat points of type $(3,2^{k^*(3,4;m,1)})$ such that $X \cap D$ is general of type $(3,2^{k^*(3,4;m-1,1})$ on $D$. Consider the Castelnuovo exact sequence~\eqref{exact_sequence_bihom}.
	\begin{itemize}
		\item \textit{Trace.} The trace $\Tr_D(X)$ is general of type $(3,2^{k^*(3,4;m-1,1)})$ on $D \cong \bbP^{m-1}\times \bbP^n$ and the linear system $\calL^{3,4}_{(m-1)\times 1}(3,2^{k^*(3,4;m-1,1)})$ is regular by induction.
\item \textit{Residue.} $\Res_D(X)$ is of type $(2^{1+k^*(3,4;m,1)-k^*(3,4;m-1,1)},1^{k^*(3,4;m-1,1)})$ where $\Res_D(X) \cap D$ is general of type $(2,1^{k^*(3,4;m-1,1)})$. The virtual dimension of $\calL_{m\times 1}^{2,4}(\Res_D(X))$ is non-negative by Lemma~\ref{appendix:3,4_mx1}\eqref{item: 24m1}. The linear system $\calL_{m\times 1}^{2,4}(2^{1+k^*(3,4;m,1)-k^*(3,4;m-1,1)})$ is regular by \cite[Theorem 3.1]{BalBerCat12}: indeed $1+k^*(3,4;m,1)-k^*(3,4;m-1,1) \leq 2(m+1)$ by Lemma \ref{appendix:fewPoints_2,4}. Now we have to prove that the $k^*(3,4;m-1,1)$ simple points on $D$ impose independent conditions on $\calL_{m\times 1}^{2,4}(2^{1+k^*(3,4;m,1)-k^*(3,4;m-1,1)})$. Thanks to Lemma~\ref{lemma:extra_points}, we only have to prove that $\calL_{m\times 1}^{1,4}(1,2^{k^*(3,4;m,1)-k^*(3,4;m-1,1)})=0$.
By Lemma \ref{appendix:3,4_mx1}\eqref{item: 14m1}, $\calL_{m\times 1}^{1,4}(1,2^{k^*(3,4;m,1)-k^*(3,4;m-1,1)})$ is expected to be zero, and it is indeed zero by \cite[Theorem 3.1]{BalBerCat12}.
\qedhere
	\end{itemize}
\end{proof}
}

Now that we proved that $\SV_{m\times n}^{3,4}$ satisfies the first hypothesis of Theorem \ref{thm: quello che usiamo}, we move to the second one.
\begin{proposition}
	\label{pro: col quartuplo + vuoto} If $m\ge 1$ and $n\ge 2$, then $\LL_{m\times n}^{3,4}(4,2^{k_*(3,4;m,n)})=0$
	\begin{proof}
We proceed by induction on $n$. The case $n=2$ is solved in Lemma \ref{lem: col quartuplo + vuoto caso iniziale}. Let $n \geq 3$. Let $D$ be a general divisor of bidegree $(0,1)$ in $\bbP^m \times \bbP^n$. Let $X$ be a scheme of type $(4,2^{k_*(3,4;m,n)})$ such that $X \cap D$ is general of type $(4,2^{k_*(3,4;m,n-1)})$ on $D$. By the Castelnuovo exact sequence, it is enough to prove that trace and residue are zero.
		\begin{itemize}
			\item \textit{Trace.} The trace linear system is $\calL_{m\times (n-1)}^{3,4}(4,2^{k_*(3,4;m,n-1)})$ which is zero by induction.
\item \textit{Residue.} The residue $\Res_D(X)$ is of type $(3,2^{k_*(3,4;m,n)-k_*(3,4;m,n-1)}, 1^{k_*(3,4;m,n-1)})$ such that $\Res_D(X) \cap D$ is general of type $(3,1^{k_*(3,4;m,n-1)})$. The linear system $\calL_{m\times n}^{3,3}(\Res_D(X))$ is expected to be zero by Lemma \ref{conto: 3,3 triplodoppiesemplici atteso vuoto}. We show that $\calL_{m\times n}^{3,3}(3,2^{k_*(3,4;m,n)-k_*(3,4;m,n-1)})$ is regular: if $(m,n) \neq (1,3)$ it follows from Proposition \ref{pro: triplo e doppi}, since
			\[
			k_*(3,4;m,n)-k_*(3,4;m,n-1)\leq k^*(3,3;m,n)
			\]
by Lemma \ref{appendix:kdown34vskdown33}, while for $(m,n) = (1,3)$ it is checked directly by software computation. Now we prove that the extra $k_*(3,4;m,n-1)$ simple points on $D$ are enough to annihilate the linear system. Thanks to Lemma \ref{lemma:extra_points}, it suffices to show that $\calL_{m\times n}^{3,2}(2^{1+k_*(3,4;m,n)-k_*(3,4;m,n-1)}) = 0$. 
By construction, $1+k_*(3,4;m,n)-k_*(3,4;m,n-1)\ge u(m,n)$, so we conclude by Lemma~\ref{lemma:enough2fat_3,2}.
\qedhere
		\end{itemize}
	\end{proof}
\end{proposition}

\begin{lemma}
	\label{lem: col quartuplo + vuoto caso iniziale} If $m$ is a positive integer, then $\LL_{m\times 2}^{3,4}(4,2^{k_*(3,4;m,2)})=0$.
	\begin{proof}
		We work by induction on $m$. We check that $\LL_{1\times 2}^{3,4}(4,2^{k_*(3,4;1,2)})=0$ by a software computation and we assume that $m\ge 2$. Let $D\subset\p^m\times\p^2$ be a divisor of bidegree $(1,0)$. Let $X\subset\p^m\times\p^2$ be a scheme of fat points of type $(4,2^{k_*(3,4;m,2)})$ such that $\Tr_D(X)$ is a general scheme of fat points of type $(4,2^{k_*(3,4;m-1,2)})$ on $D$. Specialize $\LL_{m\times 2}^{3,4}(4,2^{k_*(3,4;m,2)})$ to $\calL_{m\times 2}^{3,4}(X)$ and consider the Castelnuovo exact sequence \eqref{exact_sequence_bihom}.
		\begin{itemize}
			\item \textit{Trace.} The trace linear system is $\LL_{(m-1)\times 2}^{3,4}(4,2^{k_*(3,4;m-1,2)})$ and it is zero by induction hypothesis.
			\item \textit{Residue.} $\Res_D(X)$ is a scheme of fat points of type $(3,2^{k_*(3,4;m,2)-k_*(3,4;m-1,2)},1^{k_*(3,4;m-1,2)})$, where $\Res_D(X)\cap D$ is general of type $(3,1^{k_*(3,4;m-1,2)})$ on $D$. By Lemma~\ref{conto: m x 2, 3,4 con quartplo atteso vuoto}, $\calL_{m\times 2}^{2,4}(\Res_D(X))$ is expected to be zero. The system $\calL_{m\times 2}^{2,4}(3,2^{k_*(3,4;m,2)-k_*(3,4;m-1,2)})$ is regular by Lemma~\ref{lem:  2,4 triplo e doppi è regolare}. Now we only have to show that the $k_*(3,4;m-1,2)$ simple points on $D$ impose enough linear conditions on $\calL_{m\times 2}^{2,4}(3,2^{k_*(3,4;m,2)-k_*(3,4;m-1,2)})$ to make it zero. By Lemma~\ref{lemma:extra_points}, it is enough to prove that $\calL_{m\times 2}^{1,4}(2,2^{k_*(3,4;m,2)-k_*(3,4;m-1,2)})=0$. Thanks to \cite[Corollary 2.2]{BCC}, we only need to check that
			\[
			    1+k_*(3,4;m,2)-k_*(3,4;m-1,2)\ge\ru{\frac{15(m+1)}{m+3}}+m,
			\]
			and this is accomplished in Lemma \ref{conto: differenza dei kdown permette bcc}.\qedhere
		\end{itemize}
	\end{proof}
\end{lemma}

\begin{lemma}
	\label{lem:  2,4 triplo e doppi è regolare} If $m\ge 2$, then $\calL_{m\times 2}^{2,4}(3,2^{k_*(3,4;m,2)-k_*(3,4;m-1,2)})$ is regular.
	\begin{proof}
		In order to shorten the notation, set
		\[j(m):=k_*(3,4;m,2)-k_*(3,4;m-1,2).\]
		We work by induction on $m$. We check that $\calL_{2\times 2}^{2,4}(3,2^{j(2)})$ is regular by a software computation and we assume that $m\ge 3$. Let $D\subset\p^m\times\p^2$ be a divisor of bidegree $(1,0)$. Let $X\subset\p^m\times\p^2$ be a scheme of fat points of type $(3,2^{j(m)})$ such that $\Tr_D(X)$ is a general scheme of fat points of type $(3,2^{j(m-1)})$ on $D$. This is possible because $j(m)-j(m-1)=5$ by Lemma \ref{conto:differenza dei j fa 5}. Specialize $\calL_{m\times 2}^{2,4}(3,2^{j(m)})$ to $\calL_{m\times 2}^{2,4}(X)$ and consider the Castelnuovo exact sequence
		\[
		0 \rightarrow
		\calL_{m\times 2}^{1,4}(\Res_D(X)) \rightarrow
		\calL_{m\times 2}^{2,4}(X)
		\rightarrow \calL_{(m-1)\times 2}^{2,4}(\Tr_D(X)).
		\]
		\begin{itemize}
\item \textit{Trace.} The trace linear system is $\calL_{(m-1)\times 2}^{2,4}(3,2^{j(m-1)})$ and it is regular by induction hypothesis.
			\item \textit{Residue.} $\Res_D(X)$ is a scheme of fat points of type $(2,2^{5},1^{j(m-1)})$, where $\Res_D(X)\cap D$ is general of type $(2,1^{j(m-1)})$ on $D$. 
			By Lemma \ref{conto: vdim 1,4 m x 2 positiva}, $\vdim\LL_{m\times 2}^{1,4}(2,2^{5},1^{j(m-1)})\ge 0$. The system $\LL_{m\times 2}^{1,4}(2^6)$ is regular by Lemma \ref{lem: really? Davvero dobbiamo dimostrare anche questo? Non possiamo lasciarlo come esercizio al lettore? O magari come esempio di come funziona la castelnuovo, nella sezione dei preliminari?}. Now we only have to show that the $j(m-1)$ simple points on $D$ impose independent conditions on $\LL_{m\times 2}^{1,4}(2^6)$. By Lemma \ref{lemma:extra_points}, it is enough to prove that $\LL_{m\times 2}^{0,4}(1,2^5)=0$. It is straightforward to check that $\LL_{m\times 2}^{0,4}(1,2^5)\cong\LL_{2}^{4}(1,2^5)=0$.\qedhere
		\end{itemize}
	\end{proof}
\end{lemma}

\begin{lemma}
	\label{lem: really? Davvero dobbiamo dimostrare anche questo? Non possiamo lasciarlo come esercizio al lettore? O magari come esempio di come funziona la castelnuovo, nella sezione dei preliminari?}
	If $m$ is a positive integer, then $\LL_{m\times 2}^{1,4}(2^6)$ is regular.
	\begin{proof}
		We proceed by induction on $m$. The case $m = 1$ follows from \cite{BalBerCat12}. Let $m \geq 2$ and let $X$ be a scheme of type $(2^6)$ with general support on a general divisor $D$ of bi-degree $(1,0)$. From the Castelnuovo exact sequence, we get
		\[
		    \dim \calL_{m\times 2}^{1,4}(X) \leq \dim \calL_{(m-1)\times 2}^{1,4}(\Tr_D(X)) + \dim \calL_{m\times 2}^{0,4}(\Res_D(X)) =9(m+1)-12.
		\]
		The latter equality follows by induction and since the projection of the support of $\Res_D(X)$ on the second factor is general. Since $\vdim\calL_{m\times 2}^{1,4}(2^6) = 9(m+1)-12$, we conclude.
\end{proof}
\end{lemma}

\section{\texorpdfstring{$\SV_{m\times n}^{4,4}$}{SVmn44} is not defective}\label{sec: bidegree (4,4)}
The last step to complete our proof of Theorem \ref{thm: main} is to prove that $\SV_{m\times n}^{4,4}$ is non-defective.
\begin{proposition}
    If $m$ and $n$ are positive integers, then $\calL_{m\times n}^{4,4}(2^r)$ is nonspecial for every $r \in \bbN$.
\end{proposition}
\begin{proof}
    For $m = 1$, the statement is known by \cite[Theorem 3.1]{BalBerCat12}. For $m \geq 2$ we apply Theorem~\ref{thm: quello che usiamo}: condition (1) is Proposition \ref{pro: 4,4 triplo e doppi regolare}, condition (2) is Proposition \ref{pro: 4,4 quartuplo e doppi vuoto}, condition (3) is checked directly because a single $j$-fat point always impose independent conditions on $\calL_{m\times n}^{c,d}$ for $c,d \geq j-1$ and condition (4) is trivial.
\end{proof}
\begin{proposition}
	\label{pro: 4,4 triplo e doppi regolare} If $m$ and $n$ are positive integers, then $\calL_{m\times n}^{4,4}(3,2^{k^*(4,4;m,n)})$ is regular.
	\begin{proof}
		We argue by induction on $m$.  The case $m=1$ is solved by Lemma \ref{lem: 4,4 triplo e doppi caso iniziale}. Assume that $m\ge 2$ and let $D\subset\p^m\times\p^n$ be a divisor of bidegree $(1,0)$. Let $X\subset\p^m\times\p^n$ be a scheme of fat points of type $(3,2^{k^*(4,4;m,n)})$ such that $\Tr_D(X)$ is a general scheme of fat points of type $(3,2^{k^*(4,4;m-1,n)})$ on $D$. Specialize $\calL_{m\times n}^{4,4}(3,2^{k^*(4,4;m,n)})$ to $\calL_{m\times n}^{4,4}(X)$ and consider the Castelnuovo exact sequence
		\[
		0\to
		\calL_{m\times n}^{3,4}(\Res_D(X))
		\to
		\calL_{m\times n}^{4,4}(X)
		\to
		\calL_{(m-1)\times n}^{4,4}(\Tr_D(X)).
		\]
		\begin{itemize}
			\item \textit{Trace.} The trace linear system is $\calL_{(m-1)\times n}^{4,4}(3,2^{k^*(4,4;m-1,n)})$ and it is regular by induction hypothesis.
			\item \textit{Residue.} $\Res_D(X)$ is a scheme of fat points of type $(2,2^{k^*(4,4;m,n)-k^*(4,4;m-1,n)},1^{k^*(4,4;m-1,n)})$ such that $\Res_D(X)\cap D$ is a scheme of fat points of type $(2,1^{k^*(4,4;m-1,n)})$ on $D$. By Lemma \ref{conto: 3,4, mxn doppisemplici atteso non vuoto}, $\vdim\calL_{m\times n}^{3,4}(\Res_D(X))\ge 0$. The linear system $\calL_{m\times n}^{3,4}(2,2^{k^*(4,4;m,n)-k^*(4,4;m-1,n)})$ is regular by Proposition \ref{pro: bigrado 34 come punti doppi}. We only have to prove that the $k^*(4,4;m-1,n)$ simple points on $D$ impose independent conditions on $\calL_{m\times n}^{3,4}(2,2^{k^*(4,4;m,n)-k^*(4,4;m-1,n)})$. Thanks to Lemma \ref{lemma:extra_points}, it is enough to prove that $\calL_{m\times n}^{2,4}(1,2^{k^*(4,4;m,n)-k^*(4,4;m-1,n)})=0$, and this is accomplished in Lemma \ref{lemma: 2,4 con tanti punti doppi}.\qedhere
		\end{itemize}
	\end{proof}
\end{proposition}

\begin{lemma}
	\label{lem: 4,4 triplo e doppi caso iniziale} If $n$ is a positive integer, then $\calL_{1\times n}^{4,4}(3,2^{k^*(4,4;1,n)})$ is regular.
	\begin{proof}
		We argue by induction on $n$. A software computation shows that $\calL_{1\times 1}^{4,4}(3,2^{k^*(4,4;1,1)})$ is regular. Assume that $n\ge 2$ and let $D\subset\p^1\times\p^n$ be a divisor of bidegree $(0,1)$. Let $X$ be a scheme of fat points of type $(3,2^{k^*(4,4;1,n)})$ such that $\Tr_D(X)$ is a general scheme of fat points of type $(3,2^{k^*(4,4;1,n-1)})$ on $D$. Specialize $\LL_{1\times n}^{4,4}(3,2^{k^*(4,4;1,n)})$ to $\LL_{1\times n}^{4,4}(X)$ and consider the Castelnuovo exact sequence
		\[
		0\to
		\LL_{1\times n}^{4,3}(\Res_D(X))
		\to
		\LL_{1\times n}^{4,4}(X)
		\to
		\LL_{1\times (n-1)}^{4,4}(\Tr_D(X))
		.\]
		\begin{itemize}
			\item \textit{Trace.} The trace linear system is $\LL_{1\times (n-1)}^{4,4}(3,2^{k^*(4,4;1,n-1)})$ and it is regular by induction hypothesis.
			\item \textit{Residue.} $\Res_D(X)$ is a scheme of fat points of type $(2,2^{k^*(4,4;1,n)-k^*(4,4;1,n-1)},1^{k^*(4,4;1,n-1)})$ such that $\Res_D(X)\cap D$ is a general scheme of fat points of type $(2,1^{k^*(4,4;1,n-1)})$ on $D$. By Lemma \ref{conto: 4,3 1 xn, con doppi atteso non vuoto}, $\vdim\LL_{1\times n}^{4,3}(\Res_D(X))\ge 0$. The linear system $\LL_{1\times n}^{4,3}(2,2^{k^*(4,4;1,n)-k^*(4,4;1,n-1)})$ is regular by Proposition \ref{pro: bigrado 34 come punti doppi}. Now we show that the $k^*(4,4;1,n-1)$ simple points on $D$ impose independent conditions on $\LL_{1\times n}^{4,3}(2,2^{k^*(4,4;1,n)-k^*(4,4;1,n-1)})$. Thanks to Lemma \ref{lemma:extra_points}, it suffices to prove that $\LL_{1\times n}^{4,2}(1,2^{k^*(4,4;1,n)-k^*(4,4;1,n-1)})=0$. This follows from Lemma \ref{lemma: 2,4 con tanti punti doppi}.
			\qedhere
		\end{itemize}
	\end{proof}
\end{lemma}

\begin{lemma}
	\label{lemma: 2,4 con tanti punti doppi}
	If $m$ and $n$ are positive integers, then $\LL_{m\times n}^{2,4}(
	2^{k^*(4,4;m,n)-k^*(4,4;m-1,n)})=0$.
	\begin{proof}
		In order to lighten the notation, we set
		\[w(m,n):=k^*(4,4;m,n)-k^*(4,4;m-1,n).\]
		We prove the statement by induction on $m$. The system $\LL_{1\times n}^{2,4}(
		2^{w(1,n)})$ is expected to be zero by Lemma~\ref{conto: vdim con w sottozero} and it is therefore zero by \cite[Theorem 3.1]{BalBerCat12}. Assume that $m\ge 2$. First of all we consider the case $n=1$. If $m\ge 2$, then $\LL_{m\times 1}^{2,4}(
		2^{w(m,1)})$ is zero by Lemma \ref{conto: evitare problemi w bbc} and \cite[Theorem 3.1]{BalBerCat12}. Now assume that $m,n\ge 2$ and let $D\subset\p^m\times\p^n$ be a divisor of bidegree $(1,0)$. Since $w(m,n)\ge w(m-1,n)$ by Lemma \ref{conto: w crescente e ammette bcc}, we can consider a scheme of fat points $X$ of type $(
		2^{w(m,n)})$ such that $\Tr_D(X)$ is a general scheme of fat points of type $(
		2^{w(m-1,n)})$ on $D$. Specialize $\LL_{m\times n}^{2,4}(
		2^{w(m,n)})$ to $\LL_{m\times n}^{2,4}(X)$ and consider the Castelnuovo exact sequence
		\[
		0\to
		\LL_{m\times n}^{1,4}(\Res_D(X))
		\to
		\LL_{m\times n}^{2,4}(X)
		\to
		\LL_{(m-1)\times n}^{2,4}(\Tr_D(X))
		.\]
		\begin{itemize}
			\item \textit{Trace.} The trace linear system is $\LL_{(m-1)\times n}^{2,4}(
			2^{w(m-1,n)})$ and it is zero by induction hypothesis.
			\item \textit{Residue.} $\Res_D(X)$ is a scheme of fat points of type $(2^{w(m,n)-w(m-1,n)},1^{w(m-1,n)})$ such that $\Res_D(X)\cap D$ is a general scheme of fat points of type $(1^{w(m-1,n)})$ on $D$. In order to prove that $\LL_{m\times n}^{1,4}(\Res_D(X))=0$, it suffices to show that $\LL_{m\times n}^{1,4}(2^{w(m,n)-w(m-1,n)})=0$. This follows from \cite[Corollary 2.2]{BCC} and Lemma \ref{conto: w crescente e ammette bcc}.
			\qedhere
		\end{itemize}
	\end{proof}
\end{lemma}

\begin{proposition}
	\label{pro: 4,4 quartuplo e doppi vuoto}
	If $(m,n)\neq (1,1)$, then $\calL_{m\times n}^{4,4}(4,2^{k_*(4,4;m,n)})=0$.
	\begin{proof}
		Without loss of generality, we assume that $m\le n$. We argue by induction on $m$. We solve the base case $m=1$ in Lemma \ref{lem: 4,4 quartuplo e doppi caso iniziale} and we assume that $n\ge m\ge 2$. We check that the system $\calL_{2\times 2}^{4,4}(4,2^{k_*(4,4;2,2)})$ is zero via a software computation. For $n\ge 3$, let $D\subset\p^m\times\p^n$ be a divisor of bidegree $(1,0)$. Let $X\subset\p^m\times\p^n$ be a scheme of fat points of type $(4,2^{k_*(4,4;m,n)})$ such that $\Tr_D(X)$ is  a general scheme of fat points of type $(4,2^{k_*(4,4;m-1,n)})$ on $D$. Specialize $\calL_{m\times n}^{4,4}(4,2^{k_*(4,4;m,n)})$ to $\calL_{m\times n}^{4,4}(X)$ and consider the Castelnuovo exact sequence
		\[
		0\to
		\LL_{m\times n}^{3,4}(\Res_D(X))
		\to
		\LL_{m\times n}^{4,4}(X)
		\to
		\LL_{(m-1)\times n}^{4,4}(\Tr_D(X))
		.\]
		\begin{itemize}
			\item \textit{Trace.} The trace linear system is $\LL_{(m-1)\times n}^{4,4}(4,2^{k_*(4,4;m-1,n)})$. Since $m\le n$, we also have $m-1\le n$ hence we can apply the induction hypothesis to conclude that $\LL_{(m-1)\times n}^{4,4}(4,2^{k_*(4,4;m-1,n)})=0$.
\item \textit{Residue.} $\Res_D(X)$ is a scheme of fat points of type $(3,2^{k_*(4,4;m,n)-k_*(4,4;m-1,n)},1^{k_*(4,4;m-1,n)})$ such that $\Res_D(X)\cap D$ is a general scheme of fat points of type $(3,1^{k_*(4,4;m-1,n)})$ on $D$. The system $\LL_{m\times n}^{3,4}(\Res_D(X))$ is expected to be zero by Lemma \ref{conto: orfanello}. Since $$k_*(4,4;m,n)-k_*(4,4;m-1,n)\le k^*(3,4;m,n)$$ by Lemma \ref{conto: kappa alto e kappa basso e differenze}, the system $\LL_{m\times n}^{3,4}(3,2^{k_*(4,4;m,n)-k_*(4,4;m-1,n)})$ is regular by Proposition \ref{lemma:3+2kup_mxn}. Now we have to prove that the $k_*(4,4;m-1,n)$ simple points on $D$ impose enough independent conditions to make $\LL_{m\times n}^{3,4}(\Res_D(X))$ zero. Thanks to Lemma \ref{lemma:extra_points}, we only have to show that $\LL_{m\times n}^{2,4}(2,2^{k_*(4,4;m,n)-k_*(4,4;m-1,n)})=0$. This follows from Lemma \ref{lemma: 2,4 con tanti punti doppi} and Lemma \ref{conto: ultimo conto}.
			\qedhere
		\end{itemize}
	\end{proof}
\end{proposition}

\begin{lemma}
	\label{lem: 4,4 quartuplo e doppi caso iniziale}
	If $n\ge 2$, then $\LL_{1\times n}^{4,4}(4,2^{k_*(4,4;1,n)})=0$.
	\begin{proof}
		A software computation shows that $\LL_{1\times 2}^{4,4}(4,2^{k_*(4,4;1,2)})=\LL_{1\times 3}^{4,4}(4,2^{k_*(4,4;1,3)})=0$. We assume that $n\ge 4$ and we argue by induction on $n$. Let $D\subset\p^1\times\p^n$ be a divisor of bidegree $(0,1)$. Let $X\subset\p^1\times\p^n$ be a scheme of fat points of type $(4,2^{k_*(4,4;1,n)})$ such that $\Tr_D(X)$ is a general scheme of fat points of type $(4,2^{k_*(4,4;1,n-1)})$ on $D$. Specialize $\LL_{1\times n}^{4,4}(4,2^{k_*(4,4;1,n)})$ to $\calL_{m\times n}^{4,4}(X)$ and consider the Castelnuovo exact sequence
		\[
		0\to
		\LL_{1\times n}^{4,3}(\Res_D(X))
		\to
		\LL_{1\times n}^{4,4}(X)
		\to
		\LL_{1\times (n-1)}^{4,4}(\Tr_D(X))
		.\]
		\begin{itemize}
			\item \textit{Trace.} The trace linear system is $\LL_{1\times (n-1)}^{4,4}(4,2^{k_*(4,4;1,n-1)})$ and it is zero by induction hypothesis.
			\item \textit{Residue.} $\Res_D(X)$ is a scheme of fat points of type $(3,2^{k_*(4,4;1,n)-k_*(4,4;1,n-1)},1^{k_*(4,4;1,n-1)})$ such that $\Res_D(X)\cap D$ is a general scheme of fat points of type $(3,1^{k_*(4,4;1,n-1)})$ on $D$. It is expected to be zero by Lemma \ref{conto: o il terzultimo}. Since $$k_*(4,4;1,n)-k_*(4,4;1,n-1)\le k^*(4,3;1,n)$$ by Lemma \ref{conto: quello di prima era il penultimo}, the system $\LL_{1\times n}^{4,3}(3,2^{k_*(4,4;1,n)-k_*(4,4;1,n-1)})$ is regular by Lemma \ref{lemma:3+2kup_mx1}. Now we have to prove that the $k_*(4,4;1,n-1)$ simple points on $D$ impose enough independent conditions to make $\LL_{1\times n}^{4,3}(\Res_D(X))$ zero. Thanks to Lemma \ref{lemma:extra_points}, we only have to show that $\LL_{1\times n}^{4,2}(2,2^{k_*(4,4;1,n)-k_*(4,4;1,n-1)})=0$. This follows from Lemma \ref{lemma: 2,4 con tanti punti doppi} and Lemma \ref{conto: ultimo conto}.
			\qedhere
		\end{itemize}
	\end{proof}
\end{lemma}

\begin{appendix}
	\section{Computations with algebraic software}\label{appendix: M2}
	
In this appendix we explain how to compute the dimension of a linear system with a straightforward interpolation, using the algebra software Macaulay2. This served us to check the base cases of our inductive proofs. The complete code can be found in the ancillary file of the arXiv submission or on the webpage of the second author. 

Fixed positive integers $c,d,m,n$, we consider the monomial basis $\ttB$ of the vector space of forms of bidegree $(c,d)$ in the bigraded coordinate ring $\ttS$ of $\bbP^m \times \bbP^n$, i.e.,
\begin{verbatim}
    S = QQ[x_0..x_m] ** QQ[y_0..y_n]; 
    B = first entries super basis({c,d},S).
\end{verbatim}
Sometimes we deal with linear subspaces of $\calL_{m\times n}^{c,d}$ whose base locus contains a union of general subspaces defined by bidegree $(0,1)$ forms. In this case we assume that such forms are chosen among the coordinates of $\bbP^m \times \bbP^n$ and we use as $\ttB$ the monomial basis of the homogeneous part in bidegree $(c,d)$ of the ideal defining the subspace. For example, in the case of Lemma \ref{lem:1xn_(3,2s)_4thStep}, the base locus contains the union of three general codimension-$2$ subspaces defined by forms of bidegree $(0,1)$; hence
\begin{verbatim}
    A1 = sub(ideal(y_0,y_1),S);
    A2 = sub(ideal(y_2,y_3),S);
    A3 = sub(ideal(y_4,y_5),S);
    B = first entries super basis({c,d},intersect({A1,A2,A3}));
\end{verbatim}
Now we consider the generic element in the span of $\ttB$, i.e.,
\begin{verbatim}
    C = QQ[c_0..c_(#B-1)];
    R = C[x_0..x_m]**C[y_0..y_n];
    F = sum for i to #B-1 list c_i*sub(B_i,R).
\end{verbatim}
At this point, we impose the conditions given by the scheme of fat points in the base locus. The scheme of fat points is defined by two attributes: a matrix $\ttP$ whose columns are the coordinates of the points supporting the scheme and a list of integers $\ttM$ giving the type of the scheme. Hence we obtain a system of linear equations in the $\ttc_i$'s whose solution is exactly the vector space we want to compute. Therefore we just need to compute the rank of the matrix $\ttV$ of the coefficients.
\begin{verbatim}
    V = sub(sub(diff(matrix {for j to #B-1 list C_j}, 
            transpose diff(symmetricPower(M_0-1,vars(R)),F)),
    	    for i to m+n+1 list R_i => P_0_i),QQ);
    for j from 1 to #M-1 do (
        V = V || sub(sub(diff(matrix {for j to #B-1 list C_j}, 
    	                transpose diff(symmetricPower(M_j-1,vars(R)),F)),
        	         for i to m+n+1 list R_i => P_j_i),QQ);
    	);
\end{verbatim}

	\section{Arithmetic computations}\label{appendix: contacci}
	
Here we collect some arithmetic properties that we use in the paper. For the convenience of the reader, we recall here the definition of the constants we used.
{\small
\begin{align*}
&r^*(c,d;m,n):=\ru{\frac{\binom{m+c}{m}\binom{n+d}{n}}{m+n+1}}&& 	r_*(c,d;m,n):=\rd{\frac{\binom{m+c}{m}\binom{n+d}{n}}{m+n+1}}\\
&k^*(c,d;m,n):=r^*(c,d;m,n)-m-n-1&& k_*(c,d;m,n):=r_*(c,d;m,n)-m-n-1\\
&f(m,n):=1+k^*(3,3;m,n)-k^*(3,3;m-1,n)&& \ell(m,n):=k_*(3,3;m,n)-k_*(3,3;m-1,n)\\
&s(n):=\frac{n(n+3)}{2}&&b(n):=k_*(3,3;2,n)-k_*(3,3;2,n-1) \\
&u(m,n) := k^*(3,4;m,n)-k^*(3,4;m,n-1)&&w(m,n):=k^*(4,4;m,n)-k^*(4,4;m-1,n)
\\
&v(n):=10(n+1)-(n+3)(1+b(n))+(n+2)b(n-1)&& 
\end{align*}
}

In the following, whenever we compute a root of a univariate polynomial, we use the command \texttt{solveSystem} from the {\tt NumericalAlgebraicGeometry} Macaulay2 package. The code can be found on the webpage of the second author.

\Fra{
\begin{remark}In many of our computations,
we will bound the integer part of the ratio $\frac{a}{b}$ of two natural numbers. Often we will apply the trivial properties
$$\frac{a}{b}-1\le\rd{\frac{a}{b}}\le\frac{a}{b}\le \ru{\frac{a}{b}}\le\frac{a}{b}+1,$$
but sometimes we will need the slightly more refined inequalities
$$\frac{a-(b-1)}{b}\le\rd{\frac{a}{b}}\le\frac{a}{b}\le \ru{\frac{a}{b}}\le\frac{a+b-1}{b}.$$
\end{remark}
}
\begin{lemma}
\label{conto: kbasso, usare bcc}
If $2\leq m \leq n$, then \[k^*(3,3;m,n)-k^*(3,3;m-1,n)\ge \ru{\frac{m+1}{m+n+1}\binom{n+3}{3}}+m.\]
\end{lemma}
\begin{proof} First we bound
	\begin{align*}
		k^*& (3,3;m,n) - k^*(3,3;m-1,n) - \left\lceil \frac{m+1}{m+n+1}\binom{n+3}{3}\right\rceil - m	\\
		& = \left\lceil \frac{\binom{m+3}{3}\binom{n+3}{3}}{m+n+1} \right\rceil - 1- \left\lceil \frac{\binom{m+2}{3}\binom{n+3}{3}}{m+n}\right\rceil - \left\lceil {\frac{m+1}{m+n+1}} \binom{n+3}{3}\right\rceil - m	\\
		& \geq \frac{\binom{m+3}{3}\binom{n+3}{3}}{m+n+1} - 1 - \frac{\binom{m+2}{3}\binom{n+3}{3}}{m+n} - 1 - {\frac{m+1}{m+n+1}}\binom{n+3}{3} -1- m = \frac{A(m,n)}{(m+n)(m+n+1)}, \\
	\end{align*}
	where 
	\begin{align*}
		A(m,n) = \binom{n+3}{3} & \left( \binom{m+3}{3}(m+n) - \binom{m+2}{3}(m+n+1) - (m+1)(m+n) \right) \\
		& - (m+3)(m+n)(m+n+1).
	\end{align*}
We distinguish two cases.
If $m\ge 9$, then we consider $A(m,n) - A(m,n-1)\in\bbC[m][n]$, which equals
	\[
\left(\frac{1}{3}\,m^{2}+\frac{1}{3}\,m\right)n^{3}+\left(\frac{1}{6}\,m^{3}+m^{2}+\frac{5}{6}\,m\right)n^{2}+\left(\frac{1}{2}\,m^{3}+\frac{2}{3}\,m^{2}-\frac{11}{6}\,m-6\right)n+\frac{1}{3}\,m^{3}-2\,m^{2}-\frac{19}{3}\,m.
	\]
Since $m \geq 9$, all the coefficients are positive and so $A(m,n)\ge A(m,n-1)$. 
In particular, since $m \leq n$, we have that
	\[
		A(m,n)  \geq A(m,m) = \frac{5}{36}\,m^{6}+\frac{11}{12}\,m^{5}+\frac{71}{36}\,m^{4}-\frac{31}{12}\,m^{3}-\frac{127}{9}\,m^{2}-\frac{19}{3}\,m\ge 0.
	\]
Assume now that $m\in\{2,\dots,8\}$. The polynomial $A(m,n) \in \bbC[m][n]$ has positive leading coefficient. It is easy to check that $A(m,n) \geq 0$ for $n \geq N(m)$, where 
	\[
N(m) = \begin{cases}
	3 & \mbox{ if } m = 2 \\
			2 & \mbox{ if } m\in\{3,\dots,8\}.
		\end{cases}
	\]
	Since $m \leq n$, the only remaining case is $(m,n) = (2,2)$, which is checked directly. 
\end{proof}

\begin{lemma}
\label{conto: vdim positiva, bigrado (2,3)} If $m$ and $n$ are positive integers, then \[\vdim\LL_{m\times n}^{2,3}(2^{k^*(3,3;m,n)-k^*(3,3;m-1,n)+1}) \ge k^*(3,3;m-1,n).\]
\begin{proof} We bound
\begin{align*}
&\binom{m+2}{2}\binom{n+3}{3}-(m+n+1)\ru{\frac{\binom{m+3}{3}\binom{n+3}{3}}{m+n+1}}+(m+n)\ru{\frac{\binom{m+2}{3}\binom{n+3}{3}}{m+n}}+m+n\\
&\ge\binom{m+2}{2}\binom{n+3}{3}-\binom{m+3}{3}\binom{n+3}{3}-m-n+\binom{m+2}{3}\binom{n+3}{3}+m+n=0.
\qedhere
\end{align*}
\end{proof}
\end{lemma}

\begin{lemma}
\label{conto: bigrado 3,1 atteso vuoto}
If $n\ge 2$, then $\vdim\LL_{1\times n}^{3,1}(1,2^{k^*(3,3;1,n)-k^*(3,3;1,n-1)})\le 0$.
\end{lemma}
\begin{proof} We bound
	\begin{align*}
& 4(n+1) - 1 - (n+2)\left( \left\lceil \frac{4\binom{n+3}{3}}{n+2}\right\rceil - (n+2) - \left\lceil \frac{4\binom{n+2}{3}}{n+1}\right\rceil + (n+1) \right) \\
& \leq 4n+3 - (n+2) \left(\frac{4\binom{n+3}{3}}{n+2} - (n+2) - \frac{4\binom{n+2}{3}}{n+1} - 1 + (n+1)\right) = -\frac{4}{3}n^2 + \frac{4}{3}n + 3 < 1. 
\qedhere	\end{align*}
\end{proof}
\Fra{
As in the proof of Lemma \ref{conto: bigrado 3,1 atteso vuoto}, in some of the following computations we will use the simple observation that if a linear system $\LL$ satisfies $\vdim(\calL) < 1$ (respectively, $\vdim(\calL) > -1$) then $\vdim(\calL) \leq 0$ (respectively, $\vdim(\calL) \geq 0$). 
}

\begin{lemma}
\label{conto:funzioniCrescenti}
If $n\ge m\ge 2$, then
\begin{enumerate}
\item $\vdim\LL_{n}^{3}(1,2^{\ell(m,n)-\ell(m-1,n)}) \le 0$,
\item $\ell(m-1,n)\le \ell(m,n)$,
\item $\vdim\LL_{n}^{3}(2^{f(m,n)-f(m-1,n)})\le 0$, and
\item $f(m-1,n)\le f(m,n)$.
\end{enumerate}

\begin{proof}
    We prove (1). We bound
    \begin{align*}
        & \binom{n+3}{3}-1-(n+1)\left(\rd{\frac{\binom{m+3}{3}\binom{n+3}{3}}{m+n+1}}  - 2\rd{\frac{\binom{m+2}{3}\binom{n+3}{3}}{m+n}}+\rd{\frac{\binom{m+1}{3}\binom{n+3}{3}}{m+n-1}}\right) \\
        & \quad \leq \binom{n+3}{3}-1-(n+1)\left({\frac{\binom{m+3}{3}\binom{n+3}{3}}{m+n+1}}+1 - 2{\frac{\binom{m+2}{3}\binom{n+3}{3}}{m+n}}+{\frac{\binom{m+1}{3}\binom{n+3}{3}}{m+n-1}}+1\right) \\
        & = \frac{A(m,n)}{18(m+n+1)(m+n)(m+n-1)}
    \end{align*} 
    where
    \begin{align*}
        A(m,n) & = \left(-n^{4}-4\,n^{3}+n^{2}-20\,n-42\right)m^{3}+\left(-3\,n^{5}-12\,n^{4}+3\,n^{3}-60\,n^{2}-126\,n\right)m^{2}\\
        &+\left(-3\,n^{6}-12\,n^{5}+4\,n^{4}-56\,n^{3}-127\,n^{2}+20\,n+42\right)m-36\,n^{4
      }-54\,n^{3}+36\,n^{2}+54\,n.
    \end{align*}
    All coefficients of $A(m,n)$ as a univariate polynomial in $\bbC[n][m]$ are negative for $n \geq 2$. Hence, $A(m,n) \leq 0$ for $n \geq m \geq 2$. \Fra{Statement (2) follows} directly from (1).
    
    We prove (3). We bound
    \begin{align*}
        & \binom{n+3}{3}-(n+1)\left(\ru{\frac{\binom{m+3}{3}\binom{n+3}{3}}{m+n+1}}  - 2\ru{\frac{\binom{m+2}{3}\binom{n+3}{3}}{m+n}}+\ru{\frac{\binom{m+1}{3}\binom{n+3}{3}}{m+n-1}}\right) \\
        & \leq \binom{n+3}{3}-(n+1)\left({\frac{\binom{m+3}{3}\binom{n+3}{3}}{m+n+1}}  - 2{\frac{\binom{m+2}{3}\binom{n+3}{3}}{m+n}}-2+{\frac{\binom{m+1}{3}\binom{n+3}{3}}{m+n-1}}\right) \\
        & = \frac{A(m,n)}{18(m+n+1)(m+n)(m+n-1)}
    \end{align*}
    where
    \begin{align*}
        A(m,n) & = \left(-n^{4}-4\,n^{3}+n^{2}+52\,n+48\right)m^{3}+\left(-3\,n^{5}-12\,n^{4}+3\,n^{3}+156\,n^{2}+144\,n\right)m^{2}\\
        &+\left(-3\,n^{6}-12\,n^{5}+4\,n^{4}+160\,n^{3}+143\,n^{2}-52\,n-48\right)m+36\,n
      ^{4}+36\,n^{3}-36\,n^{2}-36\,n.
    \end{align*}
    For $n \geq 4$, the univariate polynomial $A(m,n) \in \bbC[n][m]$ has only one change of sign in the coefficients and, therefore, it has at most one positive root by Descartes' rule of signs. Note that $A(0,n) = 36\,n^{4}+36\,n^{3}-36\,n^{2}-36\,n$ and $A(1,n) = -3\,n^{6}-15\,n^{5}+27\,n^{4}+195\,n^{3}+264\,n^{2}+108\,n$. For $n \geq 4$, we have $A(0,n) > 0$ and $A(1,n) < 0$, i.e., the only positive root of $A(m,n)$ belongs to the interval $m \in (0,1)$. We conclude that $A(m,n) \leq 0$ for $m \geq 2$ and $n \geq 4$. The remaining cases $(m,n) \in \{(2,2),(2,3),(3,3)\}$ can be checked directly. \Fra{Statement (4) follows directly from (3)}.
\end{proof}

\end{lemma}

\begin{lemma}\label{conto:f_vdim_positiva}
If $n\ge m \geq 2$, then 
$\vdim\calL_{m\times n}^{2,3}(2^{f(m,n)}) \geq 0.$
\begin{proof} We bound
    \begin{align*}
	& \binom{m+2}{2}\binom{n+3}{3} - (m+n+1)\left(1+\left\lceil \frac{\binom{m+3}{3}\binom{n+3}{3}}{m+n+1} \right\rceil - \left\lceil \frac{\binom{m+2}{3}\binom{n+3}{3}}{m+n} \right\rceil\right) \\
	& \geq \binom{m+2}{2}\binom{n+3}{3} - (m+n+1) - \binom{m+3}{3}\binom{n+3}{3} - (m+n) + \frac{m+n+1}{m+n}\binom{m+2}{3}\binom{n+3}{3} \\
	& = \frac{1}{m+n}\binom{m+2}{3}\binom{n+3}{3} - m - n  = \frac{A(m,n)}{36(m+n)},
\end{align*}
where
\begin{align*}
    A(m,n)=& \left(m^{3}+3\,m^{2}+2\,m\right)n^{3}+\left(6\,m^{3}+18\,m^{2}+12\,m-72\right)n^{2}\\
    &+\left(11\,m^{3}+33\,m^{2}-122\,m-36\right)n+6\,m^{3}-54\,m^{2}-24\,m.
\end{align*}
For $m \geq 10$, all the coefficients of the univariate polynomial $A(m,n) \in \bbC[m][n]$ are positive and therefore $A(m,n) > 0$. For $2 \leq m \leq 9$, since the leading coefficient is positive, then we know that $A(m,n) > 0$ for $n \geq N(m) \geq m$. It is easy to check that $N(m) = m$ for all $m \in \{2,\ldots,9\}$. 
\end{proof}
\end{lemma}

\begin{lemma}
\label{conto:upperBound_l-l}
If $n \geq m \geq 2$, then 
	\begin{enumerate}
\item $f(m,n) - f(m-1,n) \leq \left\lfloor \frac{m+1}{m+n+1}\binom{n+3}{3} \right\rfloor - m$ and
		\item $1 + \ell(m,n) - \ell(m-1,n) \leq \left\lfloor \frac{m+1}{m+n+1}\binom{n+3}{3}\right\rfloor - m$.
	\end{enumerate}
\end{lemma}
\begin{proof}
In the cases $(m,n) \in \{(2,2),(3,3)\}$ the two statements are checked directly. For $n \geq m \geq 2$ and $(m,n) \not\in \{(2,2),(3,3)\}$, we prove the following
\begin{equation}\label{eq:upperBound_l-l}
2 + \ell(m,n) - \ell(m-1,n) \leq \left\lfloor \frac{m+1}{m+n+1}\binom{n+3}{3} \right\rfloor- m.
	\end{equation}
Note that $(1)$ follows from \eqref{eq:upperBound_l-l}, because $f(m,n) - f(m-1,n) \leq 2 + \ell(m,n) - \ell(m-1,n)$, while $(2)$ becomes trivial.
	Hence we only have to prove \eqref{eq:upperBound_l-l}. We bound
	\begin{align*}
		& 2 + \left\lfloor \frac{\binom{m+3}{3}\binom{n+3}{3}}{m+n+1} \right\rfloor - 2 \left\lfloor \frac{\binom{m+2}{3}\binom{n+3}{3}}{m+n} \right\rfloor +  \left\lfloor \frac{\binom{m+1}{3}\binom{n+3}{3}}{m+n-1} \right\rfloor - \left\lfloor \frac{(m+1)\binom{n+3}{3}}{m+n+1} \right\rfloor + m \\
		& \leq 2 + \frac{\binom{m+3}{3}\binom{n+3}{3}}{m+n+1} - 2 \frac{\binom{m+2}{3}\binom{n+3}{3}}{m+n} + 2 + \frac{\binom{m+1}{3}\binom{n+3}{3}}{m+n-1} - \frac{(m+1)\binom{n+3}{3}}{m+n+1} + 1 + m\\
		& = \frac{A(m,n)}{18(m+n)(m+n-1)(m+n+1)}
	\end{align*}
	where
	\begin{align*}
		A(m& ,n) = \left(-3\,m^{2}-3\,m\right)n^{4}+\left(-2\,m^{3}-18\,m^{2}+2\,m+90\right)n^{3}\\
		&+\left(-12\,m^{3}+21\,m^{2}+249\,m\right)n^{2}+\left(32\,m^{3}+252\,m^{2}-14\,m-90\right)n\\
		&+18\,m^{4}+78\,m^{3}-18\,
       m^{2}-78\,m
	\end{align*}
	We regard at $A(m,n)$ as a univariate polynomial in $\bbC[m][n]$. 
	\begin{itemize}
	    \item For $m \geq 6$, the sequence of coefficients of $A(m,n)$ has only one change of sign and, by Descartes' rule of signs, it has at most one positive root. Note that $A(m,0) = 18\,m^{4}+78\,m^{3}-18\,m^{2}-78\,m > 0$ and $A(m,m) = -5\,m^{6}-33\,m^{5}+73\,m^{4}+669\,m^{3}-32\,m^{2}-168\,m < 0$. Hence, the only positive root belongs to the interval $n \in (0,m)$. In particular, for $n \geq m$ we have $A(m,n) \leq 0$. 
	    \item For $m \in \{2,\ldots,5\}$, since the leading coefficient is negative, we know that there exists $N(m)$ such that $A(m,n) \leq 0$ for $n \geq N(m)$. In this case:
	    \[
	        N(m) = \begin{cases}
	            7 & \text{ for } m = 2; \\
	            5 & \text{ for } m \in \{3,4,5\}.
	        \end{cases}
	    \]
	    For the remaining cases $(m,n) \in \{(2,3),(2,4),(2,5),(2,6),(3,4),(4,4)\}$ we check \eqref{eq:upperBound_l-l} directly.
	\end{itemize}
\end{proof}

\begin{lemma}
	\label{conto:vdim_lowerbound_13}
If $n \geq m \geq 2$, then 
$\vdim \calL_{m\times n}^{1,3}(2^{f(m,n)-f(m-1,n)}) \geq f(m-1,n)$.
\end{lemma}
\begin{proof}
	We bound
	\begin{align*}
		&(m+1)  \binom{n+3}{3} - (m+n+1)(k^*(3,3;m,n) - 2k^*(3,3;m-1,n) + k^*(3,3;m-2,n)) \\
		& \quad - k^*(3,3;m-1,n) + k^*(3,3;m-2,n) \\
		& \geq (m+1)\binom{n+3}{3}- (m+n+1)\left( \left\lceil \frac{\binom{m+3}{3}\binom{n+3}{3}}{m+n+1} \right\rceil - (m+n+1) - 2 \left\lceil \frac{\binom{m+2}{3}\binom{n+3}{3}}{m+n} \right\rceil + 2(m+n)\right. \\ 
		& \quad  \left. +  \left\lceil \frac{\binom{m+1}{3}\binom{n+3}{3}}{m+n-1} \right\rceil - (m+n-1) \right) - \left\lceil \frac{\binom{m+2}{3}\binom{n+3}{3}}{m+n} \right\rceil + (m+n) +  \left\lceil \frac{\binom{m+1}{3}\binom{n+3}{3}}{m+n-1} \right\rceil  - (m+n-1) \\
		& \geq (m+1)\binom{n+3}{3}- (m+n+1)\left( \frac{\binom{m+3}{3}\binom{n+3}{3}}{m+n+1} + 1 - 2 \frac{\binom{m+2}{3}\binom{n+3}{3}}{m+n} + \frac{\binom{m+1}{3}\binom{n+3}{3}}{m+n-1} + 1 \right) \\ & \quad -  \frac{\binom{m+2}{3}\binom{n+3}{3}}{m+n} +  \frac{\binom{m+1}{3}\binom{n+3}{3}}{m+n-1} = \frac{A(m,n)}{36(m+n)(m+n-1)},
	\end{align*}
	where
	\begin{align*}
		A(m,n) = & \left(3\,m^{2}+3\,m\right)n^{4}+\left(2\,m^{3}+18\,m^{2}+16\,m-72\right)
      n^{3} \\
      & +\left(12\,m^{3}+33\,m^{2}-195\,m\right)n^{2}+\left(22\,m^{3}-198\,m
      ^{2}-4\,m+72\right)n-60\,m^{3}+60\,m.
	\end{align*}
	We consider $A(m,n)$ as a polynomial in $\bbC[m][n]$. For any $m \geq 2$, the change of signs in the coefficients is equal to one. Hence, by Descartes' rule of signs, $A(m,n)$ has a unique positive root. Moreover, it is immediate to check that $A(m,0) < 0$ while $A(m,m) > 0$ whenever $m \geq 3$. Therefore such unique positive root is in the interval $(0,m)$ and by the assumption $n \geq m$, we deduce that $A(m,n) \geq 0$. In the case $m = 2$, the univariate polynomial $A(2,n)$ is positive for $n \geq 4$. In the remaining cases $(m,n) \in \{(2,2),(2,3)\}$, we check the statement directly.
	\qedhere
\end{proof}

\begin{lemma}
\label{conto: bigrado 2,3, triplo e doppi, left hand side}
\Ale{Let $n \geq m \geq 2$}. Then $\vdim\LL_{m\times n}^{2,3}(3,2^{k_*(3,3;m,n)-k_*(3,3;m-1,n)}) \leq k_*(3,3;m-1,n)$.
\end{lemma}
\begin{proof}
	We need to show that
	\begin{align*}
		\binom{m+2}{2}\binom{n+3}{3} &- \binom{m+n+2}{2} - (m+n+1)\left(\left\lfloor \frac{\binom{m+3}{3}\binom{n+3}{3}}{m+n+1} \right\rfloor - (m+n+1) \right. \\
		& \left. -\left\lfloor \frac{\binom{m+2}{3}\binom{n+3}{3}}{m+n} \right\rfloor + (m+n)\right) - \left\lfloor \frac{\binom{m+2}{3}\binom{n+3}{3}}{m+n} \right\rfloor + (m+n)\leq 0.
	\end{align*}
	The left-hand-side is smaller or equal to 
	\begin{align*}
		\binom{m+2}{2}\binom{n+3}{3} &- \binom{m+n+2}{2} - (m+n+1)\left( \frac{\binom{m+3}{3}\binom{n+3}{3}}{m+n+1} - 1 - (m+n+1) \right. \\
		& \left. - \frac{\binom{m+2}{3}\binom{n+3}{3}}{m+n}  +  (m+n)\right) -  \frac{\binom{m+2}{3}\binom{n+3}{3}}{m+n} + 1 + (m+n) = \frac{A(m,n)}{2}
	\end{align*}
	with $A(m,n) = -m^{2}+\left(-2\,n+3\right)m-n^{2}+3\,n+4 \in \bbC[n][m]$. For $n \geq 4$, the univariate polynomial $A(m,n)$ has negative coefficients and we deduce that $A(m,n) \leq 0$. In the remaining cases $(m,n) \in \{(2,2),(2,3),(3,3)\}$, we check the statement directly.
\end{proof}

\begin{lemma}\label{conto: bigrado 2,3, usare bcc}
\Ale{If $n\geq m \geq 2$}, then
\[1+k_*(3,3;m,n)-k_*(3,3;m-1,n)\ge\ru{\frac{m+1}{m+n+1}\binom{n+3}{3}}+m.\]
\end{lemma}
\begin{proof}
	We bound
	\begin{align*}
		& 1 + \left\lfloor \frac{\binom{m+3}{3}\binom{n+3}{3}}{m+n+1} \right\rfloor - (m+n+1) -\left\lfloor \frac{\binom{m+2}{3}\binom{n+3}{3}}{m+n} \right\rfloor + (m+n) - \ru{\frac{m+1}{m+n+1}\binom{n+3}{3}}-m \\
		& \geq \frac{\binom{m+3}{3}\binom{n+3}{3}}{m+n+1} -1 - \frac{\binom{m+2}{3}\binom{n+3}{3}}{m+n} - \frac{m+1}{m+n+1}\binom{n+3}{3}-1-m = \frac{A(m,n)}{36(m+n)(m+n+1)},
	\end{align*}
	where
	\begin{align*}
		A(&m,n) = \left(3\,m^{2}+3\,m\right)n^{4}+\left(2\,m^{3}+18\,m^{2}+16\,m\right)n^{3} \\
		& +\left(12\,m^{3}+33\,m^{2}-15\,m-72\right)n^{2}+\left(22\,m^{3}-54\,m^{
      2}-184\,m-72\right)n-24\,m^{3}-108\,m^{2}-84\,m.
	\end{align*}
	For $m \geq 2$, the sequence of coefficients of the univariate polynomial $A(m,n) \in \bbC[m][n]$ has only one change of signs. By Descartes' rule of signs, $A(m,n)$ has a unique positive root. It is immediate to check that $A(m,0) < 0$ while $A(m,m) \geq 0$. Hence such unique positive root is in the interval $(0,m]$. Since $n \geq m$, we deduce that $A(m,n) \geq 0$. \qedhere
\end{proof}

\begin{lemma}
\label{conto: left hand side del punto quartuplo, atteso vuoto}
If $n\ge 2$, then $\vdim\LL_{2\times n}^{3,2}(3,2^{k_*(3,3;2,n)-k_*(3,3;2,n-1)}) \leq k_*(3,3;2,n-1)$.
\end{lemma}
\begin{proof}
	We bound
	\begin{align*}
		& 10\binom{n+2}{2} - \binom{n+4}{2}- (n+3)\left(\left\lfloor \frac{10\binom{n+3}{3}}{n+3} \right\rfloor - (n+3) - \left\lfloor \frac{10\binom{n+2}{3}}{n+2} \right\rfloor + (n+2) \right) \\
		& \qquad -\left\lfloor \frac{10\binom{n+2}{3}}{n+2} \right\rfloor + (n+2) \\
		& \leq 10\binom{n+2}{2} - \binom{n+4}{2}- (n+3)\left( \frac{10\binom{n+3}{3}}{n+3} - 2  - \frac{10\binom{n+2}{3}}{n+2} \right) - \frac{10\binom{n+2}{3}}{n+2} + 1 + (n+2) \\
		& = \frac{-n^2-n-6}{2}\le 0.\qedhere
	\end{align*}
\end{proof}

\begin{lemma}
\label{conto: bigrado (3,1) è vuoto}
\Ale{If $n \geq 2$}, then $1+k_*(3,3;2,n)-k_*(3,3;2,n-1)\ge\ru{\frac{10(n+1)}{n+3}}+n$.
\end{lemma}
\begin{proof}
    The statement is a special case of Lemma \ref{conto: bigrado 2,3, usare bcc} by inverting $m$ and $n$ and replacing $n = 2$.
\end{proof}

\begin{lemma}\label{conto: bigrado 1,3, doppi e semplici}
If $n\ge m\ge 2$, then $\vdim\LL_{m\times n}^{1,3}(2^{1+\ell(m,n)-\ell(m-1,n)}) \geq \ell(m-1,n)$.
\begin{proof} 
	We bound
	\begin{align*}
    & (m+1)\binom{n+3}{3} - (m+n+1)\left(1+k_*(3,3;m,n) -2k_*(3,3;m-1,n)+k_*(3,3;m-2,n)\right)\\ 
    & \qquad - k_*(3,3;m-1,n)+k_*(3,3;m-2,n) \\
	& \geq (m+1)\binom{n+3}{3} - (m+n+1)\left(1+\left\lfloor \frac{\binom{m+3}{3}\binom{n+3}{3}}{m+n+1} \right\rfloor -2\left\lfloor \frac{\binom{m+2}{3}\binom{n+3}{3}}{m+n} \right\rfloor+\left\lfloor \frac{\binom{m+1}{3}\binom{n+3}{3}}{m+n-1} \right\rfloor\right)\\ 
	& \qquad -\left\lfloor \frac{\binom{m+2}{3}\binom{n+3}{3}}{m+n} \right\rfloor+\left\lfloor \frac{\binom{m+1}{3}\binom{n+3}{3}}{m+n-1} \right\rfloor \\
	&\geq (m+1)\binom{n+3}{3} - (m+n+1)\left(1+ \frac{\binom{m+3}{3}\binom{n+3}{3}}{m+n+1} -2 \frac{\binom{m+2}{3}\binom{n+3}{3}}{m+n} - 2 + \frac{\binom{m+1}{3}\binom{n+3}{3}}{m+n-1} \right)\\ 
	& \qquad - \frac{\binom{m+2}{3}\binom{n+3}{3}}{m+n} + \frac{\binom{m+1}{3}\binom{n+3}{3}}{m+n-1} - 1 = \frac{A(m,n)}{36(m+n)(m+n-1)},
	\end{align*}
	where
	\begin{align*}
		A(m,n) = & \left(3\,m^{2}+3\,m\right)n^{4}+\left(2\,m^{3}+18\,m^{2}+16\,m+36\right) n^{3} \\
		&+\left(12\,m^{3}+33\,m^{2}+129\,m-36\right)n^{2}+\left(22\,m^{3}+126
      \,m^{2}-76\,m\right)n\\ 
      &+48\,m^{3}-36\,m^{2}-12\,m.
	\end{align*}
	The univariate polynomial $A(m,n) \in \bbC[m][n]$ has positive coefficients and so $A(m,n) \geq 0$.
\end{proof}
\end{lemma}

\begin{lemma}
\label{conto: s consente di usare bcc}
If $n\ge 3$, then $1+\ell(2,n)-s(n)\le\rd{\frac{3}{n+3}\binom{n+3}{3}}-2$.
\begin{proof}  
A software computation shows that the claim holds for $n\in\{3,4\}$. We assume that $n \ge 5$ and we bound\begin{align*}
1+&\ell(2,n)-s(n)-\rd{\frac{3}{n+3}\binom{n+3}{3}}+2=3+k_*(2,n)-k_*(1,n)-\frac{(n+3)n}{2}-\frac{(n+2)(n+1)}{2}\\
&=3+\rd{\frac{10\binom{n+3}{3}}{n+3}}-(n+3)-\rd{\frac{4\binom{n+3}{3}}{n+2}}+(n+2)-\frac{(n+3)n}{2}-\frac{(n+2)(n+1)}{2}\\
&\le 3+\frac{5(n+2)(n+1)}{3}-\frac{2(n+3)(n+1)}{3}-\frac{(n+3)n}{2}-\frac{(n+2)(n+1)}{2}=\frac{2}{3}(5-n)\le 0.\qedhere
\end{align*}
\end{proof}
\end{lemma}

\begin{lemma}
\label{conto: s minore di l}
Let $n\ge 3$. Then
\begin{equation}\label{eq:s_minore_l}
\ell(2,n)-s(n)\ge \ru{\frac{\binom{n+3}{3}-1}{n+1}}.
\end{equation}
In particular
\begin{enumerate}
\item $s(n)\le \ell(2,n)$ and
\item $\vdim\LL_{n}^{3}(1,2^{\ell(2,n)-s(n)})\le 0$.
\end{enumerate}
\begin{proof}
We directly check that
\begin{align*}
\ell&(2,n)-s(n)- \ru{\frac{\binom{n+3}{3}-1}{n+1}}=k_*(2,n)-k_*(1,n)-s(n)-\ru{\frac{\binom{n+3}{3}-1}{n+1}}\\
&=\rd{\frac{10\binom{n+3}{3}}{n+3}}-(n+3)-\rd{\frac{4\binom{n+3}{3}}{n+2}}+(n+2)-s(n)- \ru{\frac{n(n^2+6n+11)}{6(n+1)}}\\
&\ge \frac{5(n+2)(n+1)}{3}-2-\frac{2(n+3)(n+1)}{3}-\frac{n(n^2+6n+11)}{6(n+1)}-1=\frac{5\,n^{3}+14\,n^{2}-7n-10}{6(n+1)}\ge 0.\qedhere
\end{align*}
\end{proof}
\end{lemma}

\begin{lemma}
\label{conto:lowerBound_vdim_s}
If $n \geq 3$, then $
		\vdim\calL_{2\times n}^{1,3}(2^{1+\ell(2,n)-s(n)}) \geq s(n)$.
\end{lemma}
\begin{proof}
For $n\in\{3,\dots,7\}$, the statement can be checked directly. For $n\ge 8$ we bound
	\begin{align*}
3&\binom{n+3}{3} - (n+3)\left(1+\ell(2,n)-s(n)\right) - s(n)\\
& =3\binom{n+3}{3}- (n+3)\left(1+\left\lfloor \frac{10\binom{n+3}{3}}{n+3} \right\rfloor-(n+3)-\left\lfloor \frac{4\binom{n+3}{3}}{n+2} \right\rfloor+(n+2)\right)+(n+2)\frac{n(n+3)}{2}\\
&\geq 3\binom{n+3}{3} - (n+3)\left(1+ \frac{10\binom{n+3}{3}}{n+3}- \frac{4\binom{n+3}{3}}{n+2} \right)+(n+2)\frac{n(n+3)}{2} = \frac{n^{2}-5\,n-24}{6}\ge 0.\qedhere
	\end{align*}
\end{proof}

\begin{lemma}
	\label{conto:b-b_modulo3}
If $n \geq 3$, then
	\[
	b(n)-b(n-1)  = 
		\begin{cases}
4 & \mbox{ if }n \equiv 1 \mod 3 \\
3 & \text{ otherwise}.
		\end{cases}
	\]
	In particular, we deduce that $b(n) - b(n-3) = 10 > 0$.
\end{lemma}
\begin{proof}
	By definition
	\[
		b(n) = k_*(3,3;2,n) - k_*(3,3;2,n-1) = \left\lfloor \frac{10\binom{n+3}{3}}{n+3} \right \rfloor - (n+3) - \left\lfloor \frac{10\binom{n+2}{3}}{n+2} \right \rfloor + (n+2).
	\]
	Then
	\begin{align*}
		b(n) - b(n-1) & =  \left\lfloor \frac{10\binom{n+3}{3}}{n+3} \right \rfloor - 2 \left\lfloor \frac{10\binom{n+2}{3}}{n+2} \right \rfloor + \left\lfloor \frac{10\binom{n+1}{3}}{n+1} \right \rfloor \\
		& = \left\lfloor \frac{5(n+2)(n+1)}{3} \right \rfloor - 2 \left\lfloor \frac{5(n+1)n}{3} \right \rfloor + \left\lfloor \frac{5n(n-1)}{3} \right \rfloor.
	\end{align*}
	For $n = 3m+1$, $m \in \bbZ$:
	\[
	    b(n)-b(n-1) = 5(m+1)(3m+2) - 2\left\lfloor \frac{5(3m+2)(3m+1)}{3} \right \rfloor + 5(3m+1)m = 4.
	\]
	For $n = 3m+2$, $m \in \bbZ$:
	\[
	    b(n)-b(n-1) = 5(3m+4)(m+1) - 10(3m+2)(m+1) + \left\lfloor \frac{5(3m+2)(3m+1)}{3} \right\rfloor = 3.
	\]
	For $n = 3m$, $m \in \bbZ$:
	\[
	    b(n)-b(n-1) = \left\lfloor \frac{5(3m+2)(3m+1)}{3} \right \rfloor - 10(3m+1)m + 5m(3m-1)  = 3. \qedhere
	\]
\end{proof}

\begin{lemma}\label{conto:v-v_modulo3}
		If $n\ge 4$, then
		\[
			v(n) - v(n-3) = 
			\begin{cases}
5 & \text{ if } n \equiv 1 \mod 3 \\
				8 & \text{ otherwise}.
			\end{cases}
		\]
\end{lemma}
\begin{proof}
	By definition,
	\begin{align*}
		v(n) - v(n-3) 
		= 30 - 3(1+b(n)-b(n-1)) - (b(n-1)-b(n-4)).
	\end{align*}
	By Lemma \ref{conto:b-b_modulo3}, we deduce that, for any $n \ge 4$, $b(n)-b(n-3) = 10$. Hence
	\[
		v(n) - v(n-3) = \begin{cases}
			30 - 3\cdot 5 - 10 = 5 & \text{ if } n \equiv 1 \mod 3 \\	
			30 - 3\cdot 4 - 10 = 8 & \text{ otherwise}.
		\end{cases}\qedhere
	\]
\end{proof}

	\begin{lemma}\label{appendix:few2_3,3}
	If $m$ and $n$ are positive integers, then $\vdim \calL_{m\times n}^{3,3}(2^{1+u(m,n)})\geq k^*(3,4;m,n-1)$.
\end{lemma}
\begin{proof}
	We bound
	\begin{align*}
		& \binom{m+3}{3} \binom{n+3}{3} - (m+n+1) \left( 1 + \left\lceil \frac{\binom{m+3}{3}\binom{n+4}{4}}{m+n+1} \right\rceil - (m+n+1)\right) \\ 
		& \qquad + (m+n)\left( \left\lceil \frac{\binom{m+3}{3}\binom{n+3}{4}}{m+n} \right\rceil - (m+n)\right) \\
		&\geq \binom{m+3}{3}\binom{n+3}{3} - \binom{m+3}{3}\binom{n+4}{4} - (m+n) + \binom{m+3}{3}\binom{n+3}{4} + (m+n) = 0.\qedhere
	\end{align*}
\end{proof}

\begin{lemma}\label{appendix:enough2fat_3,2}
	If $m$ and $n$ are positive integers, then $\vdim \calL_{m\times n}^{3,2}(1,2^{u(m,n)}) \leq 0$.
\end{lemma}
\begin{proof}
	We bound
	\begin{align*}
	&\binom{m+3}{3}\binom{n+2}{2}-1-(m+n+1)\left(\left\lceil \frac{\binom{m+3}{3}\binom{n+4}{4}}{m+n+1} \right\rceil -1 - \left\lceil \frac{\binom{m+3}{3}\binom{n+3}{4}}{m+n} \right\rceil \right) \\
	&\leq \binom{m+3}{3}\binom{n+2}{2} -1-\binom{m+3}{3}\binom{n+4}{4} + 2(m+n+1) + \frac{m+n+1}{m+n}\binom{m+3}{3}\binom{n+3}{4} \\
	& = \frac{A(m,n)}{144(m+n)},
	\end{align*}
	where, as a polynomial in $\bbC[n][m]$,
	\begin{align*}
		A(m,n) = & \left(-4\,n^{3}-12\,n^{2}-8\,n\right)m^{4}+\left(-3\,n^{4}-30\,n^{3}-69\,n^{2}-42\,n\right)m^{3} \\  	& + \left(-18\,n^{4}-80\,n^{3}-114\,n^{2}-52\,n+288\right)m^{2}+\left(-33\,n^{4}-90\,n^{3}-39\,n^{2}+
		594\,n+144\right)m\\ & -18\,n^{4}-36\,n^{3}+306\,n^{2}+180\,n.
	\end{align*}
	For $n \geq 4$, all coefficients of the latter univariate polynomial are negative, so $A(m,n) \leq 0$. Since the leading coefficient of $A(m,n)$ is negative, we have that $A(m,n) \leq 0$ for $m \geq M(n)$, where
	\[
	M(n) = \begin{cases}
		3 & \text {if } n = 1 \\
		1 & \text{if } n \in\{2,3\}.
	\end{cases}
	\]
	Therefore we are left with the cases $(m,n) \in \{(1,1),(2,1)\}$ for which we check the claim directly. 
\end{proof}

\begin{lemma}\label{appendix:fewPoints_2,4}
	If $m$ is a positive integer, $1+k^*(3,4;m,1)-k^*(3,4;m-1,1) \leq 2(m+1).$
\end{lemma}
\begin{proof}For $m = 1$ the statement is checked directly. If $m\ge 2$, then
	\begin{align*}
		1&+\left\lceil \frac{5\binom{m+3}{3}}{m+2}\right\rceil-(m+2)-\left\lceil\frac{5\binom{m+2}{3}}{m+1}\right\rceil+(m+1)-2(m+1)    \\
		&\le  \frac{5(m+3)(m+1)}{6}+\frac{5}{6}-\frac{5(m+2)m}{6}-2(m+1) = -\frac{1}{3}m+\frac{4}{3}<1.\qedhere
	\end{align*}
\end{proof}

\begin{lemma}\label{appendix:enough2fat_3,2_mx1}
	If $m$ is a positive integer, then $u(m,1) \geq \left\lceil \frac{3\binom{m+3}{3}}{m+2}\right\rceil$.
\end{lemma}
\begin{proof}We can directly check that the statement holds for $m\in\{1,2\}$. Let $m\ge 3$. We bound
\begin{align*}
&u(m,1)-\left\lceil \frac{3\binom{m+3}{3}}{m+2}\right\rceil =\left\lceil \frac{5\binom{m+3}{3}}{m+2} \right\rceil - (m+2) - \left\lceil \frac{\binom{m+3}{3}}{m+1} \right\rceil + (m+1)- \left\lceil \frac{(m+3)(m+1)}{2}\right\rceil\\
&\ge \frac{5(m+3)(m+1)}{6}  -1 - \frac{(m+3)(m+2)}{6} - \frac{5}{6} - \frac{(m+3)(m+1)}{2} - \frac{1}{2}  = \frac{1}{6}\,m^{2}+\frac{1}{2}\,m-\frac{7}{3}\ge 0.\qedhere
\end{align*}
\end{proof}

\begin{lemma}\label{appendix:3,4_mx1}
	If $m$ is a positive integer, then	    \begin{enumerate}
\item\label{item: 24m1} $\vdim\calL_{m\times 1}^{2,4}(2^{1+k^*(3,4;m,1)-k^*(3,4;m-1,1)}) \geq k^*(3,4;m-1,1)$ and
\item\label{item: 14m1} $\vdim\calL_{m\times 1}^{1,4}(1,2^{k^*(3,4;m,1)-k^*(3,4;m-1,1)}) \leq 0.$
	\end{enumerate}
\end{lemma}
\begin{proof}
(1) We bound
	\begin{align*}
	& 5\binom{m+2}{2}-(m+2)\left( \left\lceil\frac{5\binom{m+3}{3}}{m+2}\right\rceil- \left\lceil\frac{5\binom{m+2}{3}}{m+1}\right\rceil \right) - \left\lceil\frac{5\binom{m+2}{3}}{m+1}\right\rceil + (m+1)\\
	& \geq 5\binom{m+2}{2} - 5\binom{m+3}{3} - \frac{5(m+2)}{6} + \frac{5(m+2)^2m}{6} - \frac{5(m+2)m}{6} - \frac{5}{6} + (m+1)= \frac{1}{6}m-\frac{3}{2}.
	\end{align*}
	The latter is strictly larger than $-1$ for $m \geq 4$. We check the cases $m \in \{1,2,3\}$ directly.
	
\noindent (2) We directly check that $\vdim\calL_{1\times 1}^{1,4}(1,2^{k^*(3,4;1,1)-k^*(3,4;0,1)}) = 0$. For $m\ge 2$ we bound
\begin{align*}
&\vdim\calL_{m\times 1}^{1,4}(1,2^{k^*(3,4;m,1)-k^*(3,4;m-1,1)})\\
&=5(m+1)-1-(m+2)\left(\ru{\frac{5}{m+2}\binom{m+3}{3}}-1-\ru{\frac{5}{m+1}\binom{m+2}{3}}\right)\\
&\le 5m+4-(m+2)\left(\frac{5}{m+2}\binom{m+3}{3}-2-\frac{5}{m+1}\binom{m+2}{3}\right)
=-\frac{10m^2-7m-18}{6}\le 0.\qedhere
\end{align*}
\end{proof}

\begin{lemma}\label{appendix:applyBCC_1+2u_mxn}\label{conto: h è crescente}\label{conto: h soddisfa BCC} Let $n \geq 2$.
	\begin{enumerate}
		\item\label{bullet: u crescente} If $m \geq 2$, then $u(m,n) - u(m,n-1) \geq \left\lceil \binom{m+3}{3}\frac{n+1}{m+n+1} \right\rceil + n$.

		\item If $m = 1$, then $u(1,n) - u(1,n-1) \geq \left\lceil \frac{4(n+1)}{n+2} \right\rceil$.
	\end{enumerate}
\end{lemma}
\begin{proof} 
For the first statement, we have to prove that
\[	k^*(3,4;m,n)-2k^*(3,4;m,n-1)+k^*(3,4;m,n-2)\geq  \left\lceil \binom{m+3}{3}\frac{n+1}{m+n+1} \right\rceil + n.	\]
	We bound
\begin{align*}&\left\lceil \frac{\binom{m+3}{3}\binom{n+4}{4}}{m+n+1} \right\rceil - 2\left\lceil \frac{\binom{m+3}{3}\binom{n+3}{4}}{m+n} \right\rceil + \left\lceil \frac{\binom{m+3}{3}\binom{n+2}{4}}{m+n-1} \right\rceil - \left\lceil \binom{m+3}{3}\frac{n+1}{m+n+1} \right\rceil - n \\
& \geq \frac{\binom{m+3}{3}\binom{n+4}{4}}{m+n+1} -  \frac{2\binom{m+3}{3}\binom{n+3}{4}}{m+n} - 2 +  \frac{\binom{m+3}{3}\binom{n+2}{4}}{m+n-1} - \binom{m+3}{3}\frac{n+1}{m+n+1} - 1- n \\
& = \frac{A(m,n)}{72(m+n+1)(m+n)(m+n-1)},	\end{align*}
	
	where, as a polynomial in $\bbC[n][m]$,
	\begin{align*}
A(m,n) = &(6n^2+6n)m^5+(8n^3+30n^2+22n)m^4\\
&+(3n^4+42n^3+27n^2-84n-216)m^3+(18n^4+52n^3-264n^2-730n)m^2\\
&+(33n^4-234n^3-717n^2+54n+216)m-54n^4-252n^3+54n^2+252n.
	\end{align*}
	The sequence of coefficients of the univariate polynomial $A(m,n) \in \bbC[n][m]$ has only one change of signs for any $n \geq 2$. By Descartes' rule of signs, we have at most one positive root. Note that
	\begin{align*}
		A(1,n) &= -384n^3-864n^2-480n < 0\mbox{ and}\\ A(2,n) &=108n^4-48n^3-1548n^2-2688n-1296 > 0 \text{ for } n \geq 5.
	\end{align*}
In particular, for $n \geq 5$ the unique positive root of the univariate polynomial $A(m,n)$ belongs to the interval $(1,2)$ for every $n\ge 5$. Since $m \geq 2$, we get $A(m,n) \geq 0$. Now assume that $n\in\{2,3,4\}$. 
Since the leading coefficient of $A(m,n)$ is positive and it has a unique positive root, we have that $A(m,n) \geq 0$ for $m \geq 3$. We are left with the cases $(m,n) \in \{(2,2),(2,3),(2,4)\}$, for which we check the statement directly.
	
    Consider the second statement. For $n\in\{2,3\}$, we check it directly. If $n\ge 4$, then
	\begin{align*}
	& k^*(3,4;1,n)-2k^*(3,4;1,n-1)+k^*(3,4;1,n-2) - \left\lceil \frac{4(n+1)}{n+2} \right\rceil \\
	& \geq \left\lceil \frac{4\binom{n+4}{4}}{n+2} \right\rceil - 2\left\lceil \frac{4\binom{n+3}{4}}{n+1} \right\rceil + \left\lceil \frac{4\binom{n+2}{4}}{n} \right\rceil - \left\lceil \frac{4(n+1)}{n+2} \right\rceil \\
	& =\frac{(n+4)(n+3)(n+1)}{6} - \frac{(n+3)(n+2)n}{3}  - \frac{5}{3} +  \frac{(n+2)(n+1)(n-1)}{6} -\frac{4(n+1)}{n+2} - \frac{n+1}{n+2} \\ 
	& = \frac{n^{2}-3n-5}{n+2} > -1.\qedhere
	\end{align*} 
\end{proof}

\begin{lemma}
	\label{appendix:kdown34vskdown33} If $n\ge 3$ and $(m,n)\neq (1,3)$, then $k_*(3,4;m,n)-k_*(3,4;m,n-1)\le k^*(3,3;m,n)$.
\end{lemma}	
\begin{proof}
	We bound
	\begin{align*}
	& \left\lfloor \frac{\binom{m+3}{3}\binom{n+4}{4}}{m+n+1} \right\rfloor - (m+n+1) - \left\lfloor \frac{\binom{m+3}{3}\binom{n+3}{4}}{m+n} \right\rfloor + (m+n) - \left\lceil \frac{\binom{m+3}{3}\binom{n+3}{3}}{m+n+1} \right\rceil + (m+n+1) \\
	& \leq \frac{\binom{m+3}{3}\binom{n+4}{4}}{m+n+1} - \frac{\binom{m+3}{3}\binom{n+3}{4}}{m+n} + \frac{m+n-1}{m+n} - \frac{\binom{m+3}{3}\binom{n+3}{3}}{m+n+1} + m + n = \frac{A(m,n)}{144(m+n)(m+n+1)},
	\end{align*}
	where $A(m,n)$ as a univariate polynomial in $\bbC[n][m]$ is
	\begin{align*}
		&\left(-n^{4}-6\,n^{3}-11\,n^{2}-6\,n+144\right)m^{3}+\left(-6\,n^{4}-36
		\,n^{3}-66\,n^{2}+396\,n+288\right)m^{2} \\ &+\left(-11\,n^{4}-66\,n^{3}+311\,
		n^{2}+510\,n\right)m-6\,n^{4}+108\,n^{3}+222\,n^{2}-36\,n-144.
	\end{align*}
	For $n \geq 20$ all the coefficients of $A(m,n) \in \bbC[n][m]$ are negative and we deduce $A(m,n) \leq 0$. For $3\leq n \leq 19$ the leading coefficient is negative, therefore there exists $M(n) \geq 1$ such that $A(m,n) \leq 0$ for $m \geq M(n)$. More precisely we have
	\[
	M(n) = \begin{cases}
		3 & \mbox{if }n=3 \\
		2 & \mbox{if }n \in \{4,5\} \\
		1 & \mbox{if }n \in \{6,\ldots,19\}.
	\end{cases}
	\]
	Hence we are left with the cases $(m,n) \in \{(2,3),(1,4),(1,5)\}$ for which the statement can be checked directly.
\end{proof}

\begin{lemma}\label{conto: 3,3 triplodoppiesemplici atteso vuoto}
	If $n\geq 3$ and $m \geq 1$, then $\vdim \calL_{m\times n}^{3,3}(3,2^{k_*(3,4;m,n)-k_*(3,4;m,n-1)})\le k_*(3,4;m,n-1) $.
\end{lemma}
\begin{proof}
	We bound
	\begin{align*}
	& \binom{m+3}{3}\binom{n+3}{3} -\binom{m+n+2}{2}-(m+n+1)\left(\left\lfloor \frac{\binom{m+3}{3}\binom{n+4}{4}}{m+n+1} \right\rfloor - (m+n+1)\right) \\
	& \qquad +(m+n)\left(\left\lfloor \frac{\binom{m+3}{3}\binom{n+3}{4}}{m+n} \right\rfloor - (m+n)\right)\\
	& \leq \binom{m+3}{3}\binom{n+3}{3}-\binom{m+n+2}{2}-\binom{m+3}{3}\binom{n+4}{4}+(m+n)+(m+n+1)^2 \\ 
	& \qquad + \binom{m+3}{3}\binom{n+3}{4} - (m+n)^2 =-\frac{1}{2}\,m^{2}+\left(-n+\frac{3}{2}\right)m-\frac{1}{2}\,n^{2}+\frac{3}{2}\,n\le 0.\qedhere
	\end{align*}
\end{proof}

\begin{lemma}
	\label{conto: 3, 2 con punti doppi atteso vuoto} If $m$ and $n$ are positive integers, then $$\vdim\LL_{1\times n}^{3,2}(2^{h(1,n)})\le 0\mbox{ and }\vdim\LL_{m\times 1}^{3,2}(2^{h(m,1)})\le 0.$$
\end{lemma}
\begin{proof}
	We bound
	\begin{align*} 
		&4\binom{n+2}{2}-(n+2)\left(\rd{\frac{4}{n+2}\binom{n+4}{4}}-\rd{\frac{4}{n+1}\binom{n+3}{4}}\right)\\
		&=(n+2)\left(2n+2-\rd{\frac{(n+4)(n+3)(n+1)}{6}}+\rd{\frac{(n+3)(n+2)n}{6}}\right)\\
		&\le (n+2)\left(2n+2-\frac{(n+4)(n+3)(n+1)-5}{6}+\frac{(n+3)(n+2)n}{6}\right)\\
		&= \frac{(n+2)(5-n-3n^2)}{6}<1
	\end{align*}
	and
	\begin{align*} 
		&3\binom{m+3}{3}-(m+2)\left(\rd{\frac{5}{m+2}\binom{m+3}{3}}-\rd{\frac{1}{m+1}\binom{m+3}{3}}\right)\\
		&=3\binom{m+3}{3}-(m+2)\left(\rd{\frac{5(m+3)(m+1)}{6}}-\rd{\frac{(m+3)(m+2)}{6}}\right)\\
		&\le 3\binom{m+3}{3}-(m+2)\left({\frac{5(m+3)(m+1)-5}{6}}-\frac{(m+3)(m+2)}{6}\right)\\
		&= \frac{(m+2)(5-3m-m^2)}{6}<1.\qedhere
	\end{align*}
\end{proof}

\begin{lemma}
	\label{conto: m x 2, 3,4 con quartplo atteso vuoto} If $m$ is a positive integer, then $k_*(3,4;m,2)=\frac{5m^2+13m+4}{2}$ and \[\vdim\calL_{m\times 2}^{2,4}(3,2^{k_*(3,4;m,2)-k_*(3,4;m-1,2)},1^{k_*(3,4;m-1,2)})\le 0.\]
	\begin{proof}
	    The equality $k_*(3,4;m,2)=\frac{5m^2+13m+4}{2}$ is a direct check.
		We bound
		\begin{align*}
			&\vdim\calL_{m\times 2}^{2,4}(3,2^{k_*(3,4;m,2)-k_*(3,4;m-1,2)},1^{k_*(3,4;m-1,2)})\\&=15\binom{m+2}{2}-\binom{m+4}{2}-(m+3)\frac{5m^2+13m+4}{2}+(m+2)\frac{5(m-1)^2+13(m-1)+4}{2}\\
			&=\frac{-m^2-3m-2}{2}\le 0.\qedhere
		\end{align*}
	\end{proof}
\end{lemma}

\begin{lemma}\label{conto: vdim 1,4 m x 2 positiva}\label{conto: differenza dei kdown permette bcc}
	\label{conto:differenza dei j fa 5} If $m$ is a positive integer, then $j(m)=5m+4$. If $m\ge 2$, then \begin{enumerate}
		\item \label{item: cinque}$j(m)-j(m-1)=5$,
		\item \label{item: vdim con j}$\vdim\LL_{m\times 2}^{1,4}(2,2^{5},1^{j(m-1)})\ge 0$,
		\item \label{item: j permette bcc}$1+j(m)\ge\ru{\frac{15(m+1)}{m+3}}+m$.
	\end{enumerate}
	\begin{proof}
		Since $k_*(3,4;m,2)=\frac{5m^2+13m+4}{2}$,
		\[
		    j(m)=k_*(3,4;m,2)-k_*(3,4;m-1,2)=5m+4.
		\]
		Part \eqref{item: cinque} and part \eqref{item: vdim con j} follow directly.
		Finally,
		\begin{align*}
1+&j(m)-\ru{\frac{15(m+1)}{m+3}}-m=1+ 5m+4-\ru{\frac{15(m+1)}{m+3}}-m\\
&\ge 4m+5-\frac{15(m+1)}{m+3}-1=
\frac{4m^2+m-3}{m+3}\ge 0.
			\qedhere
		\end{align*}
	\end{proof}
\end{lemma}

\begin{lemma}
	\label{conto: 3,4, mxn doppisemplici atteso non vuoto} 
	\label{conto: 4,3 1 xn, con doppi atteso non vuoto}
	If $m$ and $n$ are positive integers, then \[\vdim\calL_{m\times n}^{3,4}(2,2^{k^*(4,4;m,n)-k^*(4,4;m-1,n)},1^{k^*(4,4;m-1,n)})\ge 0.\]
	In particular, $\vdim\LL_{1\times n}^{4,3}(2,2^{k^*(4,4;1,n)-k^*(4,4;1,n-1)},1^{k^*(4,4;1,n-1)})\ge 0$.
	\begin{proof}
		We bound
		\begin{align*}
&\binom{m+3}{3}\binom{n+4}{4}-(m+n+1)(1+k^*(4,4;m,n))+(m+n)k^*(4,4;m-1,n)\\
&\resizebox{\textwidth}{!}{$\ge\binom{m+3}{3}\binom{n+4}{4}-(m+n+1)\left(\frac{\binom{m+4}{4}\binom{n+4}{4}+m+n}{m+n+1}-m-n\right)+(m+n)\left(\frac{\binom{m+3}{4}\binom{n+4}{4}}{m+n}-m-n\right)$}
\\
&=0.\qedhere
		\end{align*}
	\end{proof}
\end{lemma}

\begin{lemma}
    \label{conto: orfanello}
If $m$ and $n$ are positive integers, then 
\[
    \vdim\LL_{m\times n}^{3,4}(3,2^{k_*(4,4;m,n)-k_*(4,4;m-1,n)},1^{k_*(4,4;m-1,n)})\le 0.
\]
	\begin{proof}
		We bound
		\begin{align*}
&\binom{m+3}{3}\binom{n+4}{4}-(m+n+1)(1+k^*(4,4;m,n))+(m+n)k^*(4,4;m-1,n)\\
&\le\binom{m+3}{3}\binom{n+4}{4}-\binom{m+n+2}{2}-(m+n+1)\left(\frac{\binom{m+4}{4}\binom{n+4}{4}}{m+n+1}-m-n-1\right)\\
&\hskip50pt+(m+n)\left(\frac{\binom{m+3}{4}\binom{n+4}{4}+m+n-1}{m+n}-m-n\right)\\
&=\binom{m+3}{3}\binom{n+4}{4}-\binom{m+n+2}{2}-\binom{m+4}{4}\binom{n+4}4+(m+n+1)^2\\
&=-\frac12(n+m-1)(n+m-2)\leq 0.\qedhere
		\end{align*}
	\end{proof}
\end{lemma}

\begin{lemma}
	\label{conto: evitare problemi w bbc}
	If $m\ge 2$, then $\vdim\LL_{m\times 1}^{2,4}(2^{w(m,1)})\le 0$ and $w(m,1)> 3m+2$.
	\begin{proof}
	The most restrictive inequality is the second one.
	Hence, we bound
		\begin{align*}
		    \left\lceil \frac{5\binom{m+4}{m}}{m+2}\right\rceil & - (m+2) -  \left\lceil \frac{5\binom{m+3}{m-1}}{m+1}\right\rceil + (m+1) - 3m - 2 \\ 
		    & \geq 5\frac{\binom{m+4}{m}}{m+2} - (m+2) - 5\frac{\binom{m+3}{m-1}}{m+1} - 1 + (m+1) - 3m - 2 = \frac{15m^2-7m-36}{24}\ge 0.\qedhere
		\end{align*}
	\end{proof}
\end{lemma}

\begin{lemma}
	\label{conto: vdim con w sottozero}
	If $n$ is a positive integer, then $\vdim\LL_{1\times n}^{2,4}(2^{w(1,n)})
	\le 0$.
\begin{proof}
We directly check that $\vdim\LL_{1\times 1}^{2,4}(2^{w(1,1)})=0$. We assume that $n\ge 2$ and we bound
\begin{align*}
& \vdim\LL_{1\times n}^{2,4}(2^{w(1,n)})
\le 3\binom{n+4}{4}-(n+2)\left(\frac{5}{n+2}\binom{n+4}{4}-2-\frac{1}{n+1}\binom{n+4}{4}\right)\\
& =(n+2)\left(2-\frac{n(n+4)(n+3)}{24}\right)\le 0.\qedhere
\end{align*}
\end{proof}
\end{lemma}

\begin{lemma}
	\label{conto: w crescente e ammette bcc}
	If $m,n\ge 2$, then $w(m,n)-w(m-1,n)\ge \ru{\frac{m+1}{m+n+1}\binom{n+4}{4}}+m$.
\end{lemma}
\begin{proof}
    We consider
    \begin{align*}
        k^*&(4,4,m,n)- 2k^*(4,4,m-1,n) + k^*(4,4,m-2,n) - \left\lceil \frac{m+1}{m+n+1}\binom{n+4}{4}\right\rceil-m \\
        & = \left\lceil \frac{\binom{m+4}{4}\binom{n+4}{4}}{m+n+1} \right\rceil -2 \left\lceil \frac{\binom{m+3}{4}\binom{n+4}{4}}{m+n} \right\rceil + \left\lceil \frac{\binom{m+2}{4}\binom{n+4}{4}}{m+n-1} \right\rceil - \left\lceil \frac{m+1}{m+n+1}\binom{n+4}{4}\right\rceil-m \\
        & \geq \frac{\binom{m+4}{4}\binom{n+4}{4}}{m+n+1} -2  \frac{\binom{m+3}{4}\binom{n+4}{4}}{m+n} - 2 + \frac{\binom{m+2}{4}\binom{n+4}{4}}{m+n-1} -  \frac{m+1}{m+n+1}\binom{n+4}{4}-1-m \\
        & = \frac{A(m,n)}{288(m+n+1)(m+n)(m+n-1)}
    \end{align*}
    where
    \begin{align*}
        A(m,n) = & \left(6\,m^{2}+6\,m\right)n^{6}+\left(8\,m^{3}+54\,m^{2}+46\,m\right)n^{
      5}+\left(3\,m^{4}+74\,m^{3}+147\,m^{2}+76\,m\right)n^{4}+\\
      &\left(30\,m^{4}+
      220\,m^{3}+60\,m^{2}-418\,m-864\right)n^{3}+\left(105\,m^{4}+190\,m^{3}-
      1125\,m^{2}-2938\,m\right)n^{2}\\
      &+\left(150\,m^{4}-972\,m^{3}-2886\,m^{2}+ 252\,m+864\right)n-216\,m^{4}-1008\,m^{3}+216\,m^{2}+1008\,m.
    \end{align*}
    Since $n \geq 2$ and the first four coefficients of $A(m,n)$ as univariate polynomial in $\bbC[m][n]$ are positive, we bound 
    \begin{align*}
        & \left(6\,m^{2}+6\,m\right)n^{6} + \left(8\,m^{3}+54\,m^{2}+46\,m\right)n^{
      5} + \left(3\,m^{4}+74\,m^{3}+147\,m^{2}+76\,m\right)n^{4} \\
      &\quad \geq 16\left(6\,m^{2}+6\,m\right)n^2 + 8\left(8\,m^{3}+54\,m^{2}+46\,m\right)n^2 + 4\left(3\,m^{4}+74\,m^{3}+147\,m^{2}+76\,m\right)n^2 
    \end{align*}
    and $\left(30\,m^{4}+
      220\,m^{3}+60\,m^{2}-418\,m-864\right)n^{3} \geq 8\left(30\,m^{4}+
      220\,m^{3}+60\,m^{2}-418\,m-864\right)$. Hence
    \begin{align*}
        A(m,n) \geq & \left(117\,m^{4}+550\,m^{3}-9\,m^{2}-2170\,m\right)n^{2}+\\
        &\left(150\,m^{
       4}-972\,m^{3}-2886\,m^{2}+252\,m+864\right)n\\
       &+24\,m^{4}+752\,m^{3}+696\,m
       ^{2}-2336\,m-6912
    \end{align*}
    The coefficients of the latter polynomial are all positive for $m \geq 9$ from which we deduce that $A(m,n) \geq 0$. For $2 \leq m \leq 8$, the leading term is positive hence there is $N(m) \geq 0$ such that $A(m,n) \geq 0$ for $n \geq N(m)$ with
    \[
  N(m) = \begin{cases}
   9 & \text{ if } m = 2 \\
                2 & \text{ if } m \in \{3,\ldots,8\}.
            \end{cases}
    \]
    Hence we are left with the cases $m = 2$ and $n \in \{2,\ldots,7\}$ for which the statement is checked directly.
\end{proof}

\begin{lemma}
	\label{conto: kappa alto e kappa basso e differenze}
	If $n\ge m\ge 2$ and $(m,n)\neq (2,2)$, then \[
	k_*(4,4;m,n)-k_*(4,4;m-1,n)\le k^*(3,4;m,n)
	.\]
	\begin{proof}
		We bound
		\begin{align*}
			&k_*(4,4;m,n)-k_*(4,4;m-1,n)- k^*(3,4;m,n)\\
			&=\rd{\frac{\binom{m+4}{4}\binom{n+4}{4}}{m+n+1}}-\rd{\frac{\binom{m+3}{4}\binom{n+4}{4}}{m+n}}-\ru{\frac{\binom{m+3}{3}\binom{n+4}{4}}{m+n+1}}+m+n\\
			&\le\frac{\binom{m+4}{4}\binom{n+4}{4}}{m+n+1}-\frac{\binom{m+3}{4}\binom{n+4}{4}-(m+n-1)}{m+n}-\frac{\binom{m+3}{3}\binom{n+4}{4}}{m+n+1}+m+n=\frac{A(m,n)}{576(m+n+1)(m+n)},
		\end{align*}
		where
\begin{align*}
A&(m,n)=\left(-m^{4}-6\,m^{3}-11\,m^{2}-6\,m\right)n^{4}+\left(-10\,m^{4}-60\,m^{3}-110\,m^{2}-60\,m+576\right)n^{3}\\
&+\left(-35\,m^{4}-210\,m^{3}-385\,m^{2}+1518\,m+1152\right)n^{2}+\left(-50\,m^{4}-300\,m^{3}+1178\,m^{2}+2004\,m\right)n\\
&-24\,m^{4}+432\,m^{3}+888\,m^{2}-144\,m-576.
\end{align*}
Since $n \geq m$ and the leading coefficient is always negative, we can write
\begin{align*}
    A(m,n) &\leq \left(-10\,m^{4}-60\,m^{3}-110\,m^{2}-60\,m+576\right)n^{3}\\
    &+\left(-35\,
       m^{4}-210\,m^{3}-385\,m^{2}+1518\,m+1152\right)n^{2}\\
       &+\left(-50\,m^{4}-
       300\,m^{3}+1178\,m^{2}+2004\,m\right)n\\
       &-m^{8}-6\,m^{7}-11\,m^{6}-6\,m^{5
       }-24\,m^{4}+432\,m^{3}+888\,m^{2}-144\,m-576.
\end{align*}
The coefficients of the latter univariate polynomial in $\bbC[m][n]$ are negative for $m \geq 4$. For $m \in \{2,3\}$, since the leading coefficient is negative, there exists $N(m)$ such that the polynomial is negative for $n \geq N(m)$ where
\[
    N(m) = \begin{cases}
            4 & \text{ if } m = 2\\
            2 & \text{ if } m =3.
    \end{cases}
\]
Hence we are left with the cases $m=2$ and $n \in \{2,3\}$ for which the statement is checked directly.
	\end{proof}
\end{lemma}

\begin{lemma}
	\label{conto: ultimo conto}
	If $m$ and $n$ are positive integers, then \begin{enumerate}
		\item\label{item: quello con m e n} $1+k_*(4,4;m,n)-k_*(4,4;m-1,n)\ge k^*(4,4;m,n)-k^*(4,4;m-1,n)$ and
		\item\label{item: puoi scambiare, tanto k44 è simmetrico} $1+k_*(4,4;1,n)-k_*(4,4;1,n-1)\ge k^*(4,4;1,n)-k^*(4,4;1,n-1)$.
	\end{enumerate}
	\begin{proof}By the definition,  it is apparent that $$1+k_*(4,4;m,n)\ge k^*(4,4;m,n)\mbox{ and }k^*(4,4;m-1,n)\ge k_*(4,4;m-1,n) ,$$
		so the first statement holds. If we write down part \eqref{item: quello con m e n} for $n=1$ we obtain
		$$1+k_*(4,4;m,1)-k_*(4,4;m-1,1)\ge k^*(4,4;m,1)-k^*(4,4;m-1,1).$$
		Note that $k_*(4,4;m,n)=k_*(4,4;n,m)$, so we can swap the last two arguments in each instance to get
		$$1+k_*(4,4;1,m)-k_*(4,4;1,m-1)\ge k^*(4,4;1,m)-k^*(4,4;1,m-1),$$
		hence part \eqref{item: puoi scambiare, tanto k44 è simmetrico} follows.
	\end{proof}
\end{lemma}

\begin{lemma}
	\label{conto: quello di prima era il penultimo}
	If $n\ge 4$, then $k_*(4,4;1,n)-k_*(4,4;1,n-1)\le k^*(4,3;1,n)$.
	\begin{proof}
	    This is a special case of Lemma \ref{conto: ultimo conto} by inverting the role of $m$ and $n$ and replacing $n = 1$. 
	\end{proof}
\end{lemma}

\begin{lemma}
	\label{conto: o il terzultimo}
	If $n\ge 2$, then $\vdim\LL_{1\times n}^{4,3}(3,2^{k_*(4,4;1,n)-k_*(4,4;1,n-1)})\le k_*(4,4;1,n-1)$.
	\begin{proof}
	    	We bound
		\begin{align*}
			&\vdim\LL_{1\times n}^{4,3}(3,2^{k_*(4,4;1,n)-k_*(4,4;1,n-1)})- k_*(4,4;1,n-1)\\
			&\le 5\binom{n+3}{3}-\binom{n+3}{2}-(n+2)\left(\frac{5\binom{n+4}{4}-n-1}{n+2}-n-2\right)+(n+1)\left(\frac{5\binom{n+3}{4}}{n+1}-n-1\right)\\
			&= -\binom{n+3}{2}+n+1+(n+2)^2-(n+1)^2=\frac{-n^2 + n + 2}{2}\le 0.\qedhere
		\end{align*}
	\end{proof}
\end{lemma}
\end{appendix}

\Fra{
\bibliographystyle{alpha}
\bibliography{defSV_collapsing_sub.bib}
}

\end{document}